\documentclass[12pt,a4paper]{article}

\pdfoutput=1

\usepackage[utf8]{inputenc}
\usepackage[T1]{fontenc}
\usepackage{listings}
\usepackage{mathtools}
\usepackage{algorithm}
\usepackage{algpseudocode}
\usepackage{amsthm}
\usepackage{amssymb}
\usepackage{amsfonts}
\usepackage[title]{appendix}%
\usepackage{authblk} 
\usepackage{array}
\usepackage[english]{babel}
\usepackage{bbm}  
\usepackage{empheq}
\usepackage[shortlabels]{enumitem} 
\usepackage{graphicx}
\usepackage[hypertexnames=false]{hyperref}	
\usepackage{cleveref} 
\usepackage[top=3cm, bottom=3cm, left=2.5cm, right=2.5cm]{geometry}
\usepackage{fancyhdr}
\usepackage{lastpage}
\usepackage{lmodern}
\usepackage{tikz,tkz-tab}
\usetikzlibrary{babel}
\usepackage{subfiles}
\usepackage{stmaryrd}
\usepackage{subcaption} 
\usepackage{tabularx}
\usepackage{pgf}
\usepackage{pgfplots} 
\usepackage{wrapfig}

\DeclareUnicodeCharacter{2212}{-}
\usepgfplotslibrary{groupplots,dateplot,fillbetween}
\usetikzlibrary{patterns,shapes.arrows,intersections,backgrounds,external}

\newcommand{\customTikzOutputFolder}{tikz/output}
\pgfplotsset{compat=newest}

\usepackage{float} 

\definecolor{dkgreen}{rgb}{0,0.5,0}
\definecolor{gray}{rgb}{0.5,0.5,0.5}
\definecolor{mauve}{rgb}{0.58,0,0.82}

\setenumerate[1]{label=\thesection.\arabic*.} 

\hypersetup{
colorlinks=true, 
linkcolor=red	
}

\fancypagestyle{plain}{
	
	\fancyhead[]{}	
	\fancyfoot[L]{}	
	\fancyfoot[R]{}
	\fancyfoot[C]{\thepage{} / \pageref{LastPage}} 
}
\crefname{thm}{theorem}{theorems}
\Crefname{thm}{Theorem}{Theorems}

\crefname{demo}{proof}{proofs}
\Crefname{demo}{Proof}{Proofs}

\crefname{prop}{proposition}{propositions}
\Crefname{prop}{Proposition}{Propositions}

\crefname{lemma}{lemma}{lemmata}
\Crefname{lemma}{Lemma}{Lemmaata}

\crefname{definition}{definition}{definitions}
\Crefname{definition}{Definition}{Definitions}

\crefname{rmk}{remark}{remarks}
\Crefname{rmk}{Remark}{Remarks}

\theoremstyle{plain}

\theoremstyle{plain}

\theoremstyle{plain}

\theoremstyle{plain}
\newtheorem{lemma}{Lemma}[section]

\theoremstyle{definition}

\theoremstyle{remark}
\newtheorem{remark}{Remark}[section]

\makeatletter 
\@ifclassloaded{report}{%
\renewcommand{\thesection}{\thechapter.\arabic{section}} 
}{}
\makeatother

\numberwithin{equation}{section} 

\newcommand{\textd}{\textnormal{d}}

\DeclareMathOperator{\arctanTwo}{arctan2}

\allowdisplaybreaks 

\newif\ifcustomcompileusingtikzexternalize
\customcompileusingtikzexternalizefalse 
\ifcustomcompileusingtikzexternalize
    
    \renewcommand{\customTikzOutputFolder}{./tikz-messy/output/} 
\else
    
    \renewcommand{\customTikzOutputFolder}{./tikz/} 
\fi

\tikzsetexternalprefix{\customTikzOutputFolder}

\author[1]{Erwan Faou\footnote{\url{erwan.faou@inria.fr}}}
\author[1]{Yoann Le H{\'e}naff\footnote{\url{yoann.lehenaff@univ-rennes.fr}}}
\author[2]{Pierre Rapha{\"e}l\footnote{\url{pr463@dpmms.cam.ac.uk}}}

\affil[1]{\small Rennes University, IRMAR UMR 6625\\INRIA centre at Rennes University (MINGuS project-team), Rennes, France}
\affil[2]{\small Department of Pure Mathematics and Mathematical Statistics, University of Cambridge, Cambridge, U.K.}

\date{}

\begin{document}

\title{Modulation algorithm for the nonlinear Schrödinger equation}

\maketitle

\begin{abstract}
    Based on recent ideas, stemming from the use of \emph{bubbles}, we discuss an algorithm for the numerical simulation of the cubic nonlinear Schr\"odinger equation with harmonic potential in any dimension, which could be easily extended to other polynomial nonlinearities.
    For the linear part of the equation, the algorithm consists in discretizing the initial function as a sum of modulated complex functions, each one having its own set of parameters, and then updating the parameters exactly so that the modulated function remains a solution to the equation.
    When cubic interactions are introduced, the Dirac-Frenkel-MacLachlan principle is used to approximate the time evolution of parameters.
    We then obtain a grid-free algorithm in any dimension, and it is compared to a spectral method on numerical examples.
\end{abstract}

\tableofcontents


\section{Introduction}

We are interested in approximating numerically the solution $\psi(t,x)$ to the cubic nonlinear Schr{\"o}dinger equation with harmonic potential,
\begin{equation}
    i \partial_{t} \psi + \Delta_x  \psi - | x |^2 \psi = \psi |\psi|^2, \quad x\in \mathbb{R}^d
    \tag{cNLS}
    \label{eqn: cNLS with psi}
    ,
\end{equation}
where \( d \geq 1 \), \( | \cdot | \) denotes the usual Euclidian norm over \( \mathbb{R}^d \), and \( \Delta_x \) denotes the Laplace operator over \( \mathbb{R}^d \): \( \Delta_x = \sum_{i=1}^d \partial^2_{x_i} \).
This equation is also sometimes called \emph{time-dependent Gross-Pitaevskii equation} \cite{cerimeleNumericalSolutionGrossPitaevskii2000,baoNumericalSolutionGross2003,wangNumericalSimulationGrossPitaevskii2017,tsedneeNumericalSolutionTimedependent2022}.
We focus on a cubic nonlinearity for the sake of clarity, but we emphasize the fact that everything we present is also applicable to other types of polynomial nonlinearities, \emph{mutatis mudandis}.
Similarly, the extension to the equation \eqref{NLSoptic} -- which is \eqref{eqn: cNLS with psi} without the harmonic potential -- is also straightforward.

Very recent works \cite{martelStronglyInteractingBlow2018,faouWeaklyTurbulentSolutions2020} suggest to discretize the solution \( \psi \) of \eqref{eqn: cNLS with psi} as a sum of \( N \) modulated functions, which write as:
\begin{equation}
    \label{eqn: generic discretization of psi -- ansatz}
    \psi(t,x) \approx u(t,x) := \sum_{j=1}^N u_j(t,x)
    ,
\end{equation}
where
\begin{equation}
    \label{eqn: generic discretization of psi -- expression for uj}
    u_j(t,x) := \frac{A_j}{L_j} e^{i\gamma_j + i L_j \beta_j \cdot y_j - i \frac{B_j}{4} |y_j|^2} v_j(s_j, y_j), \quad \text{with } \quad
    \left|
    \begin{aligned}
         & \frac{\textnormal{d} s_j}{\textnormal{d} t} := \frac{1}{L_j^2}, \\
         & y_j := \frac{x - X_j}{L_j},
    \end{aligned}
    \right.
\end{equation}
and \( N \in \mathbb{N}^* \).
In the cited works, the modulated functions \( u_j \) are called \emph{bubbles}. Throughout this work, we may refer to the variables \( (s_j, y_j) \) as the \emph{modulation frame} of the bubble labelled \( j \).

The time dependence of the parameters \( A_j, L_j, B_j, X_j, \beta_j, \gamma_j \) has not been written in \eqref{eqn: generic discretization of psi -- expression for uj} for the sake of clarity, but it is one of the main ingredients of the approach.
More precisely, the core idea is to plug the ansatz \eqref{eqn: generic discretization of psi -- ansatz} into \eqref{eqn: cNLS with psi} in order to obtain ODEs for the parameters.

The idea of relying on time-dependent parameters to represent the solution, or an approximation, is not new and has been widely studied in the linear case, {\em i.e.} when the cubic nonlinearity is replaced by some multiplication with a potential.
When the $v_j$ are chosen as Gaussian functions, it has been called \emph{Variational Gaussian wave packets} and extensively analyzed by Lasser and Lubich \cite{lasserComputingQuantumDynamics2020}, where they applied the Dirac-Frenkel-MacLachlan principe (DFMP) to the linear Schr\"odinger equation with potential.

More generally, this type of method using Gaussian functions is widely used in the field of Chemical Physics \cite{hellerTimeDependentVariational1976,huberGeneralizedGaussianWave1987,coalsonMultidimensionalVariationalGaussian1990,worthNovelAlgorithmNonadiabatic2004,adamowiczLaserinducedDynamicAlignment2022}.
The different methods used are variations of the same idea, and possess many names: superposition of Gaussian Wavepackets, Gaussian beams, Thawed Gaussians, Frozen Gaussian\dots All of these algorithms simply consist in applying a Dirac-Frenkel-MacLachlan principle to linear Schr\"odinger equations, the difference lying in how the parameters are updated.
For example, the Thawed Gaussian method allows the width matrix to be time-dependent while the Frozen Gaussian does not.

\vspace{1em}

Let us now explain the main ideas underlying the full modulation \eqref{eqn: generic discretization of psi -- expression for uj} -- developed in various works, see for instance \cite{martelStronglyInteractingBlow2018,faouWeaklyTurbulentSolutions2020} and the references therein -- and why it is particularly adapted to the nonlinear case.

Consider for instance the case of one bubble, {\em i.e.} $N = 1$.
When plugging the ansatz \eqref{eqn: generic discretization of psi -- expression for uj} into \eqref{eqn: cNLS with psi}, we obtain an equation of the form
\begin{equation*}
    i \partial_{s} v + \Delta_y v - | y |^2 v - |v|^2 v + P(s; y,\partial_y) v = 0,
\end{equation*}
where $P(s; y,\partial_y)$ is a quadratic operator in $y$ and $\partial_y$, which depends on time \( s \) through the parameters $(A,L,B,X,\beta,\gamma)$ and their time derivatives with respect to $s$. See \eqref{eqn: HO -- idt u - Hu -- with v} for more precise detail.
It is then possible to choose the parameters in such a way that for instance $P(s; y,\partial_y) v = -\lambda v$ for some $\lambda \in \mathbb{R}$, and to take $v$ as a soliton solution of the stationary equation
\begin{equation}
    \label{soliton}
    - \Delta_y  v + | y |^2 v + |v|^2 v = \lambda v.
\end{equation}
This yields a differential system to be solved by the parameters $(A,L,B,X,\beta,\gamma)$ which is given below by \eqref{eqn: modulation ODEs -- linear part wrt time s - with v}.
It turns out that these equations form a {\em completely integrable Poisson system} that can be solved, and the solution for a single bubble can be thus taken as a {\em modulated soliton}.

In other words, taking $v_j = v$ when $N = 1$, a solution of the nonlinear equation \eqref{soliton} yields an exact solution $u_j = u$ under the form \eqref{eqn: generic discretization of psi -- expression for uj} of \eqref{eqn: cNLS with psi}.

This kind of approach has been used successfully in various situations from a theoretical point of view, see \cite{merleBlowupDynamicUpper2005,martelStronglyInteractingBlow2018,faouWeaklyTurbulentSolutions2020,merleIMPLOSIONTHREEDIMENSIONAL2020} and the references therein.
Typically, when \( N \geq 2 \), several modulated solitons interact and this can produce finite time blow-up of growth of Sobolev norm phenomena. A large part of the analysis relies on the ability of calculating nonlinear interactions between two modulated solitons.
This can be done for instance in an integrable situation, e.g. the Szeg\"o equation \cite{gerardTwoSolitonTransientTurbulent2018}.

Another byproduct of these modulation techniques in 2D is to make a link between \eqref{eqn: cNLS with psi} on a finite time interval and the Schr\"odinger equation without harmonic potential
\begin{equation}
    \label{NLSoptic}
    i \partial_{s} \psi + \Delta_x  \psi  = \psi |\psi|^2, \quad x\in \mathbb{R}^d
\end{equation}
on an unbounded time interval.
In this case, the modulation equations generate the so-called {\em lens} transform, see for instance \cite{carlesSemiclassicalAnalysisNonlinear2021}.
Note that our algorithms could be also be applied to the latter equation but we will restrict our analysis to the Harmonic case.
Let us note as well that such modulation techniques can also be related with the families of exact splitting introduced in \cite{bernierExactSplittingMethods2020}, where the time coefficients can be seen as specific time changes $s$ in the modulation equations.

Inspired by these successful theoretical works, we retain the idea of approximating solutions to \eqref{eqn: cNLS with psi} by modulating the parameters \( A_j \), \( L_j \), \( B_j \), \( X_j \), \( \beta_j \), \( \gamma_j \) in such a way that $v_j(s_j,x)$ satisfies a {\em smoother in time} equation -- typically a stationary soliton equation.
However, from the numerical point of view, choosing the $v_j$ as stationary solitons would require first to solve explicitly the nonlinear equation \eqref{soliton} and more problematically, to estimate numerically the nonlinear interactions between the modulated solitons by using the Dirac-Frenkel-MacLachlan principle.
The latter consists essentially in a projection onto the manifold of modulated solitons, which is in practice very difficult to evaluate numerically.
Moreover, one is naturally interested in using a splitting strategy between the linear and nonlinear parts, which would typically destroy the soliton structure in the equation.
Following this idea, we split the Schr{\"o}dinger equation \eqref{eqn: cNLS with psi} into the linear part
\begin{equation}
    \label{eqn: cNLS with psi -- linear part}
    i \partial_{t} \psi + \Delta_x \psi - |x|^2 \psi = 0
    \tag{HO}
    ,
\end{equation}
and the nonlinear part
\begin{equation}
    \label{eqn: cNLS with psi -- nonlinear part}
    i \partial_{t} \psi = \psi |\psi|^2.
    \tag{cNLS-nonLin.}
\end{equation}
The linear equation \eqref{eqn: cNLS with psi -- linear part} is also called the Harmonic Oscillator.
Traditional well-known numerical schemes are based on this abstract decomposition and it is easy to determine high-order splitting methods obtained by solving alternately the linear and nonlinear parts, like Lie, Strang Splitting or triple jump composition, see for instance \cite{mclachlanSplittingMethods2002,hairerGeometricNumericalIntegration2006,casasHighorderHamiltonianSplitting2017}.
However, the approximation of the solution to each of these two parts remains to be done using time and space discretizations.
They are traditionally solved using grid-based numerical schemes (see for instance \cite{baoNumericalSolutionGross2003,qianFastGaussianWavepacket2010,faouGeometricNumericalIntegration2012,wangNumericalSimulationGrossPitaevskii2017,bernierExactSplittingMethods2020}). The computational complexity of grid-based methods is always an issue due to the bad scaling with respect to the dimension.
Fortunately, using the modulation techniques given above, the solution to the linear part \eqref{eqn: cNLS with psi -- linear part} can be simulated exactly, in a straightforward manner, and very efficiently by considering Hermite decomposition of the functions $v_j$. 
The computational cost for the simulation of the linear part only is \( \mathcal{O}(N\cdot d) \) -- recall \( N \) is the number of bubbles and \( d \) the dimension -- to be compared with grid-based complexities of order \( \mathcal{O}(M^d) \) where \( M \) would be the number of discretization points in each dimension.

To approximate the solution to the nonlinear part \eqref{eqn: cNLS with psi -- nonlinear part}, we use the Dirac-Frenkel-MacLachlan principle. In theory, when the $v_j$ are finite sums of Hermite polynomials, the calculation of the interactions boils down to the computation of integrals of products of Hermite functions in different modulation frames, which {\em a priori} can be done in a systematic way.
In practice, these computations can get heavy and to simplify them we will give the explicit result of the Gaussian case in this paper.

In the end, we thus obtain an algorithm for modulated Gauss-Hermite functions, which can be easily implemented numerically, is grid-free, and is also able to capture high oscillations of the solution.

\vspace{1em}
The paper is organized as follows.
Section \ref{sect: The Harmonic Oscillator} is devoted to \eqref{eqn: cNLS with psi -- linear part} for a general bubble decomposition.
We recall some conservation laws obtained when considering \emph{bubbles} in \eqref{eqn: cNLS with psi}, and exhibit \emph{universal modulation equations} which have a completely integrable Hamiltonian structure (see for instance \cite{leimkuhlerSimulatingHamiltonianDynamics2005,hairerGeometricNumericalIntegration2006} for more details about completely integrable Hamiltonian structures).
Analytical formula for all of the modulation parameters can then be made explicit by decomposing the \( v_j \) into the Hermite basis.

We focus in Section \ref{sect: DFMP} on \eqref{eqn: cNLS with psi -- nonlinear part}, which takes into account cubic interactions.
The nonlinearity that is introduced is the core of difficulties arising in the Schr{\"o}dinger equation, and it is hopeless to look for exact solutions in the general case.
The proposed approach uses the Dirac-Frenkel-MacLachlan principle to obtain modulation equations in the case of Gaussian functions \( v_j \).
The choice of Gaussian functions allows to perform most computations exactly, and to avoid numerical quadratures or delicate calculations of integrals of multiple products of Hermite functions.
This allows us to take into account as many bubbles as one desires at the cost of a computational complexity of order \( \mathcal{O}(N^4 d + d^3 N^3) \).
Here, \( N \) is the total number of bubbles and \( d \) the dimension.
The fourth power of \( N \) is due to polynomial interactions of order three.
This algorithm almost does not suffer from the well-known ``curse of dimensionality'' since it is at most polynomial with respect to the dimension \( d \).

Finally, Section \ref{sect: numerical experiments} is dedicated to illustrating numerically the fine details that are obtainable with the Dirac-Frenkel-MacLachlan principle, as well as the long-time behavior, compared to a FFT-based spectral scheme.
Our experiments show that, if the initial data is discretized ``nicely'', the given algorithm yields satisfying results.


\section{The Harmonic Oscillator}
\label{sect: The Harmonic Oscillator}

In this section we focus onto the linear part of the cubic Non Linear Schr\"odinger equation, namely the Harmonic Oscillator \eqref{eqn: cNLS with psi -- linear part}.

\subsection{Conservation Laws}
\label{sect: HO -- Conservation Laws}

We recall classical laws for the Harmonic oscillator equation (see for instance \cite{taoNonlinearDispersiveEquations2006,killipNonlinearSchrOdinger2013}).
Let \( \psi \) be the solution to \eqref{eqn: cNLS with psi -- linear part}.

\begin{lemma}[Conserved quantities in dimension \( d=2 \)]
    \label{lemma: conserved quantities in HO}
    We consider a two-parameter family of equations containing \eqref{eqn: cNLS with psi -- linear part} and \eqref{eqn: cNLS with psi}:
    \begin{equation*}
        i \partial_{t} \psi + \mu(\Delta \psi - |x|^2 \psi) = \lambda |\psi|^2 \psi, \quad \mu, \lambda \in \mathbb{R}
        .
    \end{equation*}
    The (radial) conservation laws are mass \( \|\psi\|_{ \mathbb{L}^2 } \), energy
    \begin{equation*}
        E_{\mu,\lambda} = \frac{\mu}{2} \left\langle H\psi, \psi \right\rangle + \frac{\lambda}{4} \left\langle |\psi|^2 \psi, \psi \right\rangle
        ,
    \end{equation*}
    where \( H = -\Delta + |x|^2 \) and \( \langle f,g \rangle := \int_{\mathbb{R}^d} f\bar{g} \),
    and momentum
    \begin{equation*}
        M_{\mu, \lambda} = \left( E_{\mu,\lambda} - \mu \|x\psi\|^2_{ \mathbb{L}^2 } \right)^2 + \mu^2 \left( \Im \int x\cdot \nabla \psi \bar{\psi} \right)^2
        ,
    \end{equation*}
    and the same applied to any power \( (-H)^s \psi \). There also holds the non radial conservation law
    \begin{equation*}
        \mathcal{P}_j = \frac{1}{4} \left( \Im \int \partial_{j} \psi \bar{\psi} \right)^2 + \mu^2 \left( \int x_j |\psi|^2 \right)^2, \quad j=1, 2
        .
    \end{equation*}
\end{lemma}

\begin{proof}
    See Appendix \ref{appendix: conserved}.
\end{proof}





\subsection{Renormalized flow}
\label{sect: HO -- Renormalized flow}

By linearity of the Harmonic oscillator part, we can reduce the problem to calculating the evolution of the decomposition \eqref{eqn: generic discretization of psi -- expression for uj} for only one bubble $v_j$.
Recall the expression of \( u(t,x) \):
\begin{equation}
    \label{eqn: HO -- Renormalized flow}
    u(t,x) = \frac{A}{L} e^{i\gamma + iL\beta \cdot y - i\frac{B}{4} |y|^2} v(s,y), \quad y = \frac{x-X(t)}{L(t)}, \, \frac{\textd s}{\textd t} = \frac{1}{L(t)^2}
    .
\end{equation}
We compute, in dimension \( d \geq 1 \):
\begin{equation*}
    \Delta_x u = \frac{Ae^{i\gamma}}{L^3} \Delta_y \left[ e^{iL\beta\cdot y - i\frac{B}{4} |y|^2} v(s, y) \right]
    ,
\end{equation*}
and
\begin{equation}
    \partial_{k} \left[ e^{ iL\beta\cdot y - i \frac{B}{4} |y|^2 } v \right]
    = e^{iL\beta\cdot y - i \frac{B}{4} |y|^2} \left[ \partial_{k} v + i\left( L\beta_k - \frac{B}{2} y_k \right) v \right],\quad k=1, \dots, d
    ,
\end{equation}
and
\begin{align*}
     & \partial_{k}^2 \left[ e^{iL\beta\cdot y - i \frac{B}{4} |y|^2} v \right]  = e^{iL\beta\cdot y - i\frac{B}{4} |y|^2}                                                                                                                                   \\
     & \quad \times \left[ \partial_{k}^2 v + i \left( L\beta_k - \frac{B}{2} y_k \right) \partial_{k} v - i\frac{B}{2} v + i\left( L\beta_k - \frac{B}{2} y_k \right) \left[ \partial_{k} v + i \left( L\beta_k - \frac{B}{2} y_k \right) v \right] \right] \\
     & = e^{iL\beta\cdot y - i \frac{B}{4} |y|^2} \left[ \partial_{k}^2 v + i\left( 2L\beta_k - By_k \right) \partial_{k} v + \left( -i\frac{B}{2} - L^2\beta_k^2 + LB\beta_k y_k - \frac{B^2}{4} y_k^2 \right) v \right].
\end{align*}
Hence,
\begin{align*}
    \Delta_x u & = \frac{A}{L^3} e^{i\gamma + iL\beta\cdot y - i\frac{B}{4} |y|^2 } v                                                                                                             \\
               & \quad \times \left[ \Delta_y v + i\left( 2L\beta - By \right) \cdot \nabla v + \left( -i\frac{B}{2}d  - L^2|\beta|^2 + LB\beta \cdot y - \frac{B^2}{4} |y|^2  \right) v \right].
\end{align*}
We have
\begin{align*}
    -|x|^2 u
     & = -\frac{A}{L} e^{i\gamma + iL\beta\cdot y - i \frac{B}{4} |y|^2} \left| Ly + X \right|^2 v                                 \\
     & = \frac{A}{L^3} e^{i\gamma + iL\beta\cdot y - i \frac{B}{4} |y|^2} \left( -L^4 |y|^2 - 2L^3 X \cdot y - L^2 |X|^2 \right) v
    ,
\end{align*}
thus
\begin{equation}
    \label{eqn: HO -- -Hu with general v}
    \begin{aligned}
        (\Delta_x - |x|^2) u & = \frac{A}{L^3} e^{i\gamma + iL\beta\cdot y - i \frac{B}{4} |y|^2} \left\{ \Delta_y v - i B \left( \frac{d}{2}  v + \Lambda v \right) - L^2\left( |\beta|^2 + |X|^2 \right) v \right. \\
                             & \qquad \left. +2iL\beta\cdot \nabla v + \left( LB\beta - 2L^3 X \right) \cdot yv + \left( -\frac{B^2}{4} -L^4 \right) |y|^2 v \right\},
    \end{aligned}
\end{equation}
where we denoted \( \Lambda v := y \cdot \nabla v \).
We now compute
\begin{align*}
    \partial_{t} u & = \partial_{t} \left( e^{i\gamma + i \beta\cdot (x-X) - i\frac{B}{4L^2} |x-X|^2} \frac{A}{L} v(s,y)  \right) \nonumber \\
                   &
    \begin{aligned}
         & = e^{i\gamma + i\beta\cdot (x-X) - i\frac{B}{4L^2} |x-X|^2} \frac{A}{L} \left[ \partial_{t} v + \frac{A_t}{A} v - \frac{L_t}{L} (v + \Lambda v) - \frac{X_t}{L} \cdot \nabla v \right] \\
         & \quad + e^{i\gamma + i\beta\cdot (x-X) - i\frac{B}{4L^2} |x-X|^2} \frac{A}{L} iv                                                                                                       \\
         & \qquad \times \left[ \gamma_t + \beta_t \cdot (x-X) - \beta\cdot X_t - \frac{B_t}{4L^2} |x-X|^2 \right.                                                                                \\
         & \qquad\qquad \left. + \frac{2L_t B}{4L^3} |x-X(t)|^2 + \frac{2B}{4L^2} (x-X) \cdot X_t \right]
    \end{aligned}
    \\
                   &
    \begin{aligned}
         & = e^{i\gamma + i\beta\cdot (x-X) - i\frac{B}{4L^2} |x-X|^2} \frac{A}{L^3} \left[ \partial_{s} v + \frac{A_s}{A} v - \frac{L_s}{L} (v+\Lambda v) - \frac{X_s}{L} \cdot \nabla v \right] \\
         & \quad + e^{i\gamma + i\beta\cdot (x-X) - i\frac{B}{4L^2} |x-X|^2} \frac{A}{L^3} iv                                                                                                     \\
         & \qquad \times \left[\gamma_s + L\beta_s \cdot y - \beta \cdot X_s - \frac{B_s}{4} |y|^2 + \frac{2L_s B}{4L} |y|^2 + \frac{B}{2} y\cdot \frac{X_s}{L}  \right],
    \end{aligned}
\end{align*}
and hence
\begin{equation}
    \label{eqn: i dt uj}
    \begin{aligned}
        i \partial_{t} u
         & = e^{i\gamma + i\beta\cdot (x-X) - i\frac{B}{4L^2} |x-X|^2} \frac{A}{L^3} \left\{ i \partial_{s} v + (-\gamma_s + \beta\cdot X_s) v + \left( \frac{A_s}{A} - \frac{L_s}{L}  \right) iv - \frac{L_s}{L} i\Lambda v \right. \\
         & \qquad\qquad \left. - i\frac{X_s}{L} \cdot \nabla v + \left( -L\beta_s - \frac{BX_s}{2L} \right) \cdot yv + \left( \frac{B_s}{4} - \frac{B}{2} \frac{L_s}{L} \right) |y|^2 v \right\}.
    \end{aligned}
\end{equation}
This yields
\begin{equation}
    \label{eqn: HO -- idt u - Hu -- with v}
    \begin{aligned}
        i \partial_{t} u + \Delta_x u - |x|^2 u
         & = \frac{A}{L^3} e^{i\gamma + iL\beta\cdot y - i\frac{B}{4} |y|^2} \left\{ i \partial_{s} v + \left( -\gamma_s + \beta\cdot X_s - L^2 \left( |\beta|^2 + |X|^2 \right) \right) v \right. \\
         & \quad + \left( \frac{A_s}{A} - \frac{L_s}{L} - B \frac{d}{2}  \right) iv + \left( - \frac{L_s}{L} - B \right)i\Lambda v + i \left( 2L\beta - \frac{X_s}{L}  \right) \cdot \nabla v      \\
         & \quad + \left( -2L^3 X + LB\beta - L\beta_s - \frac{B}{2} \frac{X_s}{L}  \right) \cdot yv                                                                                               \\
         & \quad + \left. \Delta_y v + \left[ \frac{B_s}{4} - \left( \frac{B^2}{4} + L^4 \right) - \frac{B}{2} \frac{L_s}{L} \right] |y|^2 v \right\}(s,y).
    \end{aligned}
\end{equation}

Once we have Equation \eqref{eqn: HO -- idt u - Hu -- with v}, we are free to choose the parameters as we wish.
A natural choice is to conjugate the equation back to the original one in variables $(s,y)$, {\em i.e.} to take
\begin{equation}
    \label{diffsys}
    \left|
    \begin{aligned}
         & \gamma_s - \beta \cdot X_s + L^2 \left( |\beta|^2 + |X|^2 \right) = 0                \\
         & \frac{A_s}{A} - \frac{L_s}{L} - \frac{B}{2} d = 0                                    \\
         & - \frac{L_s}{L} - B = 0                                                              \\
         & 2L\beta - \frac{X_s}{L}  = 0                                                         \\
         & -2L^3X + LB\beta - L\beta_s - \frac{BX_s}{2L} = 0                                    \\
         & \frac{B_s}{4} - \left( \frac{B^2}{4} + L^4 \right) - \frac{B}{2} \frac{L_s}{L} = -1.
    \end{aligned}
    \right.
\end{equation}

Certain choices may be more convenient than others, and \eqref{diffsys} is chosen so that \( v \) only has to solve the stationary Harmonic Oscillator in the variables \( (s, y) \):
\begin{equation}
    \label{harmosy}
    (i \partial_t  + \Delta_x  - |x|^2) u(t,x) = 0
    \qquad \iff \qquad
    ( i \partial_{s}  + \Delta_y  - |y|^2 ) v(s,y) = 0.
\end{equation}
Now we see that if $v$ is decomposed in the Hermite basis, we can solve explicitly the previous equation in variable $(s,y)$ and obtain the solution $u(t,x)$ after solving the differential system \eqref{diffsys}.

Any function statisfying equation \eqref{harmosy} can be decomposed in the Hermite basis
\begin{equation*}
    \left\{ \varphi_n := H_{n_1} \cdots H_{n_d}: n\in \mathbb{N}^d \right\}
    ,
\end{equation*}
where the function \( H_k(z) \) denotes the Hermite function of order \( k \in \mathbb{N} \), which satisfies the following diffferential equation:
\begin{equation*}
    H_k''(z) + (2k+1 - z^2) H_k(z) = 0 .
\end{equation*}
A straightforward calculation shows that $$ (- \Delta_y + |y|^2) \varphi_n = (2 |n| + d) \varphi_n $$ where \( |n| := \sum_{k=1}^d n_k \).
Hence from a decomposition
\begin{equation}
    \label{eqn: HO -- decomposition of v into hermite basis}
    v(0,y) = \sum_{n \in\mathbb{N}^d} v_n\varphi_n(y)
\end{equation}
with $v_n \in \mathbb{C}$,
we calculate that
\begin{equation}
    \label{eqvsy}
    v(s,y) = \sum_{n \in\mathbb{N}^d} v_n e^{- (2 |n| + d) i s}\varphi_n(y)
\end{equation}
is solution of \eqref{harmosy}.
It remains to obtain \( u(t,x) \) solution of \eqref{eqn: cNLS with psi -- linear part}.
In order to do this, one simply needs to integrate \eqref{diffsys} (which is independent of \( n \) in the Hermite decomposition), and to plug \eqref{eqvsy} into \eqref{eqn: HO -- Renormalized flow}.
In particular, we need to calculate the time $s(t)$ as a function of the original time $t$.

\subsection{Integrability of the modulation equations}
\label{sect: HO -- Hamiltonian formulation for the free bubble}
We rewrite \eqref{diffsys} as
\begin{equation}
    \label{eqn: modulation ODEs -- linear part wrt time s - with v}
    \left|
    \begin{aligned}
        A_s      & = \frac{AB}{2} (d-2)                    \\
        L_s      & = -BL                                   \\
        B_s      & = -4 + 4L^4 - B^2                       \\
        X_s      & = 2L^2 \beta                            \\
        \beta_s  & = -2L^2 X                               \\
        \gamma_s & = L^2 \left( |\beta|^2 - |X|^2 \right).
    \end{aligned}
    \right.
\end{equation}

In time \( t \), as \( \frac{\textd}{\textd s} = L^2 \frac{\textd}{\textd t}  \), this system is
\begin{equation}
    \label{eqn: modulation ODEs -- linear part wrt time t - with v}
    \left|
    \begin{aligned}
        A_t      & = \frac{AB}{2L^2} (d-2)                                                 \\
        L_t      & = - \frac{B}{L} = -2L \partial_{B} \mathcal{E}                          \\
        B_t      & = -\frac{4}{L^2} + 4L^2 - \frac{B^2}{L^2} = 2L \partial_{L} \mathcal{E} \\
        X_t      & = 2\beta = \nabla_\beta \mathcal{R}                                     \\
        \beta_t  & = -2X = -\nabla_X \mathcal{R}                                           \\
        \gamma_t & =  |\beta|^2 - |X|^2 ,
    \end{aligned}
    \right.
\end{equation}
with
\begin{equation*}
    \mathcal{E}(B, L) = \frac{1}{L^2} \left( 1 + \frac{B^2}{4}  \right) + L^2, \quad \text{ and } \quad
    \mathcal{R}(X, \beta) = |X|^2 + |\beta|^2
    .
\end{equation*}

Let us write explicitly the Darboux-Lie transformation associated with the previous Poisson system, to obtain canonical Hamiltonian coordinates.
We set
\begin{equation*}
    k = \frac{1}{2} \log L, \quad L = e^{2k}
    ,
\end{equation*}
and the system becomes
\begin{equation}
    \label{eqn: modulation ODEs -- linear part wrt time t - with hamiltonian and k}
    \left|
    \begin{aligned}
        k_t      & = - \partial_{B} \mathcal{H} \\
        B_t      & = \partial_{k} \mathcal{H}   \\
        X_t      & = \nabla_\beta \mathcal{H}   \\
        \beta_t  & = - \nabla_X \mathcal{H}     \\
        A_t      & = \frac{AB}{2} (d-2) e^{-4k} \\
        \gamma_t & =  |\beta|^2 - |X|^2,
    \end{aligned}
    \right.
\end{equation}
with
\begin{equation*}
    \mathcal{H}(k, B, X, \beta) = \mathcal{E}(k, B) + \mathcal{R}(X, \beta) = e^{-4k} \left( \frac{B^2}{4} + 1 \right) + e^{4k} + |X|^2 + |\beta|^2.
\end{equation*}

\begin{lemma}
    \label{lemma: HO exact integration of parameters}
    There exists a symplectic change of variable \( (X, B, k, \beta) \mapsto (h, a, \xi, \theta) \in \mathbb{R}\times \mathbb{R}^d  \times [0, 2\pi] \times [0, 2\pi]^d \), such that the Hamiltonian in these variables is given by
    \begin{equation}
        E(h, a, \xi, \theta) = 4h + 2|a|^2
        ,
    \end{equation}
    so that the flow in variable \( (h, a, \xi, \theta) \) is given by
    \begin{equation}
        \label{eqn: lemma -- HO exact integration of parameters - update of action-angle variables}
        \begin{aligned}
            a(t)      & = a(0),           \\
            \theta(t) & = \theta(0) + 2t, \\
            h(t)      & = h(0),           \\
            \xi(t)    & = \xi(0) - 4t.
        \end{aligned}
    \end{equation}
    We have the explicit formulae:
    \begin{equation}
        \label{eqn: lemma -- HO exact integration of parameters - update of parameters - with v}
        \begin{aligned}
            A(t)       & = A(0)  \left( \frac{L(t)}{L(0)}  \right)^{\frac{2-d}{2}},                                                                                     \\
            e^{4k(t)}  & = L(t)^2 = 2h(t) - \cos(\xi(t)) \sqrt{4h(t)^2 - 1},                                                                                            \\
            B(t)       & = 2\sin(\xi(t))\sqrt{4h(t)^2-1},                                                                                                               \\
            X_i(t)     & = \sin(\theta_i(t)) \sqrt{2a_i(t)}, \quad i=1, \dots, d,                                                                                       \\
            \beta_i(t) & = \cos(\theta_i(t)) \sqrt{2a_i(t)}, \quad i=1, \dots, d,                                                                                       \\
            \gamma(t)  & =
            \gamma(0) + \sum_{l=1}^d \frac{a_l(0)}{2} \left[ \sin(2\theta_l(t)) - \sin(2\theta_l(0)) \right]                                                            \\
            s(t)       & = - \frac{1}{2} \arctan\left( \left( 2h(0) + \sqrt{4h(0)^2 - 1} \right) \tan\left( \frac{\xi(0)}{2} - 2t \right) \right)                       \\
                       & \qquad + \frac{1}{2} \arctan\left( \left( 2h(0) + \sqrt{4h(0)^2 - 1} \right) \tan\left( \frac{\xi(0)}{2} \right) \right)  + m_t \frac{\pi}{2},
        \end{aligned}
    \end{equation}
    where, if \( m_0 \in \mathbb{Z} \) is such that \( \frac{\xi(0)}{2} \in m_0\pi + \left[ -\frac{\pi}{2}, \frac{\pi}{2}  \right] \), then \( m_t \in \mathbb{Z} \) is defined by \( \frac{\xi(t)}{2} \in (m_0-m_t)\pi + \left[ -\frac{\pi}{2}, \frac{\pi}{2}  \right] \).
\end{lemma}

The proof of this Lemma is given in Appendix \ref{appendix: proof of lemma exact integration of parameters}.
If one knows the parameters \( (A, L, B, X, \beta, \gamma) \), it suffices to apply \eqref{eqn: lemma -- HO exact integration of parameters - update of parameters - with v} in order to update them.
Then we combine these expressions with the decomposition \eqref{eqvsy} and the expression of $s(t)$ to obtain the expression of $u(t,x)$.

In practice, one knows the bubble's parameters and needs first to obtain action-angle variables. We have the following result:
\begin{lemma}
    \label{lemma: HO change of variables between params and action-angle variables}
    The change of variables \( (L, B, X, \beta) \mapsto (h, a, \xi, \theta) \) is explicit, and at time \( t=0 \) we have
    \begin{equation}
        \label{eqn: lemma -- HO exact integration of parameters - change of variables}
        \begin{aligned}
            a_i(0)      & = \frac{1}{2} \left( X_i(0)^2 + \beta_i(0)^2 \right), i=1, \dots, d, \\
            \theta_i(0) & = \arctan\left( \frac{X_i(0)}{\beta_i(0)}  \right), i=1, \dots, d,   \\
            h(0)        & = \frac{L(0)^4 + 1 + \frac{B(0)^2}{4}}{4L(0)^2},                     \\
            \xi(0)      & = \arctan\left( \frac{B(0)}{4h(0) - 2L(0)^2}  \right),
        \end{aligned}
    \end{equation}
    whenever \( \theta_i(0) \) and \( \xi(0) \) are well-defined.
    When any one of them is ill-defined -- which happens when \( X_i(0) = \beta_i(0) = 0, i\in \{1, \dots, d\} \) or when \( L(0)=1 \text{ and } B(0)=0 \) -- any value can be taken and the time-evolution of \( A(t), L(t), B(t), X(t), \beta(t) \) and \( \gamma(t) \) will not depend on the value.
    Moreover, in the cases where \( a_i(0) = 0,\, i\in\{1, \dots, d\} \) or \( h(0) = \frac{1}{2} \), the formula \eqref{eqn: lemma -- HO exact integration of parameters - change of variables} for \( \theta_i(0),\, i\in \{1, \dots, d\} \) or \( \xi(0) \) are ill-defined, but any value can be taken as a substitution and this will not affect the behavior of the mappings \( t\mapsto \gamma(t) \) and \( t \mapsto s(t) \).
\end{lemma}

The proof of Lemma \ref{lemma: HO change of variables between params and action-angle variables} is given in Appendix \ref{appendix: proof change of variables HO}.
Note that to define \( \xi(0) \) we could also use equation \eqref{proof: HO proof -- definition of xi with k}, but this is not appropriate from a computational point of view.
Some more details are given in Remark \ref{rmk: HO -- numerical considerations}.

Using Lemmata \ref{lemma: HO exact integration of parameters} and \ref{lemma: HO change of variables between params and action-angle variables}, we are now able to obtain a straightforward numerical algorithm which simulates exactly the evolution of bubbles according to the Harmonic Oscillator on a time interval \( [0, T] \).
It is described in Algorithm \ref{algo: HO -- exact solve}.
\begin{algorithm}
    \caption{Solving the Harmonic oscillator with Bubbles}
    \label{algo: HO -- exact solve}
    \begin{algorithmic}
        \For{\( j = 1, \ldots, N \)}
        \Comment{\(j\) denotes a bubble's index in \eqref{eqn: generic discretization of psi -- ansatz}}
        \State Use \eqref{eqn: lemma -- HO exact integration of parameters - change of variables} to get the action-angle variables \( (h, a, \xi, \theta) \) at time 0.
        \State Use \eqref{eqn: lemma -- HO exact integration of parameters - update of action-angle variables} to update the variables \( (h, a, \xi, \theta) \) up to time \( T \).
        \State Use \eqref{eqn: lemma -- HO exact integration of parameters - update of parameters - with v} to get the parameters of bubble \( u_j \) at time \( T \).
        \State Use \eqref{eqvsy} to update the Hermite decomposition of the bubble.         \EndFor
    \end{algorithmic}
\end{algorithm}

\begin{remark}[Numerical considerations]
    \label{rmk: HO -- numerical considerations}
    Here are a few remarks about Algorithm \ref{algo: HO -- exact solve}:
    \begin{itemize}
        \item When applying equation \eqref{eqn: lemma -- HO exact integration of parameters - change of variables} to obtain the action-angle variables from the bubbles' parameters, it is advised to use the function \( \arctanTwo(y, x) \) instead of \( \arctan(y/x) \) because it allows to obtain an angle lying in \( (-\pi, \pi] \) instead of \( (-\pi/2, \pi/2] \) by taking into account the signs of both \( x \) and \( y \).
              This is also the reason why we do not define \( \xi(0) \) by \eqref{proof: HO proof -- definition of xi with k}.
              Moreover most numerical implementations of \( \arctanTwo \) return a finite value for \( \arctanTwo(0, 0) \), which avoids the manual tuning of a numerical threshold to know whether \( a_i(0) \) or \( h(0) \) vanish numerically or not.
              We recall that in this case the exact value returned does not impact the behavior of \( t\mapsto (L(t), B(t), X(t), \beta(t), \gamma(t),s(t)) \).

        \item The family \( \left\{ \varphi_n =   H_{n_1} \cdots H_{n_d} : n = (n_1, \dots, n_d) \in \mathbb{N}^d \right\} \) is an orthonormal family of \( \mathbb{L}^2(\mathbb{R}^d) \), hence the discretization of any initial condition is done by calculating (or choosing) the Hermite coefficients of the functions $v_j$ in the decomposition \eqref{eqn: generic discretization of psi -- ansatz}.

        \item The algorithm yields an exact solution as soon as the initial data is a sum of bubbles.
              If not, then the only error committed is the discretization error when approximating the initial condition \( \psi(t=0) \) by the ansatz \eqref{eqn: generic discretization of psi -- ansatz}.

        \item This numerical algorithm does not need any discretization in time nor in space.

        \item The solution obtained is the \emph{exact} solution of the equation \eqref{eqn: cNLS with psi -- linear part} defined on the whole space \( \mathbb{R}^d \), and no numerical boundary conditions are needed.
    \end{itemize}
\end{remark}


\section{The Dirac-Frenkel-MacLachlan principle}
\label{sect: DFMP}

In this section we consider the Schrödinger equation \eqref{eqn: cNLS with psi}.
As it has been explained before, the equation consists in two parts: the linear part \eqref{eqn: cNLS with psi -- linear part}, and the nonlinear part \eqref{eqn: cNLS with psi -- nonlinear part}.
Section \ref{sect: The Harmonic Oscillator} was dedicated to solving the Harmonic oscillator, namely the linear part.
We are interested now in solving the nonlinear part, and as it is usually done for numerical simulations, we will use a splitting method (see for instance \cite{mclachlanSplittingMethods2002,hairerGeometricNumericalIntegration2006,casasHighorderHamiltonianSplitting2017}).
This will allow us to solve \eqref{eqn: cNLS with psi} by solving separately \eqref{eqn: cNLS with psi -- linear part} and \eqref{eqn: cNLS with psi -- nonlinear part}, one after the other.
By doing so, a splitting error is made, which depends on the timestep \( \Delta t \), and the order of the error depends on the specific splitting method.
It is also possible to apply high-order splitting methods.

\vspace{1em}

We focus on approximating numerically the solution to \eqref{eqn: cNLS with psi -- nonlinear part}:
\begin{equation*}
    i \partial_{t} \psi = \psi |\psi|^2
    .
\end{equation*}
We are free to use any method we want, but one has to keep in mind that Algorithm \ref{algo: HO -- exact solve} solves \eqref{eqn: cNLS with psi -- linear part} exactly when \( \psi \) is expressed under the form \eqref{eqn: generic discretization of psi -- ansatz}, i.e. as a sum of bubbles.
Therefore we would like the approximate solution to \eqref{eqn: cNLS with psi -- nonlinear part} to keep this particular form.
This naturally calls for the use of the Dirac-Frenkel-MacLachlan principle (abbreviated DFMP).
For more details, see \cite[Sect.~3]{lasserComputingQuantumDynamics2020}.

In theory, the calculation can be performed in a very general situation, when all the $v_j$ involved in \eqref{eqn: generic discretization of psi -- ansatz}-\eqref{eqn: generic discretization of psi -- expression for uj} are given in terms of Hermite polynomials.
In essence, the only difficulty lies in the evaluation of general integrals of products of Hermite functions in different modulation frames, which can be done using generating functions techniques for instance.
Another alternative would be to use nonlinear solitons and rely on numerical evaluations of the corresponding integrals.

In the remainder of this paper, we will consider the primary case by considering
\( \mathcal{M} \) be a manifold of complex-valued Gaussian functions:
\begin{equation}
    \label{eqn: manifold for the dirac-frenkel-maclachlan principle}
    \mathcal{M} := \left\{
        u \in \mathbb{L}^2( \mathbb{R}^d )
        \left|
        \begin{aligned}
            u(x) = \sum_{j=1}^N \frac{A_j}{L_j} e^{i\gamma_j + i\beta_j \cdot (x - X_j) - \frac{2+iB_j}{4L_j^2} \left| x - X_j \right|^2} & , \\
            A_j, B_j, \gamma_j \in \mathbb{R},\, L_j \in \mathbb{R}_+^*, \, X_j, \beta_j \in \mathbb{R}^d                                 &
        \end{aligned}
    \right.
    \right\}
    .
\end{equation}

\begin{remark}
    Note that the functions \( e^{-\frac{|y|^2}{2}} \) are simply \( \varphi_{(0, \dots, 0)}(y) = H_0(y_1) \cdots H_0(y_d) \), hence we can use Section \ref{sect: HO -- Hamiltonian formulation for the free bubble} for the linear part.
\end{remark}

We look for a function \( u \in \mathcal{M} \) that solves \eqref{eqn: cNLS with psi -- nonlinear part} on \( \mathcal{M} \).
More precisely, \( u \) is defined such that its time derivative lies in the tangent space of \( \mathcal{M} \) at \( u \), \( \mathcal{T}_{u(t)} \mathcal{M} \), and such that the residual of equation \eqref{eqn: cNLS with psi -- nonlinear part} is orthogonal to the tangent space.
That is,
\begin{equation}
    \label{eqn: definition of u(t) in DFMP}
    \begin{aligned}
         & \partial_{t} u(t) \in \mathcal{T}_{u(t)} \mathcal{M}, \quad \text{ such that }                                      \\
         & \qquad \langle f, i \partial_{t} u(t) - u(t) |u(t)|^2 \rangle = 0, \, \forall f \in \mathcal{T}_{u(t)} \mathcal{M}.
    \end{aligned}
\end{equation}

\begin{remark}
    The definition of \( \partial_{t} u(t) \) via \eqref{eqn: definition of u(t) in DFMP}, initially proposed by Dirac and Frenkel \cite{diracNoteExchangePhenomena1930,frenkelWaveMechanicsAdvanced1934}, has been later criticized by MacLachlan \cite{mclachlanVariationalSolutionTimedependent1964}.
    He proposed an alternative approach, which would consist in minimizing the quantity
    \begin{equation*}
        \left|\left| i \partial_{t} u(t) - |u(t)|^2 u(t) \right|\right|^2
        .
    \end{equation*}
    However, the two formulations are equivalent if the tangent space \( \mathcal{T}_{u(t)} \mathcal{M} \) is \( \mathbb{C} \)-linear \cite{broeckhoveEquivalenceTimedependentVariational1988}.
    This is the case here because multiplying by the complex unit \( i \) simply amounts to \( \gamma_j \mapsto \gamma_j + \frac{\pi}{2} \).
    Therefore, the approaches by Dirac-Frenkel and MacLachlan are equivalent.
\end{remark}
\bigskip

Let \( B_{u(t)} \) be a basis of \( \mathcal{T}_{u(t)} \mathcal{M} \), then \eqref{eqn: definition of u(t) in DFMP} is equivalent to
\begin{equation}
    \label{eqn: definition of u(t) in DFMP -- with general basis of Tu M}
    \begin{aligned}
         & \partial_{t} u(t) \in \mathcal{T}_{u(t)} \mathcal{M}, \quad \text{ such that }                               \\
         & \qquad \langle f, i \partial_{t} u(t) \rangle = \langle f, u(t) |u(t)|^2 \rangle, \, \forall f \in B_{u(t)}.
    \end{aligned}
\end{equation}
A family (which may happen to be linearly dependent) spanning the tangent space \( \mathcal{T}_{u(t)} \mathcal{M} \) is given by
\begin{equation}
    \label{eqn: base of Tu M in DFMP}
    \begin{aligned}
        B_{u(t)}
         & = \left\{ e^{i\Gamma_j(y_j) - \frac{| y |^2}{2}},
            (y_j)_1 e^{i\Gamma_j(y_j) - \frac{| y_j |^2}{2}},
            \dots,
            (y_j)_d e^{i\Gamma_j(y_j) - \frac{| y_j |^2}{2}},
            |y_j|^2 e^{i\Gamma_j(y_j) - \frac{| y_j |^2}{2}}
        : j=1, \dots, N   \right\},                                                             \\
         & =: \left\{ b_{j,1}, b_{j, 2}, \dots, b_{j, d+1}, b_{j,d+2} : j=1, \dots, N \right\},
    \end{aligned}
\end{equation}
where we defined
\begin{equation*}
    \Gamma_j(y_j) := \gamma_j + L_j \beta_j \cdot y_j - \frac{B_j}{4} |y_j|^2
    .
\end{equation*}

Thus, \eqref{eqn: definition of u(t) in DFMP -- with general basis of Tu M} is equivalent to
\begin{equation}
    \label{eqn: definition of u(t) in DFMP -- with particular basis of Tu M}
    \langle i \partial_{t} u(t), b_{j, l} \rangle = \langle u |u|^2, b_{j, l}  \rangle
    , \quad j = 1, \dots, N,\quad l=1, \dots, d+2
    .
\end{equation}

The next step consists in expressing \eqref{eqn: definition of u(t) in DFMP -- with particular basis of Tu M} as a linear system involving the parameters of the bubbles and their time derivative.
We then solve the linear system, which yields ODEs on the parameters that we can integrate numerically.
The main advantage of this approach is that it guarantees to keep the approximate solution of \eqref{eqn: cNLS with psi -- nonlinear part} as a sum of \( N \) bubbles.

In order to obtain the linear system, we first have to get the expression of \( i\partial_{t} u(t) \)  when \( v(y) = e^{-\frac{|y|^2}{2}} \): by summing \eqref{eqn: i dt uj} over \( j=1,\dots, N \), one has
\begin{equation}
    \label{eqn: i dt u}
    \begin{aligned}
        i \partial_{t} u
        = \sum_{j=1}^N \frac{u_j}{L_j^2}
         & \left\{ |y_j|^2 \left( i \frac{(L_j)_s}{L_j} - \frac{B_j (L_j)_s}{2L_j} + \frac{(B_j)_s}{4} \right) \right. \\
         & \quad+ y_j \cdot \left( -L_j(\beta_j)_s + i\frac{(X_j)_s}{L_j} - \frac{B_j}{2L_j} (X_j)_s  \right)          \\
         & \quad\left. + i\frac{(A_j)_s}{A_j} - i\frac{(L_j)_s}{L_j} + \beta \cdot (X_j)_s - (\gamma_j)_s \right\}.
    \end{aligned}
\end{equation}
More concisely, we have
\begin{equation}
    \label{eqn: i dt u -- with basis elements}
    \begin{aligned}
        i \partial_{t} u
         & = \sum_{j=1}^N \frac{A_j}{L_j^3} e^{i\Gamma_j -\frac{| y_j |^2}{2}}
        \left\{ |y_j|^2 \left( E_j^{(5)} + i E_j^{(6)}  \right) \right.                                                                                            \\
         & \hspace{6em}+ y_j \cdot \left( E_j^{(3)(1, \dots, d)} + iE_j^{(4)(1, \dots, d)}   \right)                                                               \\
         & \hspace{6em}\left. + \left( E_j^{(1)} + iE_j^{(2)} \right) \right\}                                                                                     \\
         & = \sum_{j=1}^N \frac{A_j}{L_j^3}
        \left\{ b_{j, 1} \left( E_j^{(1)} + iE_j^{(2)} \right) + b_{j, 2} \left( E_j^{(3)(1)} + iE_j^{(4)(1)} \right) \right.                                      \\
         & \hspace{6em}\ \, \cdots  + b_{j, d+1} \left( E_j^{(3)(d)} + iE_j^{(4)(d)} \right) + \left. b_{j, d+2} \left( E_j^{(5)} + i E_j^{(6)}  \right) \right\},
    \end{aligned}
\end{equation}
where
\begin{equation}
    \label{eqn: DFMP -- definition of the Ej}
    \begin{alignedat}{3}
        &E_j^{(1)} := \beta_j \cdot (X_j)_s - (\gamma_j)_s,        &\qquad& E_j^{(2)} := \frac{(A_j)_s}{A_j} - \frac{(L_j)_s}{L_j}, \\
        &E_j^{(3)(l)} := -L_j((\beta_j)_l)_s - \frac{B_j}{2L_j} ((X_j)_l)_s, &\qquad& E_j^{(4)(l)} := \frac{((X_j)_l)_s}{L_j}, \qquad l = 1, \dots, d, \\
        &E_j^{(5)} := \frac{(B_j)_s}{4} - \frac{B_j (L_j)_s}{2L_j}, &\qquad& E_j^{(6)} := \frac{(L_j)_s}{L_j},
    \end{alignedat}
\end{equation}
and where \( E_j^{(k)(1, \dots, d)} \) denotes the vector \( (E_j^{(k)(1)}, \dots,  E_j^{(k)(d)}) \).
We recall the subscript \( {}_s \) denotes the derivative with respect to time \( s \).

According to \eqref{eqn: definition of u(t) in DFMP -- with particular basis of Tu M}, we then want to project \( i \partial_{t} u(t) \) against every element of \( B_{u(t)} \).
We obtain the following linear system:
\begin{equation}
    \label{eqn: linear system for DFMP}
    \mathbf{A} \mathbf{E} = S
    ,
\end{equation}
where
\begin{equation*}
    \mathbf{A} :=
    \begin{pmatrix}
        \langle b_{1, 1}, b_{1, 1} \rangle   & \langle b_{1, 2}, b_{1, 1} \rangle   & \dots & \langle b_{N, d+1}, b_{1, 1} \rangle   & \langle b_{N, d+2}, b_{1, 1} \rangle   \\
        \vdots                               &                                      &       &                                        & \vdots                                 \\
        \langle b_{1, 1}, b_{N, d+2} \rangle & \langle b_{1, 2}, b_{N, d+2} \rangle & \dots & \langle b_{N, d+1}, b_{N, d+2} \rangle & \langle b_{N, d+2}, b_{N, d+2} \rangle \\
    \end{pmatrix}
    \in \mathbb{C}^{(d+2)N, (d+2)N}
    ,
\end{equation*}
\begin{equation*}
    \mathbf{E} :=
    \begin{pmatrix}
        \frac{A_1}{L_1^3} \left( E_1^{(1)} + i E_1^{(2)} \right)       \\
        \frac{A_1}{L_1^3} \left( E_1^{(3)(1)} + i E_1^{(4)(1)} \right) \\
        \vdots                                                         \\
        \frac{A_1}{L_1^3} \left( E_1^{(3)(d)} + i E_1^{(4)(d)} \right) \\
        \frac{A_1}{L_1^3} \left( E_1^{(5)} + i E_1^{(6)} \right)       \\
        \vdots                                                         \\
        \frac{A_j}{L_j^3} \left( E_j^{(1)} + i E_j^{(2)} \right)       \\
        \frac{A_j}{L_j^3} \left( E_j^{(3)(1)} + i E_j^{(4)(1)} \right) \\
        \vdots                                                         \\
        \frac{A_j}{L_j^3} \left( E_j^{(3)(d)} + i E_j^{(4)(d)} \right) \\
        \frac{A_j}{L_j^3} \left( E_j^{(5)} + i E_j^{(6)} \right)       \\
        \vdots                                                         \\
        \frac{A_N}{L_N^3} \left( E_N^{(1)} + i E_N^{(2)} \right)       \\
        \frac{A_N}{L_N^3} \left( E_N^{(3)(1)} + i E_N^{(4)(1)} \right) \\
        \vdots                                                         \\
        \frac{A_N}{L_N^3} \left( E_N^{(3)(d)} + i E_N^{(4)(d)} \right) \\
        \frac{A_N}{L_N^3} \left( E_N^{(5)} + i E_N^{(6)} \right)
    \end{pmatrix}
    \in \mathbb{R}^{(d+2)N},
\end{equation*}
and
\begin{equation*}
    \mathbf{S} :=
    \begin{pmatrix}
        \langle u|u|^2, b_{1, 1} \rangle   \\
        \vdots                             \\
        \langle u|u|^2, b_{N, d+2} \rangle \\
    \end{pmatrix}
    \in \mathbb{C}^{(d+2)N}
    .
\end{equation*}

The matrix \( \mathbf{A} \) is the Gram matrix of the family \( B_{u(t)} \), which obviously depends on time.
In order to solve the linear system \eqref{eqn: linear system for DFMP} we shall use the Moore-Penrose pseudo-inverse which always exists, and which corresponds to the Least Squares solution if the matrix \( \mathbf{A}^* \mathbf{A} \) is invertible.
The matrix \( \mathbf{A} \) is invertible if and only if \( B_{u(t)} \) is a linearly independent family of \( \mathbb{L}^2( \mathbb{R}^d ) \).
We can already notice that if two bubbles have the same parameters then the family will be linearly dependent: this is why the Moore-Penrose pseudo-inverse is used instead of \( \mathbf{A}^{-1} \), which is not always well-defined.

Once the linear system \eqref{eqn: linear system for DFMP} is solved, we obtain \( \mathbf{E} \), from which we can update the modulation parameters.
In order to solve numerically the linear system, we shall rewrite it under a more convenient form.
Let \( \mathbf{A}_\Re := \Re (\mathbf{A}) \), \( \mathbf{A}_\Im := \Im (\mathbf{A}) \), \( \mathbf{E}_\Re := \Re (\mathbf{E}) \), \( \mathbf{E}_\Im := \Im (\mathbf{E}) \), \( \mathbf{S}_\Re := \Re (\mathbf{S}) \), and \( \mathbf{S}_\Im := \Im (\mathbf{S}) \).
Then, \eqref{eqn: linear system for DFMP} writes:
\begin{align}
    \mathbf{A} \mathbf{E} = \mathbf{S}
     & \iff (\mathbf{A}_\Re + i \mathbf{A}_\Im)(\mathbf{E}_\Re + i \mathbf{E}_\Im) = \mathbf{S}_\Re + i \mathbf{S}_\Im \nonumber \\
     & \iff \left\{
        \begin{aligned}
            \mathbf{A}_\Re \mathbf{E}_\Re - \mathbf{A}_\Im \mathbf{E}_\Im & = \mathbf{S}_\Re \\
            \mathbf{A}_\Im \mathbf{E}_\Re + \mathbf{A}_\Re \mathbf{E}_\Im & = \mathbf{S}_\Im
        \end{aligned}
    \right. \nonumber                                                                                                            \\
     & \iff
    \begin{pmatrix}
        \mathbf{A}_\Re & - \mathbf{A}_\Im \\
        \mathbf{A}_\Im & \mathbf{A}_\Re
    \end{pmatrix}
    \begin{pmatrix}
        \mathbf{E}_\Re \\
        \mathbf{E}_\Im
    \end{pmatrix}
    =
    \begin{pmatrix}
        \mathbf{S}_\Re \\
        \mathbf{S}_\Im
    \end{pmatrix}
    \label{eqn: linear system for DFMP - real matrices}
    .
\end{align}
It is more convenient to solve \eqref{eqn: linear system for DFMP - real matrices} than \eqref{eqn: linear system for DFMP}, because we only have to deal with real matrices and vectors.

\begin{remark}
    We first tried to solve \eqref{eqn: linear system for DFMP} using the Moore-Penrose pseudo-inverse, however it yielded incomprehensible results.
    After some investigation, we found out that the issue seems to be the complex numbers involved, and that they do not mix well with the pseudo-inverse.
    The linear system \eqref{eqn: linear system for DFMP - real matrices} yields much better results.
\end{remark}

Once \eqref{eqn: linear system for DFMP - real matrices} is solved, we have to update the bubbles parameters according to \eqref{eqn: DFMP -- definition of the Ej}.
The parameters of the bubble labelled \( j \) can be updated with \( (\mathbf{E}_\Re)_k \) and \( (\mathbf{E}_\Im)_k \) for \( k = (d+2)(j-1)+1, \dots, (d+2)j \).
For the sake of clarity, let
\( \mathbf{F} :=
\begin{pmatrix}
    \mathbf{E}_\Re \\
    \mathbf{E}_\Im
\end{pmatrix}
\).
Then
\begin{align*}
    \frac{A_j}{L_j^3}
    \begin{pmatrix}
        E_j^{(1)}    \\
        E_j^{(3)(1)} \\
        \vdots       \\
        E_j^{(3)(d)} \\
        E_j^{(5)}    \\
        E_j^{(2)}    \\
        E_j^{(4)(1)} \\
        \vdots       \\
        E_j^{(4)(d)} \\
        E_j^{(6)}
    \end{pmatrix}
    =    &
    \begin{pmatrix}
        (\mathbf{E}_\Re)_{(d+2)(j-1)+1}   \\
        (\mathbf{E}_\Re)_{(d+2)(j-1)+2}   \\
        \vdots                            \\
        (\mathbf{E}_\Re)_{(d+2)(j-1)+d+1} \\
        (\mathbf{E}_\Re)_{(d+2)(j-1)+d+2} \\
        (\mathbf{E}_\Im)_{(d+2)(j-1)+1}   \\
        (\mathbf{E}_\Im)_{(d+2)(j-1)+2}   \\
        \vdots                            \\
        (\mathbf{E}_\Im)_{(d+2)(j-1)+d+1} \\
        (\mathbf{E}_\Im)_{(d+2)(j-1)+d+2}
    \end{pmatrix}
    =
    \begin{pmatrix}
        \mathbf{F}_{(d+2)(j-1)+1}            \\
        \mathbf{F}_{(d+2)(j-1)+2}            \\
        \vdots                               \\
        \mathbf{F}_{(d+2)(j-1)+d+1}          \\
        \mathbf{F}_{(d+2)(j-1)+d+2}          \\
        \mathbf{F}_{(d+2)N + (d+2)(j-1)+1}   \\
        \mathbf{F}_{(d+2)N + (d+2)(j-1)+2}   \\
        \vdots                               \\
        \mathbf{F}_{(d+2)N + (d+2)(j-1)+d+1} \\
        \mathbf{F}_{(d+2)N + (d+2)(j-1)+d+2}
    \end{pmatrix}
    =:
    \begin{pmatrix}
        F_j^{(1)}    \\
        F_j^{(3)(1)} \\
        \vdots       \\
        F_j^{(3)(d)} \\
        F_j^{(5)}    \\
        F_j^{(2)}    \\
        F_j^{(4)(1)} \\
        \vdots       \\
        F_j^{(4)(d)} \\
        F_j^{(6)}
    \end{pmatrix}
    \\
    \iff &
    \begin{pmatrix}
        \beta_j\cdot (X_j)_s - (\gamma_j)_s          \\
        -L_j (\beta_j)_s - \frac{B_j}{2L_j} (X_j)_s  \\
        \frac{(B_j)_s}{4} - \frac{B_j (L_j)_s}{2L_j} \\
        \frac{(A_j)_s}{A_j} - \frac{(L_j)_s}{L_j}    \\
        \frac{(X_j)_s}{L_j}                          \\
        \frac{(L_j)_s}{L_j}
    \end{pmatrix}
    = \frac{L_j^3}{A_j}
    \Re
    \begin{pmatrix}
        F_j^{(1)} \\ F_j^{(3)(1, \dots, d)} \\ F_j^{(5)} \\ F_j^{(2)} \\ F_j^{(4)(1, \dots, d)} \\ F_j^{(6)}
    \end{pmatrix}
    .
\end{align*}

Hence
\begin{equation*}
    \left\{
        \begin{aligned}
             & \beta_j\cdot (X_j)_s - (\gamma_j)_s = F_j^{(1)},                      \\
             & -L_j (\beta_j)_s - \frac{B_j}{2L_j} (X_j)_s = F_j^{(3)(1, \dots, d)}, \\
             & \frac{(B_j)_s}{4} - \frac{B_j (L_j)_s}{2L_j} = F_j^{(5)},             \\
             & \frac{(A_j)_s}{A_j} - \frac{(L_j)_s}{L_j} = F_j^{(2)},                \\
             & \frac{(X_j)_s}{L_j} = F_j^{(4)(1, \dots, d)},                         \\
             & \frac{(L_j)_s}{L_j} = F_j^{(6)},
        \end{aligned}
    \right.
    \iff
    \left\{
        \begin{aligned}
            (A_j)_s      & = A_j \left( F_j^{(2)} + F_j^{(6)} \right),                                         \\
            (L_j)_s      & = L_j F_j^{(6)},                                                                    \\
            (B_j)_s      & = 4F_j^{(5)} + 2 B_j F_j^{(6)},                                                     \\
            (X_j)_s      & = L_j F_j^{(4)(1, \dots, d)},                                                       \\
            (\beta_j)_s  & = - \frac{1}{L_j} F_j^{(3)(1, \dots, d)} - \frac{B_j}{2L_j} F_j^{(4)(1, \dots, d)}, \\
            (\gamma_j)_s & = L_j \beta_j \cdot F_j^{(4)(1, \dots, d)} - F_j^{(1)},
        \end{aligned}
    \right.
\end{equation*}
and with respect to time \( t \),
\begin{equation}
    \label{eqn: DFMP -- update of parameters with interactions -- wrt time t}
    \left\{
        \begin{aligned}
            (A_j)_s      & = \frac{A_j}{L_j^2}  \left( F_j^{(2)} + F_j^{(6)} \right),                              \\
            (L_j)_s      & = \frac{1}{L_j}  F_j^{(6)},                                                             \\
            (B_j)_s      & = \frac{4}{L_j^2} F_j^{(5)} + \frac{2}{L_j^2}  B_j F_j^{(6)},                           \\
            (X_j)_s      & = \frac{1}{L_j}  F_j^{(4)(1, \dots, d)},                                                \\
            (\beta_j)_s  & = - \frac{1}{L_j^3} F_j^{(3)(1, \dots, d)} - \frac{B_j}{2L_j^3} F_j^{(4)(1, \dots, d)}, \\
            (\gamma_j)_s & = \frac{1}{L_j}  \beta_j \cdot F_j^{(4)(1, \dots, d)} - \frac{1}{L_j^2}
            F_j^{(1)},
        \end{aligned}
    \right.
\end{equation}

\subsection{Computing coefficients of the linear system \eqref{eqn: linear system for DFMP}}

In order to be able to compute \( \mathbf{A} \) and \( \mathbf{S} \), we give the exact expression of the inner products involved.
For \( j, l = 1, \dots, N \), let
\begin{equation}
    \left |
    \begin{aligned}
        z   & := \frac{2+iB_l}{4L_l^2} + \frac{2-iB_j}{4L_j^2},                                                                                                               \\
        a   & := \frac{X_l}{L_l^2} + \frac{X_j}{L_j^2},                                                                                                                       \\
        \xi & := \frac{B_j}{2L_j^2} X_j + \beta_j - \frac{B_l}{2L_l^2} X_l - \beta_l,                                                                                         \\
        C   & = \exp\left\{i(\gamma_l - \gamma_j) - \frac{2+iB_l}{4L_l^2}  | X_l | ^2 - \frac{2-iB_j}{4L_j^2}  | X_j | ^2 - i\beta_l \cdot X_l + i\beta_j \cdot X_j \right\}.
    \end{aligned}
    \right.
\end{equation}
Those quantities obviously depend on the indices \( j \) and \( l \), but for clarity we do not write explicitly these dependences since they are pretty clear.
Then, for \( n, m = 1, \dots, d \),
\begin{align*}
    \langle b_{l, 1}, b_{j, 1} \rangle     & = C \widehat{f}(\xi)                                                                                                                                                      \\
    \langle b_{l, n+1}, b_{j, 1} \rangle   & = \frac{C}{L_l} \left( \widehat{x f}_n - (X_l)_n \widehat{f} \right)(\xi)                                                                                                 \\
    \langle b_{l, d+2}, b_{j, 1} \rangle
                                           & = \frac{C}{L_l^2} \left( \widehat{ | x | ^2 f} - 2X_l \cdot \widehat{x f} +  | X_l | ^2 \widehat{f} \right)(\xi)                                                          \\
    \langle b_{l, n+1}, b_{j, m+1} \rangle
                                           & = \frac{C}{L_j L_l} \left[ \widehat{x_n x_m f} - (X_l)_n \widehat{x_m f} - (X_j)_m \widehat{x_n f} + (X_l)_n (X_j)_m \widehat{f} \right](\xi)                             \\
    \langle b_{l, d+2}, b_{j, m+1} \rangle & =
    \frac{C}{L_l^2 L_j} \left[ \widehat{x_m | x | ^2 f} - 2X_l \cdot \widehat{x_m x f} + | X_l | ^2 \widehat{x_m f} \right.                                                                                            \\
                                           & \hspace{1cm} \left. - (X_j)_m \widehat{ | x | ^2 f} + 2(X_j)_m X_l \cdot \widehat{x f} - | X_l | ^2 (X_j)_m \widehat{f} \right](\xi)                                      \\
    \langle b_{l, d+2}, b_{j, d+2} \rangle & = \frac{C}{L_l^2 L_j^2} \left[ \widehat{ | x | ^4 f} - 2X_l \cdot \widehat{ | x | ^2 x f} +  | X_l | ^2 \widehat{ | x | ^2 f} - 2X_j \cdot \widehat{x | x | ^2 f} \right. \\
                                           & \hspace{2cm}  + 4\sum_{n, m=1}^d (X_l)_n (X_j)_m \widehat{x_n x_m f} - 2 | X_l | ^2 X_j \cdot \widehat{x f}                                                               \\
                                           & \hspace{2cm} \left. + \widehat{ | x | ^2 f}  | X_j | ^2 - 2 | X_j | ^2 X_l \cdot \widehat{x f} +  | X_l | ^2  | X_j | ^2 \widehat{f} \right](\xi)
\end{align*}
Moreover, we recall that \( \mathbf{A} \) is hermitian, so the above relations allow us to obtain all

We now compute the components of the vector \( S \).
For \( j, k, l, m=1, \dots, N \), let
\begin{equation}
    \left|
    \begin{aligned}
        C_\Im & := \exp\left\{ i\left( \gamma_k + \gamma_l - \gamma_m - \gamma_j \right) \right\}                                                                                                                         \\
              & \hspace{1em} \times \exp\left\{ i\left(\beta_j \cdot X_j + \beta_m \cdot X_m - \beta_l\cdot X_l - \beta_k \cdot X_k \right) \right\}                                                                      \\
              & \hspace{1em} \times \exp\left\{- i\left( \frac{B_k}{4L_k^2} |X_k|^2 + \frac{B_l}{4L_l^2} |X_l|^2 - \frac{B_m}{4L_m^2} |X_m|^2 - \frac{B_j}{4L_j^2} |X_j|^2  \right) \right\},                             \\
        C_\Re & := \exp\left\{-\frac{1}{2} \left( \frac{|X_k|^2}{L_k^2} + \frac{|X_l|^2}{L_l^2} + \frac{|X_m|^2}{L_m^2} + \frac{|X_j|^2}{L_j^2} \right) \right\},                                                         \\
        C     & := \frac{A_k A_l A_m}{L_k L_l L_m} C_\Im C_\Re,                                                                                                                                                           \\
        \xi   & := -\left[ \beta_k + \beta_l - \beta_m - \beta_j + \frac{B_k}{2L_k^2} X_k + \frac{B_l}{2L_l^2} X_l - \frac{B_m}{2L_m^2} X_m - \frac{B_j}{2L_j^2} X_j \right],                                             \\
        z     & := \frac{1}{2} \left( \frac{1}{L_k^2} + \frac{1}{L_l^2} + \frac{1}{L_m^2} + \frac{1}{L_j^2} \right) + i \left( \frac{B_k}{4L_k^2} + \frac{B_l}{4L_l^2} - \frac{B_m}{4L_m^2} - \frac{B_j}{4L_j^2} \right), \\
        a     & := \frac{1}{L_k^2}
        X_k + \frac{1}{L_l^2} X_l + \frac{1}{L_m^2} X_m + \frac{1}{L_j^2} X_j.
    \end{aligned}
    \right.
\end{equation}
Those quantities obviously depend on the indices \( j, k, l \) and \( m \), but for clarity we do not write explicitly these dependences since they are pretty clear.
Then, for \( 1 \leq r \leq d \),
\begin{align*}
    \langle u|u|^2, b_{j, 1} \rangle   & = \sum_{k,l,m} C \widehat{f}(\xi)                                                                                \\
    \langle u|u|^2, b_{j, r+1} \rangle & = \sum_{k,l,n} \frac{C}{L_j} \left( \widehat{x_r f} - (X_j)_r \widehat{f} \right)                                \\
    \langle u|u|^2, b_{j, d+2} \rangle & = \sum_{k,l,m} \frac{C}{L_j^2} \left( \widehat{|x|^2 f} - 2X_j \cdot \widehat{xf} + |X_j|^2 \widehat{f} \right).
\end{align*}
We refer to Appendix \ref{appendix: populating the linear system for DFMP} for more details.
Moreover, Appendix \ref{appendix: fourier transforms of Gaussians} contains Table \ref{table: useful Fourier transforms} which gives useful Fourier transforms.

\begin{remark}[Computational complexity]
    Throughout this section, we have chosen
    \begin{equation*}
        v_j(s_j, y_j) = e^{-\frac{1}{2} | y_j |^2}.
    \end{equation*}
    This choice was made so that the inner products involved in the application of the DFMP are easily computable in an exact way.
    Therefore we do not rely on numerical integration to compute the coefficients of the linear system \eqref{eqn: linear system for DFMP}.
    In particular, this shows that the computational effort required to obtain the linear system is \( \mathcal{O}(N^4d + N^2(d+2)^2)\).
    To obtain the total complexity, we have to add the cost of computing the pseudo-inverse of the hermitian matrix \( \mathbf{A} \in \mathbb{C}^{(d+2)N, (d+2)N} \), which is \( O((d+2)^3 N^3)\).
    This yields the overall computational complexity: \( \mathcal{O}(N^4d + d^3 N^3)\).
    In a more general setting, one could use the Hermite basis decomposition \eqref{eqn: HO -- decomposition of v into hermite basis} and perform all computations exactly. This would yield more involved computations and we chose the easy way out by experimenting only with Gaussian functions, but this is completely doable.
\end{remark}

We obtain Algorithm \ref{algo: DFMP -- solve approximately cNLS} which can be used to obtain an approximate solution to \eqref{eqn: cNLS with psi} as a sum of bubbles, using the Strang splitting between the linear and nonlinear parts, and using an arbitrary explicit time-integrator for the nonlinear part.

\begin{algorithm}
    \caption{Approximating a solution to \eqref{eqn: cNLS with psi} as a sum of bubbles.}
    \label{algo: DFMP -- solve approximately cNLS}
    \begin{algorithmic}
        \For{ Each timestep of size \( dt \) }
        \For{ \( j = 1, \dots, N \) }
        \Comment{\(j\) denotes a bubble's index.}
        \State Use Algorithm \ref{algo: HO -- exact solve} to update the bubbles over a timestep of size \( dt/2 \).
        \For{each stage of a time-integrator}
        \State Compute the coefficients of the linear system \eqref{eqn: linear system for DFMP}.
        \State Solve the linear system \eqref{eqn: linear system for DFMP} to obtain \( \mathbf{E} \).
        \State Use \eqref{eqn: DFMP -- update of parameters with interactions -- wrt time t} to update the parameters over a timestep whose length depends on the stage of the time-integrator.
        \EndFor
        \State Use Algorithm \ref{algo: HO -- exact solve} to update the bubbles over a timestep of size \( dt/2 \).
        \EndFor
        \EndFor
    \end{algorithmic}
\end{algorithm}


\subsection{Hamiltonian and norm conservation for the interactions}
\label{sect: Hamiltonian and norm conservation for the interactions}

When solving \eqref{eqn: cNLS with psi -- nonlinear part} via the DFMP, i.e. when solving the linear system \eqref{eqn: linear system for DFMP}, a Hamiltonian is conserved.
\begin{lemma}
    Let \( u(t) \) be the approximation to \eqref{eqn: cNLS with psi -- nonlinear part} obtained by applying the Dirac-Frenkel-MacLachlan principle, and define
    \begin{equation*}
        H_{\textnormal{interactions}}(t) := \frac{1}{4} \langle u(t), u(t)| u(t)|^2 \rangle = \frac{1}{4} \langle u(t)^2, u(t)^2 \rangle .
    \end{equation*}
    Then \( H_{\textnormal{interactions}} \) is conserved, i.e.
    \begin{equation*}
        \frac{\textnormal{d}}{\textnormal{d}t}
        H_{\textnormal{interactions}}(t) = 0 ,
    \end{equation*}
    and the \( \mathbb{L}^2 \) norm of \( u \) is also conserved.
\end{lemma}

\begin{proof}
    We have
    \begin{equation*}
        H_{\textnormal{interactions}}(t) := \frac{1}{4} \langle u(t), u(t)| u(t)|^2 \rangle = \frac{1}{4} \langle u(t)^2, u(t)^2 \rangle ,
    \end{equation*}
    by using the Hermitian property of the inner product \( \langle \cdot, \cdot \rangle \).
    Then,
    \begin{align*}
        \frac{\textnormal{d}}{\textnormal{d}t} H_{\textnormal{interactions}}(t)
         & = \frac{1}{4} \frac{\textnormal{d}}{\textnormal{d}t} \langle u(t)^2, u(t)^2 \rangle                                                                 \\
         & = \frac{1}{4} \left\langle 2u(t) \partial_{t} u(t) , u(t)^2 \right\rangle + \frac{1}{4}  \left\langle u(t)^2, 2u(t) \partial_{t} u(t) \right\rangle \\
         & = \Re \left\langle u(t) \partial_{t} u(t) , u(t)^2 \right\rangle                                                                                    \\
         & = \Re \left\langle \partial_{t} u(t) , u(t) |u(t)|^2 \right\rangle.
    \end{align*}
    By definition of \( \partial_{t} u(t) \), we have \( \partial_{t} u(t) \in \mathcal{T}_{u(t)} \mathcal{M} \), hence we can take \( f = \partial_{t} u(t) \) in \eqref{eqn: definition of u(t) in DFMP}.
    We obtain the following equality:
    \begin{equation*}
        \langle \partial_{t} u(t), u(t) |u(t)|^2 \rangle = \langle \partial_{t} u(t), i \partial_{t} u(t) \rangle = -i \| \partial_{t} u(t) \|^2
        .
    \end{equation*}
    Therefore,
    \begin{equation*}
        \frac{\textnormal{d}}{\textnormal{d}t} H_{\textnormal{interactions}}(t) = \Re \left( -i \| \partial_{t} u(t) \|^2 \right) = 0
        .
    \end{equation*}

    Using similar ideas, we can easily show the conservation of the \( \mathbb{L}^2 \) norm: we obviously have \( u(t) \in \mathcal{T}_{u(t)} \mathcal{M} \), hence
    \begin{align*}
        \frac{\textd}{\textd t} \|u(t)\|^2
         & = 2\Re\langle u(t), \partial_{t} u(t) \rangle = 2 \Re \langle u(t), -i u(t)|u(t)|^2 \rangle \\
         & = 2 \Re \left( i \langle |u(t)|^2, |u(t)|^2 \rangle\right) = 0
    \end{align*}
\end{proof}

\subsection{Recovering the Harmonic Oscillator equations}

Suppose the family \( B_{u(t)} \subset \mathbb{L}^2( \mathbb{R}^d ) \) defined by \eqref{eqn: base of Tu M in DFMP} is linearly independent, and consider the equation \eqref{eqn: cNLS with psi -- linear part}.
By summing equation \eqref{eqn: HO -- idt u - Hu -- with v} over \( j=1, \dots, N \) with \( v_j(s_j,y_j) = e^{-\frac{|y_j|^2}{2}} \), and letting this sum be equal to zero, we obtain an equation of the form
\begin{equation}
    \label{eqn: recovering the HO eqns -- write HO with generic coefficients in basis B}
    \sum_{j=1}^N \left( c_{j, 1} b_{j, 1} + c_{j, 2} b_{j, 2} + \dots + c_{j, d+1} b_{j, d+1} + c_{j, d+2} b_{j, d+2} \right) = 0
    .
\end{equation}
Thanks to the assumption that \( B_{u(t)} \) is a linearly independent family, we know that we must have
\begin{equation}
    \label{eqn: recovering HO equations from DFMP -- all coeffs equal to zero}
    c_{k, 1} = c_{k, 2} = \dots = c_{k, d+1} = c_{k, d+2} = 0,\qquad k=1,\dots, (d+2)N
    .
\end{equation}
This yields exactly the system of equations \eqref{eqn: modulation ODEs -- linear part wrt time s - with v}.
In other words, the DFMP approach gives the same equations as those given in Section \ref{sect: HO -- Hamiltonian formulation for the free bubble} when \( B_{u(t)} \) is a linearly independent family.
However, our approach as described in Section \ref{sect: HO -- Hamiltonian formulation for the free bubble} allows to solve them exactly and not only numerically with some numerical time-integrator.

Finally, if the family \( B_{u(t)} \) is linearly dependent, then we cannot write equation \eqref{eqn: recovering HO equations from DFMP -- all coeffs equal to zero} anymore, hence the DFMP approach in the linear case fails.
Our approach avoids this issue by naturally imposing conditions \eqref{eqn: recovering HO equations from DFMP -- all coeffs equal to zero} (which are the same as \eqref{diffsys}).


\section{Numerical examples}
\label{sect: numerical experiments}

In this section we will assess the efficiency of the Bubbles approach with the Dirac-Frenkel-MacLachlan approach against a spectral method in the two-dimensional case.

\subsection{Spectral scheme}

We start by discussing the spectral method we shall use to compare with the results of Algorithm \ref{algo: DFMP -- solve approximately cNLS}.
We refer to \cite{fornbergPracticalGuidePseudospectral1996} for a general introduction to spectral methods for the Schr{\"o}dinger equation, and to \cite{antoineComputationalMethodsDynamics2013} for grid-based schemes applied to the Gross-Pitaevskii equation.

We now present a method which can be understood as the application of \cite{bernierExactSplittingMethods2021} to a simpler equation, namely the Harmonic Oscillator.
We use a splitting method to simulate the linear part \eqref{eqn: cNLS with psi -- linear part}, and thanks to \cite{bernierExactSplittingMethods2020,alphonsePolarDecompositionSemigroups2021} we have:
\begin{align}
    e^{-it (-\Delta + |x|^2)}
     & = e^{- \frac{1}{2} \tanh(it) |x|^2} e^{\frac{1}{2} \sinh(2it)\Delta_x} e^{- \frac{1}{2} \tanh(it) |x|^2} \nonumber \\
     & = e^{- \frac{i}{2} \tan(t) |x|^2} e^{\frac{i}{2} \sin(2t)\Delta_x} e^{- \frac{i}{2} \tan(t) |x|^2}
    \label{eqn: num -- exact time splitting of HO}
    .
\end{align}
We can cite \cite{jinMathematicalComputationalMethods2011} which also presents a spectral method based on the Fourier transform with time splitting, however the above method is different in that \eqref{eqn: num -- exact time splitting of HO} is exact and hence we do not have any time-splitting error.

The first and third exponentials on the RHS are straightforward to compute on a grid.
For the second one, we use a Fourier transform: \( e^{\frac{i}{2} \sin(2t)\Delta_x} \) is the propagator of the following equation:
\begin{equation*}
    \partial_{t} \psi = i\cos(2t) \Delta_x \psi
    .
\end{equation*}
By using a Fourier transform, we get
\begin{equation*}
    \partial_{t} \mathcal{F}(\psi)(\xi)
    = i\cos(2t) \mathcal{F} \left( \Delta_x \psi\right)(\xi)
    = -i\cos(2t) |\xi|^2 \mathcal{F} \left( \psi \right)(\xi)
    .
\end{equation*}
Hence,
\begin{equation*}
    \mathcal{F}(\psi(t, \cdot))(\xi) = e^{-\frac{i}{2} \sin(2t) |\xi|^2} \mathcal{F}(\psi(0, \cdot))(\xi)
    .
\end{equation*}
The RHS exponential is straightforward to compute in the Fourier space.
Hence, an exact-time spectral approximation of the solution to \eqref{eqn: cNLS with psi -- linear part} is given by Algorithm \ref{algo: num -- spectral solver linear part}.
From this, it is easy to obtain an algorithm which simulates \eqref{eqn: cNLS with psi} with interactions.
It consists in using a Strang splitting method on \eqref{eqn: cNLS with psi}, by splitting the linear part \eqref{eqn: cNLS with psi -- linear part} and the nonlinear part \eqref{eqn: cNLS with psi -- nonlinear part}.
The linear part is approximated via Algorithm \ref{algo: num -- spectral solver linear part}, and the computation of interactions is explicit thanks to the fact that \( |u(t, x)|^2 \) does not depend on time (see e.g. \cite[Sect.~2.2]{faouGeometricNumericalIntegration2012}).
This fully describes Algorithm \ref{algo: num -- spectral solver}.

\begin{algorithm}
    \caption{Spectral solver for \eqref{eqn: cNLS with psi -- linear part}, with an exact time resolution for each splitting step.}
    \label{algo: num -- spectral solver linear part}
    \begin{algorithmic}
        \State Discretize the initial data \( \eta \) on a \( \text{Grid} \subset \mathbb{R}^d \).
        \For{ Each timestep of size \( \Delta t \) }
        \For{ \( x \in \textsc{Grid} \) }
        \Comment{\(x \in \mathbb{R}^d\).}
        \State Multiply \( \eta(x) \) by \( e^{- \frac{i}{2} \tan(\Delta t) |x|^2} \).
        \EndFor
        \State Apply a FFT to \( \eta \).
        \Comment{FFT: Fast Fourier Transform.}
        \For{ \( \xi \in \textsc{Fourier Grid} \) }
        \Comment{\( \xi \in \mathbb{R}^d\).}
        \State Multiply \( \mathcal{F}(\eta)(\xi) \) by \( e^{-\frac{i}{2} \sin(2\Delta t) |\xi|^2} \).
        \EndFor
        \State Apply an inverse FFT to \( \mathcal{F}(\eta) \).
        \For{ \( x \in \textsc{Grid} \) }
        \State Multiply \( \eta(x) \) by \( e^{- \frac{i}{2} \tan(\Delta t) |x|^2} \).
        \EndFor
        \EndFor
    \end{algorithmic}
\end{algorithm}

\begin{algorithm}
    \caption{Spectral solver for \eqref{eqn: cNLS with psi}, with a Strang Splitting method.}
    \label{algo: num -- spectral solver}
    \begin{algorithmic}
        \State Discretize the initial data \( \eta \) on a \( \text{Grid} \subset \mathbb{R}^d \).
        \For{ Each timestep of size \( \Delta t \) }
        \State Use Algorithm \ref{algo: num -- spectral solver linear part} with a stepsize \( \Delta t/2 \).
        \For{ \( x \in \textsc{Grid} \) }
        \Comment{Add interactions.}
        \State Multiply \( \eta(x) \) by \( e^{-i\,\Delta t\, |\eta(x)|^2} \).
        \EndFor
        \State Use Algorithm \ref{algo: num -- spectral solver linear part} with a stepsize \( \Delta t/2 \).
        \EndFor
    \end{algorithmic}
\end{algorithm}

Of course, in pratical applications one is not able to define a grid over \( \mathbb{R}^d \).
Hence, Algorithms \ref{algo: num -- spectral solver linear part} and \ref{algo: num -- spectral solver} have to be modified by defining \textsc{Grid} as a discretization of a finite-volume subset of \( \mathbb{R}^d \), typically a product of intervals in each dimension.
For all of our numerical examples, this will \( [-15, 15]\times[-15, 15] \), discretized using \( N_x\times N_y \) points.
In order to have an easily computable FFT, one has to use a spatial uniform grid, which then defines the \textsc{Fourier Grid}.
Special care has to be paid when choosing the number of points: if we have Fourier frequencies larger than the \emph{Nyquist frequency}, then we will observe a phenomenon known as \emph{aliasing}.
This may not be problematic for the Harmonic Oscillator \eqref{eqn: cNLS with psi -- linear part} depending on the initial condition, but will eventually become an issue when simulating \eqref{eqn: cNLS with psi} because it involves interactions and hence an infinite number of frequencies.
Moreover, by using a FFT-based algorithm we implicitly impose periodic boundary conditions.

\subsection{Discretization into a sum of Bubbles}

We need to decompose any arbitrary function into a finite sum of \( N \) bubbles.
A solution to this question has been proposed in \cite{qianFastGaussianWavepacket2010}, but it involves integrals over the whole phase space, which is something we want to avoid.

We could also use a nonlinear least squares approach, but our experimental results showed that it tends to yield spread out Gaussians, which may present huge overlaps between them.
The overlaps cause issues with the DFMP, for instance a blow-up of the conservative quantities.
This has been observed during our experiments but the results are not reported in this paper.
The issue of discretizing an arbitrary function into a sum of bubbles without too much overlap is not the main concern of this paper, hence we will use a visual trial-and-error discretization.
Another possible way of discretizing the initial data is outlined in \cite{adamowiczLaserinducedDynamicAlignment2022}.
Finally, if we do not restrict ourselves to Gaussian functions and allow general Hermite functions, then the discretization simply consists in projecting the initial condition onto this basis, and truncating the highest modes if necessary.

\subsection{Observables}

In order to compare the bubbles scheme against the spectral method, we compare them in absence of interactions, i.e. on the Harmonic Oscillator \eqref{eqn: cNLS with psi -- linear part}, as well as in the presence of nonlinear interactions, i.e. on \eqref{eqn: cNLS with psi}.
We showed in Lemma \ref{lemma: conserved quantities in HO} the conservation of some quantities for \eqref{eqn: cNLS with psi -- linear part} and \eqref{eqn: cNLS with psi}, we will focus on mass, energy and momentum.
When computing the observables for the spectral solution, we noted that the approximation of the gradient using finite differences with periodic boundary conditions yielded very rough results while the gradient approximation using the Fast Fourier Transform gave more accurate results.
We use the latter approximation in the Figures of Section \ref{subsect: num results}.
In the case of bubbles, we compute every integral by hand thanks to the assumption \( v(s,y) = e^{-\frac{|y|^2}{2}} \), some details are given in Appendix \ref{appendix: Miscellaneous computations}.
When reporting the results in the following \( \log \)-plots, all values with an amplitude smaller than \( 10^{-16} \) were set to be equal to \( 10^{-16} \).

For all of the results shown, the spectral scheme is supplied with the exact initial condition and not a projection on the grid of the bubbles discretization.

\subsection{Results}
\label{subsect: num results}

We consider examples adapted from \cite{baoNumericalSolutionGross2003}.

\subsubsection{Test case 1: Zero phase initial data}

The initial condition reads
\begin{equation}
    \psi(t=0, x) = \pi e^{-\frac{|x-\mu_1|^2}{2} } + 2e^{-\frac{|x-\mu_2|^2}{2} }, \quad x\in \mathbb{R}^2,\quad \mu_1 = (0, 2), \ \mu_2 = (1, 0)
    .
\end{equation}

\begin{figure}
    \begin{subfigure}{\textwidth}
        \centering
        %
    \tikzsetnextfilename{bubblesVSspectral_conservative_quantities_dt-0.001_coeffs-1.0-0.0_testcase1_Nx128_Ny129}%
    \ifcustomcompileusingtikzexternalize
	\resizebox{1\textwidth}{!}{\input{\customTikzInputFolderbubblesVSspectral_conservative_quantities_dt-0.001_coeffs-1.0-0.0_testcase1_Nx128_Ny129.tex}}%
	\else
	\includegraphics[width=1\textwidth]{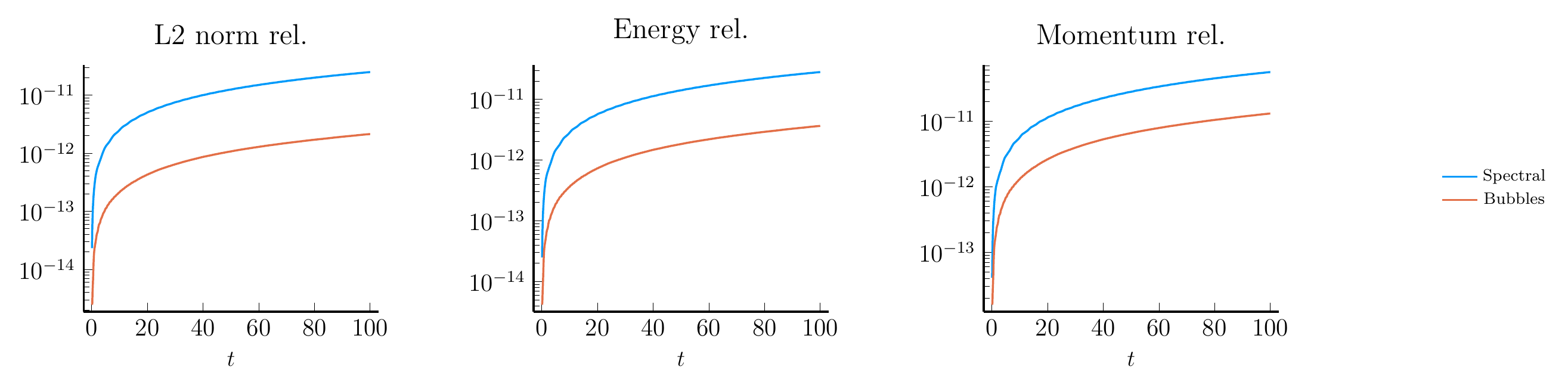}
	\fi

        \caption{\centering Approximate solution to the Harmonic Oscillator \eqref{eqn: cNLS with psi -- linear part}.}
    \end{subfigure}
    \begin{subfigure}{\textwidth}
        \centering
        %
    \tikzsetnextfilename{bubblesVSspectral_conservative_quantities_dt-0.001_coeffs-1.0-1.0_testcase1_Nx128_Ny129}%
    \ifcustomcompileusingtikzexternalize
	\resizebox{1\textwidth}{!}{\input{\customTikzInputFolderbubblesVSspectral_conservative_quantities_dt-0.001_coeffs-1.0-1.0_testcase1_Nx128_Ny129.tex}}%
	\else
	\includegraphics[width=1\textwidth]{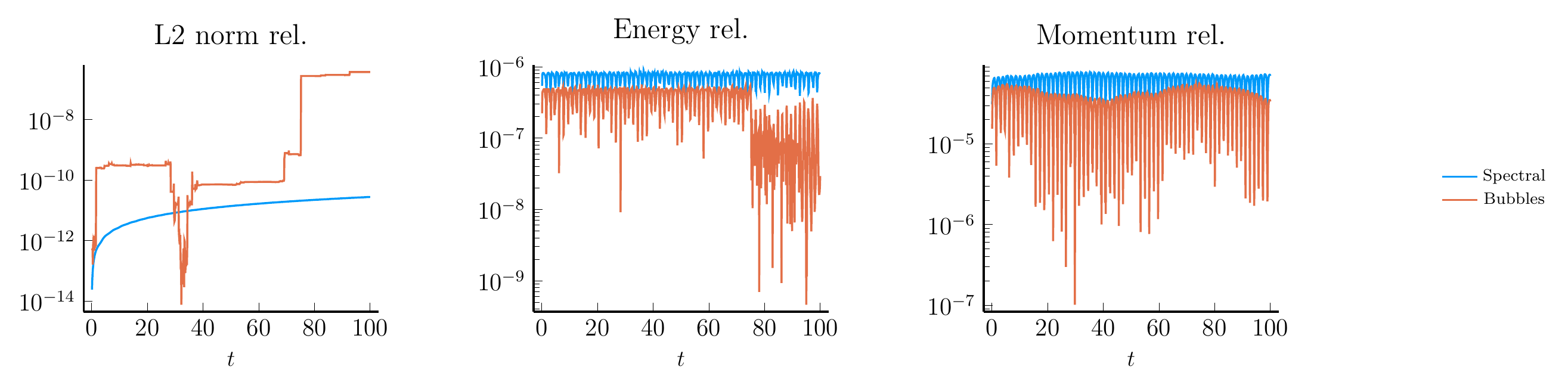}
	\fi

        \caption{\centering Approximate solution to the Schrödinger equation \eqref{eqn: cNLS with psi} using DFMP Algorithm.}
    \end{subfigure}
    \caption{\centering Test case 1.
        Relative evolution of mass, energy and momentum with bubbles and spectral methods.
        \( \Delta t = 10^{-3} \).
        Time-integrator for the nonlinear part of the splitting: Runge-Kutta of order 4.
        Spectral scheme with \( N_x = 128, N_y = 129 \).
    }
    \label{fig: num -- test case 1}
\end{figure}


The results are displayed in Figure \ref{fig: num -- test case 1}.
The solution approximated with the DFMP approach globally outperforms the spectral method on both the Harmonic Oscillator and the cubic NonLinear Schrödinger equations.
On the Harmonic Oscillator, the solution obtained with the Bubbles scheme is about one order of magnitude better than the spectral scheme.
When we compare them on \eqref{eqn: cNLS with psi}, i.e. when adding interactions, the \( \mathbb{L}^2 \) norm is better conserved for the spectral scheme, but the other conservative quantities are better conserved on a long time for the Bubbles scheme.

The ``jumps'' in the DFMP approach may be explained by an ill-conditioned Gram matrix, which would then yield a very rough approximation of the modulation parameters.

\subsubsection{Test case 2: Weak interactions}

The initial condition reads
\begin{equation}
    \psi(t=0, x) = e^{-|x - \mu_3|^2} e^{i \cosh |x - \mu_3|}, \quad x\in \mathbb{R}^2, \quad \mu_3 = (1, 1)
    .
\end{equation}

The approximation of this function as a sum of bubbles is pretty straightforward: we know that for \( x \) small, \( \cosh x \approx 1 + \frac{x^2}{2}  \), hence
\begin{equation*}
    \psi(t=0, x) \approx e^{-|x - \mu_3|^2} e^{i + i\frac{|x - \mu_3|^2}{2} }, \quad x\in \mathbb{R}^2
    .
\end{equation*}

\begin{figure}
    \begin{subfigure}{\textwidth}
        \centering
        %
    \tikzsetnextfilename{bubblesVSspectral_conservative_quantities_dt-0.001_coeffs-1.0-0.0_testcase2_Nx128_Ny129}%
    \ifcustomcompileusingtikzexternalize
	\resizebox{1\textwidth}{!}{\input{\customTikzInputFolderbubblesVSspectral_conservative_quantities_dt-0.001_coeffs-1.0-0.0_testcase2_Nx128_Ny129.tex}}%
	\else
	\includegraphics[width=1\textwidth]{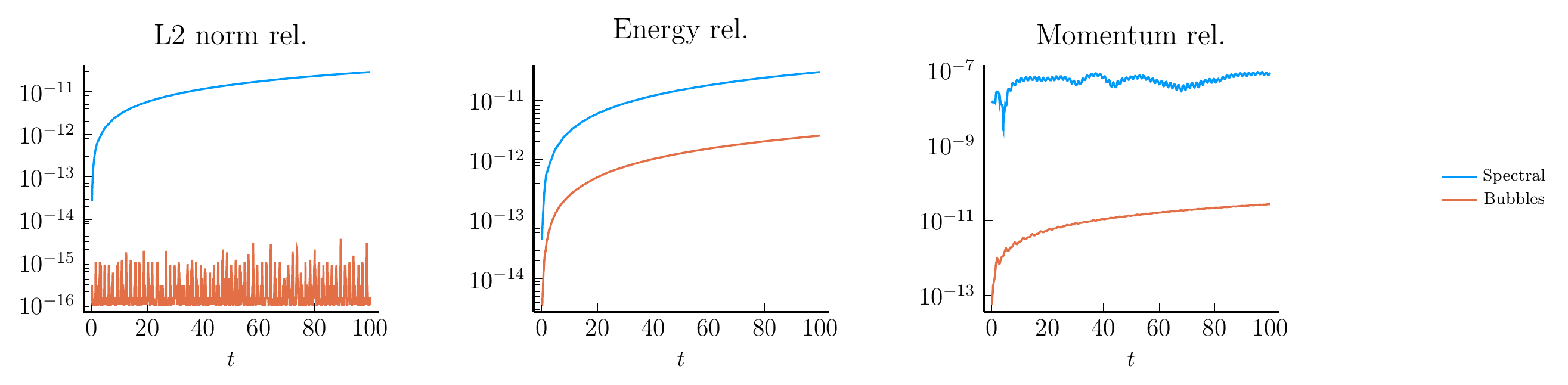}
	\fi

        \caption{\centering Approximate solution to the Harmonic Oscillator \eqref{eqn: cNLS with psi -- linear part}.}
    \end{subfigure}
    \begin{subfigure}{\textwidth}
        \centering
        %
    \tikzsetnextfilename{bubblesVSspectral_conservative_quantities_dt-0.001_coeffs-1.0-1.0_testcase2_Nx128_Ny129}%
    \ifcustomcompileusingtikzexternalize
	\resizebox{1\textwidth}{!}{\input{\customTikzInputFolderbubblesVSspectral_conservative_quantities_dt-0.001_coeffs-1.0-1.0_testcase2_Nx128_Ny129.tex}}%
	\else
	\includegraphics[width=1\textwidth]{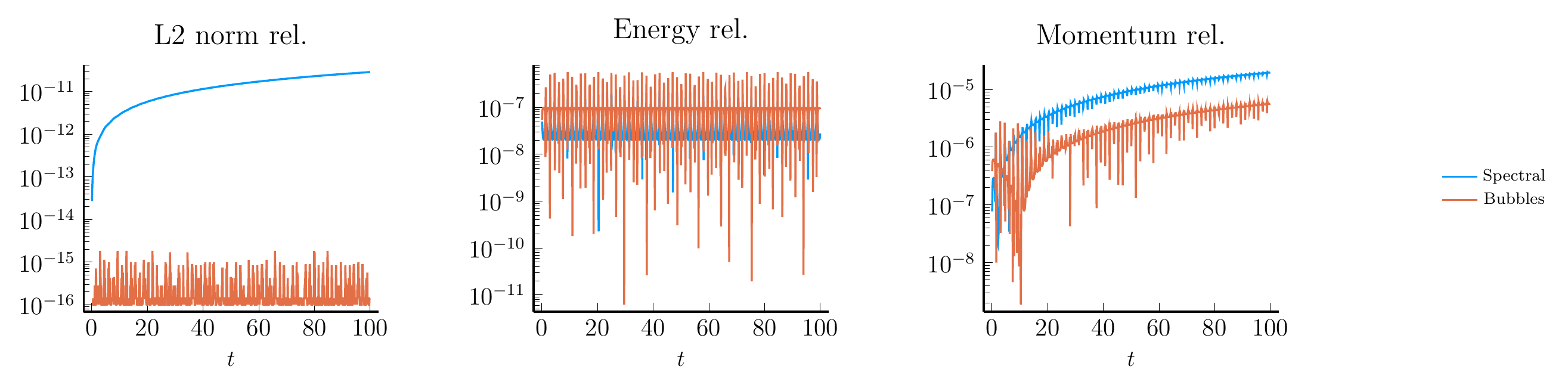}
	\fi

        \caption{\centering Approximate solution to the Schrödinger equation \eqref{eqn: cNLS with psi} using DFMP Algorithm.}
    \end{subfigure}
    \caption{\centering Test case 2.
        Relative evolution of mass, energy and momentum with bubbles and spectral methods.
        \( \Delta t = 10^{-3} \).
        Time-integrator for the nonlinear part of the splitting: Runge-Kutta of order 4.
        Spectral scheme with \( N_x = 128, N_y = 129 \).
    }
    \label{fig: num -- test case 2}
\end{figure}


The results are displayed in Figure \ref{fig: num -- test case 2}.
This example shows the performance of the DFMP approach in its most efficient setting: it only has one bubble.
This explains the very good conservation results obtained: the Bubbles scheme outperforms the spectral scheme on both \eqref{eqn: cNLS with psi -- linear part} and \eqref{eqn: cNLS with psi}, except for the energy on \eqref{eqn: cNLS with psi}.
However, even in this case, the error of the DFMP method remains generally less than one order of magnitude larger than the error from the spectral method.

\subsubsection{Test case 3: Strong interactions}

The initial condition reads
\begin{equation}
    \psi(t=0, x) =
    \begin{cases}
        \sqrt{M^2 - |x|^2} e^{i\cosh \sqrt{x_1^2 + x_2^2}}, & |x|^2 < M^2      \\
        0                                                   & \text{otherwise}
    \end{cases}
    , \quad M = 4.
\end{equation}

We apply the same approximation for the complex exponential as previously explained, and use a ``visual trial-and-error'' discretization of the square root.
It yields a number of \( N=9 \) bubbles.
We emphasize the fact that this discretization may be far from being the best one achievable, however the process of discretizing an arbitrary function into a sum of bubbles is not the main concern of this paper.
The discretization of the initial square root is given in Figure \ref{fig: num -- approximation of sqrt with bubbles}.

\begin{figure}
    \centering
    %
    \tikzsetnextfilename{approx_sqroot_with_bubbles}%
    \ifcustomcompileusingtikzexternalize
	\resizebox{0.7\textwidth}{!}{\input{\customTikzInputFolderapprox_sqroot_with_bubbles.tex}}%
	\else
	\includegraphics[width=0.7\textwidth]{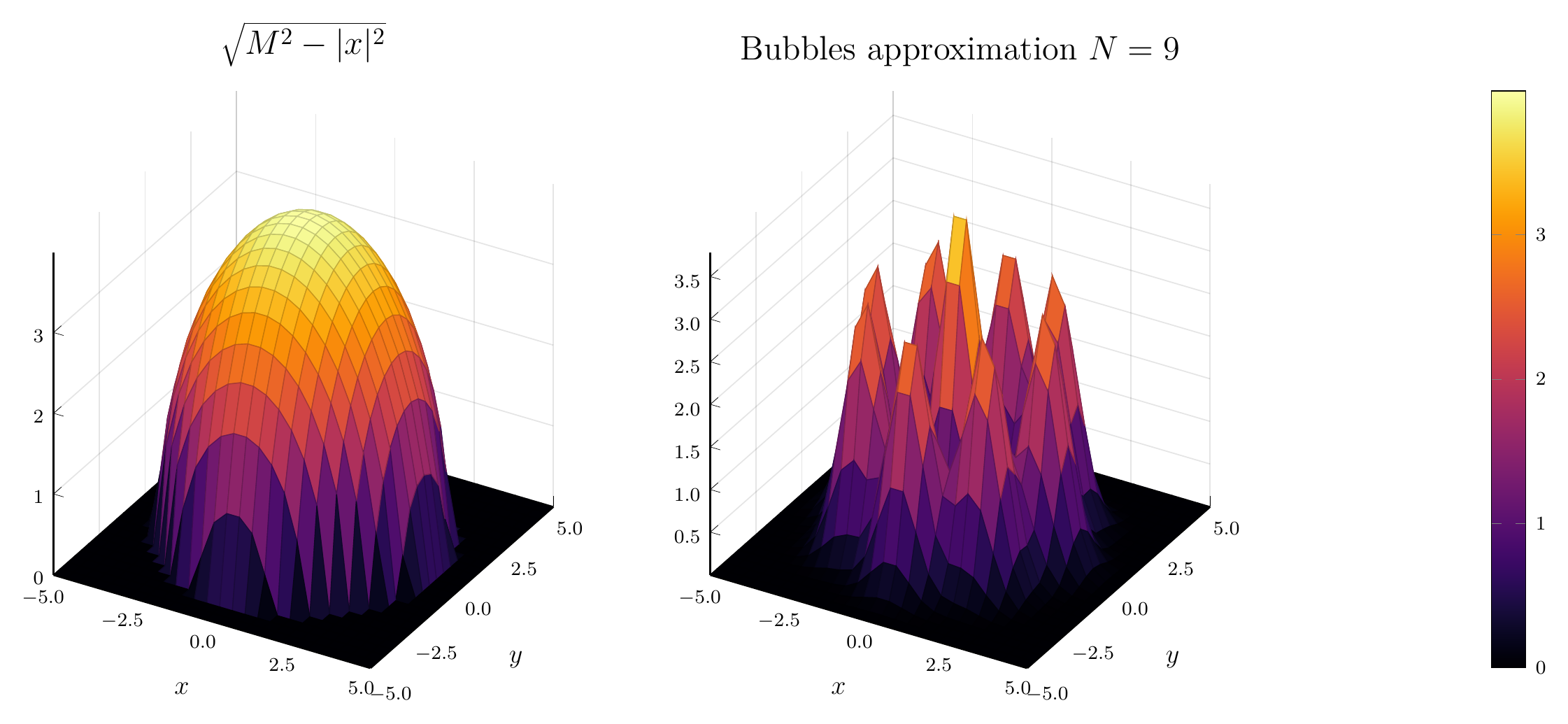}
	\fi

    \caption{\centering Approximation of \( x\mapsto \sqrt{M^2 - |x|^2} \) as a sum of bubbles}
    \label{fig: num -- approximation of sqrt with bubbles}
\end{figure}

\begin{figure}
    \begin{subfigure}{\textwidth}
        \centering
        %
    \tikzsetnextfilename{bubblesVSspectral_conservative_quantities_dt-0.001_coeffs-1.0-0.0_testcase3_Nx128_Ny129}%
    \ifcustomcompileusingtikzexternalize
	\resizebox{1\textwidth}{!}{\input{\customTikzInputFolderbubblesVSspectral_conservative_quantities_dt-0.001_coeffs-1.0-0.0_testcase3_Nx128_Ny129.tex}}%
	\else
	\includegraphics[width=1\textwidth]{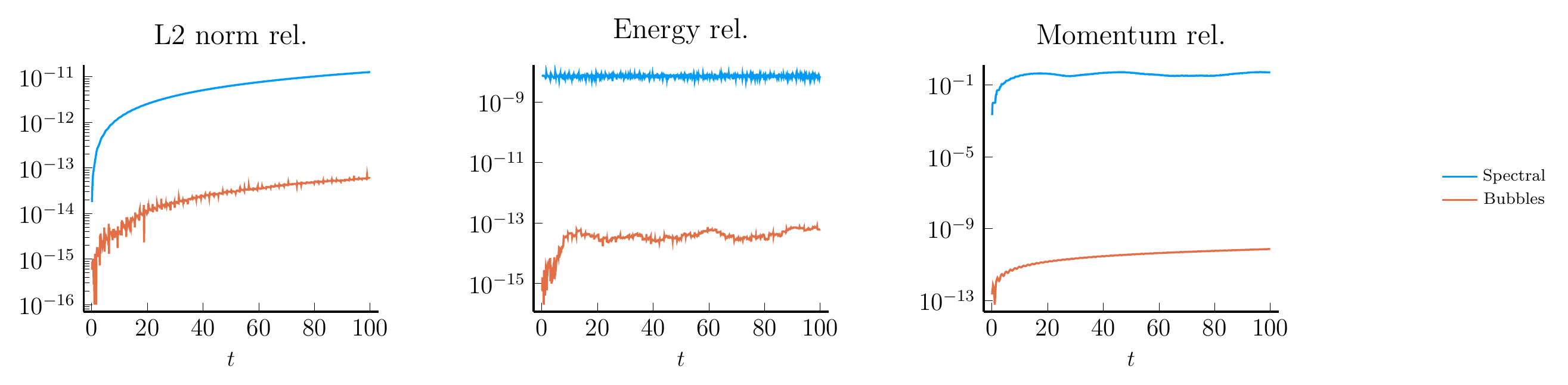}
	\fi

        \caption{\centering Approximate solution to the Harmonic Oscillator \eqref{eqn: cNLS with psi -- linear part}.}
    \end{subfigure}
    \begin{subfigure}{\textwidth}
        \centering
        %
    \tikzsetnextfilename{bubblesVSspectral_conservative_quantities_dt-0.001_coeffs-1.0-1.0_testcase3_Nx128_Ny129}%
    \ifcustomcompileusingtikzexternalize
	\resizebox{1\textwidth}{!}{\input{\customTikzInputFolderbubblesVSspectral_conservative_quantities_dt-0.001_coeffs-1.0-1.0_testcase3_Nx128_Ny129.tex}}%
	\else
	\includegraphics[width=1\textwidth]{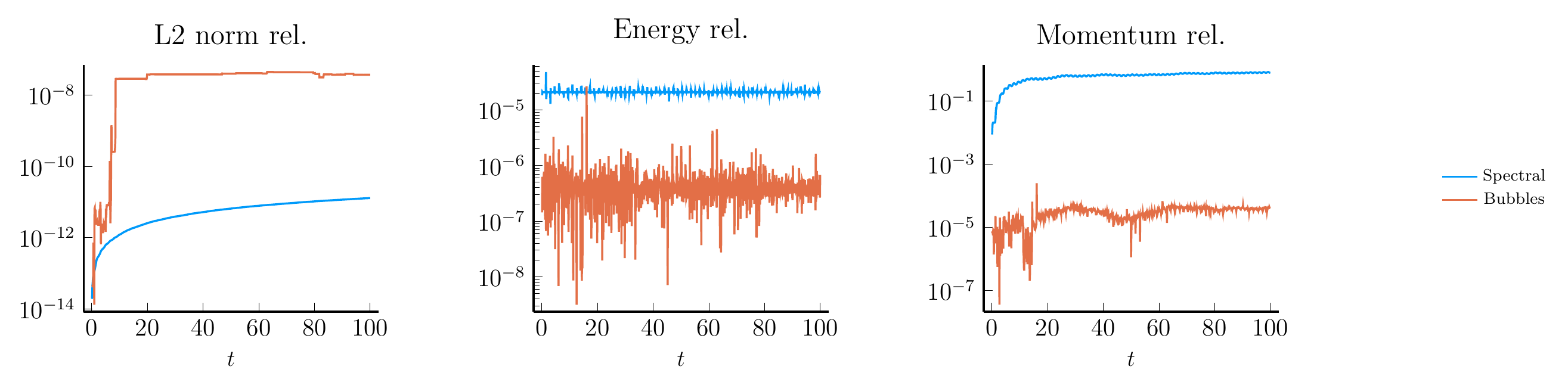}
	\fi

        \caption{\centering Approximate solution to the Schrödinger equation \eqref{eqn: cNLS with psi} using DFMP Algorithm.}
    \end{subfigure}
    \caption{\centering Test case 3.
        Relative evolution of mass, energy and momentum with bubbles and spectral methods.
        \( \Delta t = 10^{-3} \).
        Time-integrator for the nonlinear part of the splitting: Runge-Kutta of order 4.
        Spectral scheme with \( N_x = 128, N_y = 129 \).
    }
    \label{fig: num -- test case 3}
\end{figure}


The results are displayed in Figure \ref{fig: num -- test case 3}.
This example is by far the most interesting of the three test cases presented in this paper, because it shows that with the discretization given in Figure \ref{fig: num -- approximation of sqrt with bubbles} the conservation properties are pretty good for the Bubbles scheme, even when there are a lot of interactions between bubbles.
The spectral scheme is globally outperformed by the Bubbles scheme, except for the \( \mathbb{L}^2 \) norm in the presence of interactions, which is better conserved by the spectral scheme.
Even in this case, the conservation of this quantity with the DFMP is relatively correct.

The ``jumps'' in the relative evolution of conservative quantities may be explained by an ill-conditioned Gram matrix in DFMP.
It also has to be noted that if the discretization presents too much overlap between the Gaussian functions, then the DFMP approach fails and the conservative quantities blow up: this has been observed with other discretizations of the same initial data, and is not reported here.


\section{Conclusion}

We presented in this work an approach based on recent results from \cite{faouWeaklyTurbulentSolutions2020}.
It allows to solve exactly the Harmonic Oscillator \eqref{eqn: cNLS with psi -- linear part} for initial functions that can be represented as a sum of modulated functions (the \emph{bubbles}), for a certain kind of modulation.

In this context we focused on a particular subclass of such functions, modulated Hermite functions, which have the advantage of allowing explicit computations.
This is particularly interesting since we do not have to rely on any sort of discretization of the phase space, which is usually the main computational burden in numerical simulations.
We obtain an algorithm which yields an exact solution as soon as the initial data is a sum of modulated Hermite functions.
If we consider an arbitrary initial function, it suffices to project it into onto the Hermite basis and to perform analytical time-evolution.
Moreover, the algorithm only relies on a small number of parameters whose time-evolution is explicit, making it very fast and computationally efficient.

We also extended the results from \cite{faouWeaklyTurbulentSolutions2020} by allowing cubic interactions, at the cost of approximating the solution to \eqref{eqn: cNLS with psi -- nonlinear part} via the Dirac-Frenkel-MacLachlan principle.
We then only considered modulated Gaussian functions, because they allowed us to easily perform explicit computations and to obtain a numerical algorithm whose computational complexity is \( \mathcal{O}(N^4d + N^3 d^3) \).
Here \( d \) is the dimension and \( N \) is the number of bubbles.
The most critical parameter is \( N \), which corresponds roughly to the precision of the discretization when considering arbitrary initial data.
For any given function, the higher \( N \), the better we can approximate it as a sum of modulated Gaussian functions.
We then have a clear trade-off between the speed of the algorithm and the precision of the discretization.
The bubble algorithm globally outperforms on the numerical test cases a spectral method combined with time splitting, where each splitting step is solved exactly.

As a final note, any grid-based method is inherently bound to a finite subset of \( \mathbb{R}^d \) to which we have to add boundary conditions, while the bubble approach does not have such restrictions. We emphasize the fact that the algorithm presented extends in a straightforward manner when dealing with complex modulated Hermite functions.

\section*{Acknowledgements}
\addcontentsline{toc}{section}{Acknowledgements}
The authors would like to thank J.
Bernier for spotting an abnormal behavior in the spectral method and for his help in fixing it.
EF and YLH were sponsored by the Centre Henri Lebesgue, program ANR-11-LABX-0020-01.
EF and PR were sponsored by the Inria Associated team {\em Bubbles}.

\begin{appendices}


    \section{Proof of Lemma \ref{lemma: conserved quantities in HO}}
    \label{appendix: conserved}

    We will need the following result, known as the Pohozaev identity.
    \begin{lemma}[Pohozaev Identity]
        \label{lemma: pohozaev identity}
        Let \( x\in \mathbb{R}^d \), and \( f\in H^1( \mathbb{R}^d ) \) such that \( xf\in \mathbb{L}^2( \mathbb{R}^d ) \).
        Then
        \begin{equation}
            \label{eqn: pohozaev identity}
            \int \Delta f \overline{\left( \frac{d}{2} f + x\cdot \nabla f \right)} dx = - \int |\nabla f|^2 dx
            .
        \end{equation}
    \end{lemma}

    \begin{proof}
        By density, we only need to prove equation \eqref{eqn: pohozaev identity} for \( f \in \mathcal{C}^\infty_c( \mathbb{R}^d ) \), where \( \mathcal{C}^\infty_c ( \mathbb{R}^d ) \) denotes the space of infinitely smooth functions with compact support in \( \mathbb{R}^d \).
        Let
        \begin{equation*}
            f_\lambda(x) := \lambda^{\frac{d}{2} } f(\lambda x)
            ,
        \end{equation*}
        then
        \begin{equation*}
            \int |\nabla f_\lambda |^2 dx = \lambda^2 \int |\nabla f|^2 dx
            .
        \end{equation*}
        Differentiating this identity with respect to \( \lambda \) and evaluating the result at \( \lambda = 1 \) yields
        \begin{equation*}
            \int \nabla f \cdot \overline{ \nabla \left( \frac{d}{2} f + x \cdot \nabla f \right) } dx = \int |\nabla f|^2 dx
            .
        \end{equation*}
        We integrate by parts the LHS, and obtain \eqref{eqn: pohozaev identity}.
    \end{proof}

    \vspace{1em}

    Now turn to the proof of Lemma \ref{lemma: conserved quantities in HO}.   

    Mass conservation:
    \begin{align*}
        \partial_{t} \|\psi\|^2_{\mathbb{L}^2}
         & = \partial_{t} \int |\psi|^2 = 2 \Re \int \bar{\psi} \partial_{t} \psi = 2 \Re \int -i \bar{\psi} \left( -\mu\Delta \psi + \mu|x|^2 \psi + \lambda |\psi|^2 \psi \right) \\
         & = 2 \Re \int -i \left( \mu\bar{\psi} \Delta \psi + \mu|x|^2 |\psi|^2 + \lambda |\psi|^4  \right) = 2\mu\Re \int -i \bar{\psi} \Delta \psi                                \\
         & = 2 \mu \Re \int i |\nabla \psi|^2 = 0
        .
    \end{align*}
    Energy conservation:
    \begin{align*}
        \frac{\textd}{\textd t} E_{\mu,\lambda}
         & = \frac{1}{2} \frac{\textd}{\textd t} \left\langle -\mu\Delta \psi + \mu|x|^2 \psi + \frac{\lambda}{2} |\psi|^2 \psi, \psi \right\rangle                                                                                         \\
         & = \frac{1}{2} \left( -2\mu\Re\langle \Delta \psi, \partial_{t} \psi \rangle + 2\mu\Re \langle |x|^2 \psi, \partial_{t} \psi \rangle + 4 \Re\left\langle \frac{\lambda}{2} |\psi|^2 \psi, \partial_{t} \psi \right\rangle \right) \\
         & = \Re \left(i \langle \partial_{t} \psi, \partial_{t} \psi \rangle\right) = 0
        .
    \end{align*}
    For the momentum, we compute
    \begin{align}
        \frac{1}{2} \frac{\textd}{\textd t} \int |x|^2 |\psi|^2
         & = \frac{1}{2} \frac{\textd}{\textd t} \langle |x|^2 \psi, \psi \rangle = \Re \langle |x|^2 \psi, \partial_{t} \psi \rangle = \Re \langle |x|^2 \psi, i \mu\Delta \psi - i \mu|x|^2 \psi - i\lambda|\psi|^2 \psi \rangle \nonumber \\
         & = \mu\Im \langle |x|^2 \psi, \Delta \psi \rangle = \mu\Im \int |x|^2 \psi \Delta \bar{\psi} = -\mu\Im \int \nabla \bar{\psi} \cdot \nabla \left( |x|^2 \psi \right) \nonumber                                                     \\
         & = -\mu\Im \int \nabla \bar{\psi} \cdot 2x \psi - \mu\Im \int \nabla \bar{\psi} \cdot \nabla \psi |x|^2 = -2\mu\Im \int x\cdot \nabla \bar{\psi} \psi \nonumber                                                                    \\
         & = 2\mu\Im \int x\cdot \nabla \psi \bar{\psi} \label{eqn: lemma conservation -- intermediate equation no 1}
        ,
    \end{align}
    and
    \begin{equation*}
        \frac{1}{2} \frac{\textd}{\textd t} \Im \int x\cdot \nabla \psi \bar{\psi}
        = \frac{1}{2} \Im \int \left( x\cdot \nabla \partial_{t} \psi \bar{\psi} + x\cdot \nabla \psi \partial_{t} \bar{\psi} \right)
        .
    \end{equation*}
    An integration by parts gives
    \begin{equation*}
        \int x\cdot \nabla \phi \psi = - \int \phi \nabla \cdot (x \psi) = -\int \phi \left( \psi d + x \cdot \nabla \psi \right)
        ,
    \end{equation*}
    hence
    \begin{align*}
        \frac{1}{2} \frac{\textd}{\textd t} \Im \int x\cdot \nabla \psi \bar{\psi}
         & = \frac{1}{2} \Im \int \left( - \partial_{t} \psi \left( \bar{\psi} d + x \cdot \nabla \bar{\psi} \right) + x\cdot \nabla \psi \partial_{t} \bar{\psi} \right)   \\
         & = \frac{1}{2} \Im \int \left( - \partial_{t} \psi \bar{\psi} d - \partial_{t} \psi x\cdot \nabla \bar{\psi} + \partial_{t} \bar{\psi} x\cdot \nabla \psi \right) \\
         & = \frac{1}{2} \Im \int \left( - \partial_{t} \psi \bar{\psi} d + 2 i \Im \left[ \partial_{t} \bar{\psi} x\cdot \nabla \psi \right] \right)                       \\
         & = -\frac{d}{2} \Im \int \partial_{t} \psi \bar{\psi} + \Im \int \partial_{t} \bar{\psi} x\cdot \nabla \psi
        .
    \end{align*}
    Recall the equation satisfied by \( \psi \):
    \begin{equation*}
        \partial_{t} \psi = i\mu \Delta \psi - i\mu |x|^2 \psi - i \lambda |\psi|^2 \psi
        ,
    \end{equation*}
    therefore
    \begin{align*}
        \frac{1}{2} \frac{\textd}{\textd t} \Im \int x\cdot \nabla \psi \bar{\psi}
         & = -\frac{d}{2} \Im \int i\left[ \mu\Delta \psi - \mu|x|^2 \psi - \lambda |\psi|^2 \psi \right] \bar{\psi}                       \\
         & \qquad + \Im \int i\left[ - \mu\Delta \bar{\psi} + \mu|x|^2 \bar{\psi} + \lambda |\psi|^2 \bar{\psi} \right] x\cdot \nabla \psi
        .
    \end{align*}
    We have
    \begin{equation*}
        -\frac{d}{2} \Im \int i\left[ \mu\Delta \psi - \mu|x|^2 \psi - \lambda |\psi|^2 \psi \right] \bar{\psi}
        = \frac{d}{2} \int \left[ \mu|\nabla \psi|^2 + \mu |x|^2 |\psi|^2 + \lambda |\psi|^4 \right]
        ,
    \end{equation*}
    and
    \begin{equation*}
        \Im \int i\left[ - \mu\Delta \bar{\psi} + \mu|x|^2 \bar{\psi} + \lambda |\psi|^2 \bar{\psi} \right] x\cdot \nabla \psi = \Re \int \left[ - \mu \Delta \bar{\psi} + \mu |x|^2 \bar{\psi} + \lambda |\psi|^2 \bar{\psi} \right] x\cdot \nabla \psi
        .
    \end{equation*}
    Moreover,
    \begin{align*}
        \int |x|^2 \bar{\psi} x\cdot \nabla \psi
        = -\int \psi \nabla \cdot \left( x|x|^2 \bar{\psi} \right)                                          & = - \int \psi \left( d |x|^2 \bar{\psi} + 2|x|^2 \bar{\psi} + x|x|^2 \cdot \nabla \bar{\psi} \right) \\
        \iff \int |x|^2 \bar{\psi} x\cdot \nabla \psi + \overline{\int |x|^2 \bar{\psi} x\cdot \nabla \psi} & = - \int \psi \left( d |x|^2 \bar{\psi} + 2|x|^2 \bar{\psi} \right)                                  \\
        \iff \Re \int |x|^2 \bar{\psi} x\cdot \nabla \psi                                                   & = - \int \psi \left( \frac{d}{2} |x|^2 \bar{\psi} + |x|^2 \bar{\psi} \right)
        .
    \end{align*}
    Finally,
    \begin{align*}
        \frac{1}{2} \frac{\textd}{\textd t} \Im \int x\cdot \nabla \psi \bar{\psi}
         & = \frac{d}{2} \int \left[ \mu |\nabla \psi|^2 + \mu |x|^2 |\psi|^2 + \lambda |\psi|^4 \right]                                                                                                           \\
         & \qquad + \Re \int \left[ - \mu \Delta \bar{\psi} + \lambda |\psi|^2 \bar{\psi} \right] x\cdot \nabla \psi - \mu \int \psi \left( \frac{d}{2} |x|^2 \bar{\psi} + |x|^2 \bar{\psi} \right)                \\
         & = \frac{d}{2} \int \left[ \mu |\nabla \psi|^2 + \lambda |\psi|^4 \right] + \Re \int \left[ -\mu \Delta \bar{\psi} + \lambda |\psi|^2 \bar{\psi} \right] x\cdot \nabla \psi - \mu \int |x|^2 |\psi|^2    \\
         & = \frac{d}{2}\mu  \int |\nabla \psi|^2 - \mu \int |x|^2 |\psi|^2 + \frac{d}{2} \lambda \int |\psi|^4 + \Re \int \left[ - \mu \Delta \bar{\psi} + \lambda |\psi|^2 \bar{\psi} \right] x\cdot \nabla \psi
    \end{align*}
    We are in the two-dimensional case \( d=2 \), hence
    \begin{align*}
        \frac{1}{2} \frac{\textd}{\textd t} \Im \int x\cdot \nabla \psi \bar{\psi}
         & = \int \mu |\nabla \psi|^2 - \mu \int |x|^2 |\psi|^2 + \lambda \int |\psi|^4 + \Re \int \left[ - \mu \Delta \bar{\psi} + \lambda |\psi|^2 \bar{\psi} \right] x\cdot \nabla \psi \\
         & = 2 E_\lambda + \frac{\lambda}{2} \int |\psi|^4  - 2\mu \int |x|^2 |\psi|^2 + \Re \int \left[ - \mu \Delta \bar{\psi} + \lambda |\psi|^2 \bar{\psi} \right] x\cdot \nabla \psi
        .
    \end{align*}
    Moreover,
    \begin{align*}
        \int |\psi|^2 \bar{\psi} x\cdot \nabla \psi
                                                               & = - \int \psi \nabla \cdot \left( |\psi|^2 \bar{\psi} x \right)                                                                                      \\
                                                               & = - \int \psi \left( 2\Re \left( \bar{\psi} \nabla \psi \right) \cdot \bar{\psi} x + |\psi|^2 \nabla\bar{\psi}\cdot x + d|\psi|^2 \bar{\psi} \right) \\
                                                               & = - \int \left( 2\Re \left( \bar{\psi} \nabla \psi \right) \cdot |\psi|^2 x + \psi |\psi|^2 \nabla\bar{\psi}\cdot x + 2|\psi|^4 \right)              \\
        \iff 2 \Re \int |\psi|^2 \bar{\psi} x\cdot \nabla \psi & = - 2\Re \int \bar{\psi} \nabla \psi \cdot |\psi|^2 x - 2 \int |\psi|^4                                                                              \\
        \iff \Re \int |\psi|^2 \bar{\psi} x\cdot \nabla \psi   & = - \frac{1}{2} \int |\psi|^4
        ,
    \end{align*}
    Finally,
    \begin{align*}
        \frac{1}{2} \frac{\textd}{\textd t} \Im \int x\cdot \nabla \psi \bar{\psi}
         & = 2 E_\lambda + \frac{\lambda}{2} \int |\psi|^4  - 2\mu \int |x|^2 |\psi|^2 - \mu \Re \int \Delta \bar{\psi} x\cdot \nabla \psi - \frac{\lambda}{2} \int |\psi|^4 \\
         & = 2 E_\lambda - 2\mu \int |x|^2 |\psi|^2 - \mu \Re \int \Delta \bar{\psi} x\cdot \nabla \psi
    \end{align*}
    We then use the Pohozaev identity \eqref{eqn: pohozaev identity} in dimension \( d=2 \), which yields
    \begin{equation*}
        \Re\left( \int x\cdot \nabla \psi \Delta \bar{\psi} \right) = 0
        .
    \end{equation*}
    Therefore,
    \begin{equation*}
        \frac{1}{2} \frac{\textd}{\textd t} \Im \int x\cdot \nabla \psi \bar{\psi}  = 2 E_{\mu,\lambda} - 2\mu \int |x|^2 |\psi|^2
        .
    \end{equation*}

    From the conservation of the energy \( E_\lambda \) and equation \eqref{eqn: lemma conservation -- intermediate equation no 1},
    \begin{equation*}
        \frac{\textd^2}{\textd t^2} \Im \int x\cdot \nabla \psi \bar{\psi} = -16\mu^2 \Im \int x\cdot \nabla \psi \bar{\psi}
        .
    \end{equation*}
    Hence, the conservation laws
    \begin{align*}
         & \frac{1}{16} \left( \frac{\textd}{\textd t} \left[ \Im \int x\cdot \nabla \psi \bar{\psi} \right] \right)^2 + \mu^2\left( \Im \int x\cdot \nabla \psi \bar{\psi}  \right)^2 \\
         & = \left( E_{\mu,\lambda} - \mu \|x\psi\|^2_{ \mathbb{L}^2 } \right)^2 + \mu^2 \left( \Im \int x\cdot \nabla \psi \bar{\psi} \right)^2
    \end{align*}

    For the non radial conservation law:
    \begin{equation*}
        \frac{\textd}{\textd t} \Im \int \partial_{j} \psi \bar{\psi}
        = -2\Im \int \partial_{t} \psi \overline{ \partial_{j} \psi }
        = 2\Re \int i \partial_{t} \psi \overline{ \partial_{j} \psi }
        = 2 \mu \int |x|^2 \Re \left(\psi \overline{ \partial_{j} \psi }\right)
        = -2\mu \int x_j |\psi|^2
        ,
    \end{equation*}
    owing to the facts that integrations by parts yield
    \begin{equation*}
        -\Re \int \Delta \psi \partial_{j} \bar{\psi} = \Re \int \partial_{j} \psi \Delta \bar{\psi}
        ,
    \end{equation*}
    and
    \begin{equation*}
        2 \Re \int |\psi|^2 \psi \partial_{j} \bar{\psi}
        = \int |\psi|^2 \partial_{j} |\psi|^2 = - \int |\psi|^2 \partial_{j} |\psi|^2  \implies \Re \int |\psi|^2 \psi \partial_{j} \bar{\psi} = 0
        .
    \end{equation*}
    We also have
    \begin{equation*}
        \frac{1}{2} \frac{\textd}{\textd t} \int x_j |\psi|^2
        = \Re \int x_j \partial_{t} \psi \bar{\psi}
        = \Im \int x_j i \partial_{t} \psi \bar{\psi} = \mu \Im \int -\Delta \psi x_j \bar{\psi}
        = \mu \Im \int \partial_{j} \psi \bar{\psi}
        .
    \end{equation*}
    Hence the relations
    \begin{equation*}
        \left|
        \begin{aligned}
            \frac{\textd}{\textd t} \int x_j |\psi|^2                     & = 2\mu \Im \int \partial_{j} \psi \bar{\psi} \\
            \frac{\textd}{\textd t} \Im \int \partial_{j} \psi \bar{\psi} & = -2\mu \int x_j |\psi|^2,
        \end{aligned}
        \right.
    \end{equation*}
    which have the conservation law
    \begin{equation*}
        \mathcal{P}_j = \frac{1}{4} \left( \Im \int \partial_{j} \psi \bar{\psi} \right)^2 + \mu^2 \left( \int x_j |\psi|^2 \right)^2
        .
    \end{equation*}

    \section{Proof of Lemma \ref{lemma: HO exact integration of parameters}}
    \label{appendix: proof of lemma exact integration of parameters}

    For the \( (X, \beta) \) part, it suffices to check that the change of variable $(X,\beta) \mapsto (a,\theta)$ defined by
    \begin{equation*}
        X_i = \sqrt{2a_i} \sin(\theta_i) \quad \text{ and } \quad
        \beta_i = \sqrt{2a_i} \cos(\theta_i)\qquad i=1,\dots, d
        ,
    \end{equation*}
    is symplectic and that
    \begin{equation*}
        X_i = \sqrt{2a_i(0)} \sin(2t + \theta_i(0)) \quad \text{ and } \quad \beta_i = \sqrt{2a_i(0)} \cos(2t + \theta_i(0))\qquad i=1,\dots, d ,
    \end{equation*}
    are solutions.

    For the \( (k, B) \) part we use the method of generating functions, described e.g. in \cite[Sect.~VI.5]{hairerGeometricNumericalIntegration2006}.
    We can express \( B \) in terms of \( k \) and the Hamiltonian \( \mathcal{E} \), so that on the set \( \{B > 0\} \) we have:
    \begin{equation}
        \label{eqn: proof HO -- definition of B in terms of CAL E and k}
        B = 2\sqrt{e^{4k} \mathcal{E} - e^{8k} - 1} .
    \end{equation}
    This equality holds for \( e^{4k} \in [e^{4k_0}, e^{4k_1}] \), where \( e^{4k_0}, e^{4k_1} \) are the reals roots of the polynomial \( -z^2 + \mathcal{E} z - 1 \),
    \begin{equation}
        \label{eqn: HO proof -- action angle variable - definition of k0 and k1}
        e^{4k_0} = \frac{1}{2} \left( \mathcal{E} - \sqrt{ \mathcal{E}^2 - 4 } \right),\quad e^{4k_1} = \frac{1}{2} \left( \mathcal{E} + \sqrt{ \mathcal{E}^2 - 4 } \right)
        .
    \end{equation}

    In order to obtain a symplectic change of variables, we look for a function \( S(k, \mathcal{E}) \) such that
    \begin{equation*}
        B = \frac{\partial S}{\partial k} (k, \mathcal{E}) .
    \end{equation*}
    We easily obtain \( S(k, \mathcal{E}) \), by integrating on \( [k_0, k] \):
    \begin{equation*}
        S(k, \mathcal{E}) = 2 \int_{k_0}^k \sqrt{e^{4z} \mathcal{E} - e^{8z} - 1} dz .
    \end{equation*}
    The variable \( \phi \) which makes the mapping \( (B, k) \mapsto (\phi, \mathcal{E}) \) symplectic is defined by
    \begin{equation*}
        \phi = \frac{\partial S}{\partial \mathcal{E}}(k, \mathcal{E}) = \int_{k_0}^k \frac{e^{4z}}{\sqrt{e^{4z} \mathcal{E} - e^{8z} - 1}} dz
        .
    \end{equation*}

    We have
    \begin{equation*}
        \frac{\textd \phi}{\textd t} = \frac{e^{4k} k_t}{\sqrt{e^{4k} - e^{8k} - 1}} = \frac{- e^{4k} \partial_{B} \mathcal{E}}{\frac{B}{2} } =  \frac{- e^{-4k} \frac{B}{2} e^{4k}}{\frac{B}{2} } = -1
        .
    \end{equation*}
    We now proceed to obtaining an explicit expression for \( \psi \):
    \begin{align*}
        \phi & = \int_{k_0}^k \frac{e^{4z}}{\sqrt{e^{4z} \mathcal{E} - e^{8z} - 1}} dz = \frac{1}{4} \int_{e^{4k_0}}^{e^{4k}} \frac{1}{\sqrt{ \mathcal{E}u - u^2 - 1 }} du                                                                   \\
             & = \frac{1}{4 \sqrt{\frac{ \mathcal{E}^2 }{4} - 1 }}  \int_{e^{4k_0}}^{e^{4k}} \frac{1}{\sqrt{ 1 - \left( \frac{u - \frac{ \mathcal{E} }{2} }{\sqrt{ \frac{ \mathcal{E}^2 }{4} - 1  }}  \right)^2 }} du                        \\
             & = \frac{1}{4} \int_{ \frac{e^{4k_0} - \frac{ \mathcal{E} }{2} }{\sqrt{ \frac{ \mathcal{E}^2 }{4} - 1 }}  }^{ \frac{e^{4k} - \frac{ \mathcal{E} }{2} }{\sqrt{ \frac{ \mathcal{E}^2 }{4} - 1 }}  } \frac{1}{\sqrt{1 - u^2}} du.
    \end{align*}
    Recall the definition \eqref{eqn: HO proof -- action angle variable - definition of k0 and k1} of \( k_0 \), which yields
    \begin{equation*}
        e^{4k_0} - \frac{ \mathcal{E} }{2} = - \sqrt{ \frac{\mathcal{E}^2}{4} - 1 }
        .
    \end{equation*}
    Therefore,
    \begin{align*}
        \phi & = \frac{1}{4} \int_{-1}^{ \frac{e^{4k} - \frac{ \mathcal{E} }{2} }{\sqrt{ \frac{ \mathcal{E}^2 }{4} - 1 }}  } \frac{1}{\sqrt{1 - u^2}} du = \frac{1}{4} \left( \arcsin\left( { \frac{e^{4k} - \frac{ \mathcal{E} }{2} }{\sqrt{ \frac{ \mathcal{E}^2 }{4} - 1 }}  }  \right) + \frac{\pi}{2}  \right) \nonumber \\
             & = \frac{1}{4} \arcsin\left( { \frac{e^{4k} - \frac{ \mathcal{E} }{2} }{\sqrt{ \frac{ \mathcal{E}^2 }{4} - 1 }}  }  \right) + \frac{\pi}{8} \in \left[ 0, \frac{\pi}{4}  \right]
        .
    \end{align*}
    We want the angle variable to lie in \( \left[ 0, 2\pi \right] \) so the above expression describes an eighth of a period.
    But we are only considering the set \( \{B>0\} \), thus the angle \( \xi \) we are looking for must lie only in \( [0, \pi] \).
    Hence we set \( (\xi, h) = (4\phi, \mathcal{E}/4) \) and let the Hamiltonian \( \mathcal{E}(\xi, h) = 4h \) with a slight abuse of notation.
    It is then clear that \( \frac{\textd h}{\textd t} = 0 \) and \( \frac{\textd \xi}{\textd t} = -4 \).
    Moreover,
    \begin{equation}
        \label{proof: HO proof -- definition of xi with k}
        \xi = \arcsin\left( \frac{e^{4k} - \frac{ \mathcal{E} }{2} }{\sqrt{ \frac{ \mathcal{E}^2 }{4} - 1 }} \right) + \frac{\pi}{2}
        \in \left[ 0, \pi  \right]
        ,
    \end{equation}
    and hence
    \begin{equation*}
        \frac{e^{4k} - \frac{ \mathcal{E} }{2} }{\sqrt{ \frac{ \mathcal{E}^2 }{4} - 1 }} = \sin\left( \xi - \frac{\pi}{2} \right)
        = - \cos\left( \xi \right)
        .
    \end{equation*}
    We obtain
    \begin{align*}
        e^{4k} = L^2 & = \frac{ \mathcal{E}}{2} - \cos(\xi) \sqrt{ \frac{ \mathcal{E}^2 }{4} - 1 } = 2h - \cos(\xi) \sqrt{4 h^2 - 1} \\
                     & = 2h \left( 1 - \cos(\xi) \sqrt{1 - \frac{1}{4h^2}} \right)
        .
    \end{align*}

    With this formula, we have
    \begin{equation*}
        0 < L^2 < 4h = \mathcal{E}
        ,
    \end{equation*}
    and \eqref{eqn: proof HO -- definition of B in terms of CAL E and k} becomes
    \begin{align*}
        B & = 2\sqrt{ \mathcal{E}e^{4k} - e^{8k} - 1} = 2\sqrt{4he^{4k} - (e^{4k})^2 - 1} = 2\sqrt{ (4h^2 - 1)\sin^2(\xi) } \\
          & = 2 \sin(\xi) \sqrt{4h^2 - 1},
    \end{align*}
    where the last equality holds for \( \xi \in [0, \pi] \).

    We can now integrate the equations for \( A \), and $\gamma$.
    The first one is
    \begin{equation*}
        A_t = \frac{AB}{2} (d-2) e^{-4k} .
    \end{equation*}
    From the expressions we just obtained we get
    \begin{equation*}
        A_t = A (d-2) \frac{\sin(\xi) \sqrt{4h^2 - 1}}{2h - \cos(\xi) \sqrt{ 4 h^2 - 1 }} .
    \end{equation*}
    The solution to this equation is of the form
    \begin{equation*}
        A(t) = A(0) \exp\left\{(d-2) \int_{0}^t \frac{\sin(\xi(\sigma)) \sqrt{4h(\sigma)^2 - 1}}{2h(\sigma) - \cos(\xi(\sigma)) \sqrt{ 4 h(\sigma)^2 - 1 }} \textd \sigma\right\}
        .
    \end{equation*}
    Moreover, we know that \( \sigma\mapsto h(\sigma) \) is constant, and that \( \xi(\sigma) = \xi(0) - 4\sigma \).
    Hence we have to solve
    \begin{equation*}
        A(t) = A(0) \exp\left\{(d-2) \int_{0}^t \frac{\sin(\xi(0) - 4\sigma) \sqrt{4h(0)^2 - 1}}{2h(0) - \cos(\xi(0) - 4\sigma) \sqrt{ 4 h(0)^2 - 1 }} \textd \sigma\right\}
        .
    \end{equation*}
    One can easily check that we have the following equality:
    \begin{align*}
         & \int_{0}^t \frac{\sin(\xi(0) - 4\sigma) \sqrt{4h(0)^2 - 1}}{2h(0) - \cos(\xi(0) - 4\sigma) \sqrt{ 4 h(0)^2 - 1 }} \textd \sigma                        \\
         & = -\frac{1}{4} \left[ \log\left( 2h(t) - \cos(\xi(t)) \sqrt{4h(t)^2 - 1} \right) - \log\left( 2h(0) - \cos(\xi(0)) \sqrt{4h(0)^2 - 1} \right) \right].
    \end{align*}
    Note that, unless \( h(0) = \frac{1}{2} \) or \( h(t) = \frac{1}{2} \), these quantities are well-defined since \( 2h(\sigma) > \sqrt{4h(\sigma)^2 - 1}, \sigma\in \{0, t\} \).
    Thus, we obtain
    \begin{align*}
        A(t) & = A(0) e^{ \frac{2-d}{4} \left[ \log\left( 2h(t) - \cos(\xi(t)) \sqrt{4h(t)^2 - 1} \right) - \log\left( 2h(0) - \cos(\xi(0)) \sqrt{4h(0)^2 - 1} \right) \right] } \\
             & = C \left( 2h(0) - \cos(\xi(0) - 4t) \sqrt{4h(0)^2 - 1} \right)^{\frac{2-d}{4}},
    \end{align*}
    where we defined \( C :=  A(0) \left( 2h(0) - \cos(\xi(0)) \sqrt{4h(0)^2 - 1} \right)^{\frac{d-2}{4}} \).
    We recognize here the expressions for \( L(0)^2 \) and \( L(t)^2 \).

    Let us finally turn to the expression for \( \gamma(t) \).
    We proceed to the direct integration of \( \gamma_t \).
    We have
    \begin{align*}
        \gamma(t) - \gamma(0) & = \int_0^t \dot{\gamma}(\tau)d\tau   = \int_0^t \left[ |\beta(\tau)|^2 - |X(\tau)|^2 \right] d\tau                  \\
                              & = \int_0^t \left\{ \sum_{l=1}^d 2a_l \cos(\theta_l(\tau))^2 - \sum_{l=1}^d 2a_l\sin(\theta_l(\tau))^2\right\} d\tau \\
                              & = \int_0^t 2\sum_{l=1}^d a_l \left( \cos(\theta_l(\tau))^2 - \sin(\theta_l(\tau))^2 \right) d\tau                   \\
                              & = \int_0^t \sum_{l=1}^d 2a_l \cos(2\theta_l(\tau)) d\tau                                                            \\
                              & = \sum_{l=1}^d \frac{a_l}{2} \left[ \sin(2\theta_l(t)) - \sin(2\theta_l(0)) \right],
    \end{align*}
    where the last equality has been obtained using \eqref{eqn: lemma -- HO exact integration of parameters - update of action-angle variables}.

    Finally, we calculate the evolution of the time $s(t)$ in term of the original time $t$.
    Owing to the expression of \( L(t) \) we obtained earlier,
    \begin{align*}
        s(t) := \int_0^t \frac{1}{L(\tau)^2} dt & = \int_0^t \frac{1}{\underbrace{2h(0)}_{=:c_1} - \underbrace{\sqrt{4h(0)^2-1}}_{=:c_2} \cos(\xi(0)-4\tau)} d\tau \\
                                                & = \int_0^t \frac{1}{c_1 - c_2 \cos(\xi(0) - 4\tau)} d\tau                                                        \\
                                                & = \frac{1}{4} \int^{\xi(0)}_{\xi(0)-4t} \frac{1}{c_1 - c_2 \cos(\tau)} d\tau.
    \end{align*}
    Recall the following trigonometric identity:
    \begin{equation*}
        \cos(2\tau) = \frac{1 - \tan(\tau)^2}{1 + \tan(\tau)^2}, \quad \tau \in \mathbb{R}
        ,
    \end{equation*}
    hence
    \begin{align*}
         & \int_0^t \frac{1}{c_1 - c_2 \cos(\xi(0) - 4\tau)} d\tau                                                                                                                                                                                                                      \\
         & = \frac{1}{4} \int^{\xi(0)}_{\xi(0)-4t} \frac{1}{c_1 - c_2 \frac{1 - \tan(\tau/2)^2}{1 + \tan(\tau/2)^2} } d\tau                                                                                                                                                             \\
         & = \frac{1}{4} \int^{\xi(0)}_{\xi(0)-4t} \frac{1 + \tan(\tau/2)^2}{c_1(1 + \tan(\tau/2)^2) - c_2(1 - \tan(\tau/2)^2) } d\tau                                                                                                                                                  \\
         & = \frac{1}{4} \int^{\xi(0)}_{\xi(0)-4t} \frac{1 + \tan(\tau/2)^2}{(c_1 + c_2)\tan(\tau/2)^2 + c_1 - c_2} d\tau                                                                                                                                                               \\
         & = \frac{1}{4(c_1 - c_2)} \int^{\xi(0)}_{\xi(0)-4t} \frac{1 + \tan(\tau/2)^2}{\frac{c_1 + c_2}{c_1-c_2}\tan(\tau/2)^2 + 1} d\tau                                                                                                                                              \\
         & = \frac{1}{2(c_1 - c_2)} \int^{\frac{\xi(0)}{2}}_{\frac{\xi(0)}{2}-2t} \frac{1 + \tan(\tau)^2}{\frac{c_1 + c_2}{c_1-c_2}\tan(\tau)^2 + 1} d\tau                                                                                                                              \\
         & = \frac{1}{2(c_1 - c_2)} \int^{\frac{\xi(0)}{2}}_{\frac{\xi(0)}{2}-2t} \frac{ \frac{d}{d\tau} (\tan(\tau)) }{\frac{c_1 + c_2}{c_1-c_2}\tan(\tau)^2 + 1} d\tau                                                                                                                \\
         & = \frac{1}{2(c_1 - c_2)} \frac{1}{\sqrt{\frac{c_1 + c_2}{c_1-c_2}}} \int^{\frac{\xi(0)}{2}}_{\frac{\xi(0)}{2}-2t} \frac{ \frac{d}{d\tau} \left( \sqrt{\frac{c_1 + c_2}{c_1-c_2}} \tan(\tau) \right) }{\left[\sqrt{\frac{c_1 + c_2}{c_1-c_2}} \tan(\tau)\right]^2 + 1} d\tau.
    \end{align*}
    Moreover, \( (c_1 - c_2)(c_1 + c_2) = c_1^2 - c_2^2 = (2h)^2 - (4h^2 - 1) = 1 \) and \( c_1 - c_2 > 0 \), thus \( \sqrt{\frac{c_1+c_2}{c_1-c_2} } = (c_1+c_2) \) and
    \begin{equation*}
        \int_0^t \frac{1}{L(\tau)^2} d\tau
        = \frac{1}{2} \int_{\frac{\xi(0)}{2} - 2t}^{\frac{\xi(0)}{2}} \frac{\frac{d}{d\tau} \left( (c_1+c_2) \tan(\tau) \right)}{\left( (c_1+c_2) \tan(\tau) \right)^2 + 1} d\tau.
    \end{equation*}
    Now let \( m_0 \in \mathbb{Z} \) such that \( \frac{\xi(0)}{2} \in m_0\pi + \left( -\frac{\pi}{2}, \frac{\pi}{2}  \right] \), and \( m_t \in \mathbb{Z} \) such that \( \frac{\xi(t)}{2} \in (m_0-m_t)\pi + \left( -\frac{\pi}{2}, \frac{\pi}{2}  \right] \).
    We recall that \( \xi(t) = \xi(0) - 4t \).
    Then
    \begin{align*}
         & \int_0^t \frac{1}{L(\tau)^2} d\tau = \frac{1}{2} \int_{\frac{\xi(0)}{2}-2t}^{\frac{\xi(0)}{2}} \underbrace{\frac{\frac{d}{d\tau} \left( (c_1+c_2) \tan(\tau) \right)}{\left( (c_1+c_2) \tan(\tau) \right)^2 + 1}}_{=:f(\tau)} d\tau \\
         & = \frac{1}{2} \int_{m_0\pi - \frac{\pi}{2} }^{\frac{\xi(0)}{2}} f(\tau) d\tau
        + \frac{1}{2} \int_{(m_0-1)\pi - \frac{\pi}{2} }^{m_0\pi - \frac{\pi}{2}} f(\tau) d\tau
        + \dots + \frac{1}{2} \int_{\frac{\xi(0)}{2}-2t}^{(m_0-m_t)\pi + \frac{\pi}{2} } f(\tau) d\tau.
    \end{align*}
    For \( m\in \mathbb{Z} \), we have
    \begin{align*}
        \int_{m\pi - \frac{\pi}{2}}^{m\pi + \frac{\pi}{2}} f(\tau) d\tau
         & = \left[ \arctan\left( (c_1+c_2) \tan(\tau) \right) \right]_{m\pi - \frac{\pi}{2}}^{m\pi + \frac{\pi}{2}} \\
         & = \left[ \arctan\left( (c_1+c_2) \tan(\tau) \right) \right]_{- \frac{\pi}{2}}^{\frac{\pi}{2}} = \pi.
    \end{align*}
    Now write \( \widetilde{\frac{\xi(0)}{2}} := \frac{\xi(0)}{2} - m_0\pi \in \left( -\frac{\pi}{2}, \frac{\pi}{2} \right] \), and \( \widetilde{\frac{\xi(\tau)}{2}} := \frac{\xi(\tau)}{2} - (m_0 - m_t)\pi \in \left( -\frac{\pi}{2}, \frac{\pi}{2} \right] \). Then,
    \begin{align*}
         & \int_0^t \frac{1}{L(\tau)^2} d\tau                                                                                  \\
         & = \frac{1}{2} (m_t-1)\pi
        + \frac{1}{2} \int_{m_0\pi - \frac{\pi}{2} }^{\frac{\xi(0)}{2}} f(\tau) d\tau
        + \frac{1}{2} \int_{\frac{\xi(0)}{2}-2t}^{(m_0-m_t)\pi + \frac{\pi}{2}} f(\tau) d\tau                                  \\
         & = (m_t-1)\frac{\pi }{2}
        + \frac{1}{2} \int_{- \frac{\pi}{2} }^{\frac{\widetilde{\xi(0)}}{2}} f(\tau) d\tau
        + \frac{1}{2} \int_{\frac{\widetilde{\xi(t)}}{2}}^{\frac{\pi}{2}} f(\tau) d\tau                                        \\
         & = (m_t-1)\frac{\pi }{2}
        + \frac{1}{2} \left[ \arctan\left( (c_1+c_2) \tan(\tau) \right) \right]_{- \frac{\pi}{2} }^{\frac{\widetilde{\xi(0)}}{2}}
        + \frac{1}{2} \left[ \arctan\left( (c_1+c_2) \tan(\tau) \right) \right]_{\frac{\widetilde{\xi(t)}}{2}}^{\frac{\pi}{2}} \\
         & = (m_t-1)\frac{\pi }{2}
        + \frac{1}{2} \arctan\left( (c_1+c_2) \tan\left( \frac{\widetilde{\xi(0)}}{2} \right) \right) + \frac{\pi }{2}         \\
         & \qquad + \frac{\pi}{2} - \arctan\left( (c_1+c_2) \tan\left( \frac{\widetilde{\xi(t)}}{2} \right) \right)            \\
         & = m_t \frac{\pi}{2}
        + \frac{1}{2} \arctan\left( (c_1+c_2) \tan\left( \frac{\widetilde{\xi(0)}}{2} \right) \right)
        - \frac{1}{2} \arctan\left( (c_1+c_2) \tan\left( \frac{\widetilde{\xi(t)}}{2} \right) \right)                          \\
         & = m_t \frac{\pi }{2}
        + \frac{1}{2} \arctan\left( (c_1+c_2) \tan\left( \frac{\xi(0)}{2} \right) \right)
        - \frac{1}{2} \arctan\left( (c_1+c_2) \tan\left( \frac{\xi(0)}{2} - 2t \right) \right)
    \end{align*}
    Hence
    \begin{align*}
        s(t) = \int_0^t \frac{1}{L(\tau)^2} d\tau
         & = - \frac{1}{2} \arctan\left( (c_1+c_2) \tan\left( \frac{\xi(0)}{2} - 2t \right) \right)                     \\
         & \quad + \frac{1}{2} \arctan\left( (c_1+c_2) \tan\left( \frac{\xi(0)}{2} \right) \right) + m_t \frac{\pi }{2}
        .
    \end{align*}

    \section{Proof of Lemma \ref{lemma: HO change of variables between params and action-angle variables}}
    \label{appendix: proof change of variables HO}

    \begin{proof}
        We have \( a_i(0) = \frac{1}{2} \left( X_i(0)^2 + \beta_i(0)^2 \right), i=1, \dots, d \). If \( a_i(0) > 0 \) we can define \( \theta_i(0) \) as \( \theta_i(0) = \arctan \left( \frac{X_i(0)}{\beta_i(0)}  \right) \).
        Otherwise, if \( a_i(0) = 0 \), then we recall that \( a(t) = a(0) \) and hence -- whatever \( \theta(0) \) -- we have \( X_i(t) = 0 \) and \( \beta_i(t) = 0 \).
        Therefore, in the case \( a_i(0) = 0 \), the exact value of \( \theta_i(0) \) does not change the behavior of \( t\mapsto (X_i(t), \beta_i(t)) \).

        For the \( (L, B) \) part,
        \begin{equation*}
            L(0)^2 - 2h(0) = -\cos(\xi(0)) \sqrt{4h(0)^2 - 1}
            ,
        \end{equation*}
        hence
        \begin{equation*}
            (L(0)^2 - 2h(0))^2 = L(0)^4 - 4L(0)^2 h(0) + 4h(0)^2 = \cos(\xi(0))^2 \left( 4 h(0)^2 - 1 \right)
            .
        \end{equation*}
        We also have
        \begin{equation*}
            \left( \frac{B(0)}{2} \right)^2 = \frac{B(0)^2}{4} = \sin(\xi(0))^2 \left( 4h(0)^2 - 1 \right)
            .
        \end{equation*}
        Then,
        \begin{equation*}
            L(0)^4 - 4L(0)^2 h(0) + 4h(0)^2 + \frac{B(0)^2}{4} = 4h(0)^2 - 1
            ,
        \end{equation*}
        that is
        \begin{equation*}
            4L(0)^2 h(0) = L(0)^4 + \frac{B(0)^2}{4} + 1
            .
        \end{equation*}
        We deduce that \( h(0), L(0) \neq 0 \), and therefore
        \begin{equation*}
            h(0) = \frac{L(0)^4 + \frac{B(0)^2}{4} + 1}{4L(0)^2}
            .
        \end{equation*}
        Note that \( h(0) \) is bounded from below by \( \frac{1}{2} \).
        Indeed,
        \begin{align*}
                 & L(0)^4 - 2L(0)^2 + 1 + \frac{B(0)^2}{4}
            = \left( L(0)^2 - 1 \right)^2 + \frac{B(0)^2}{4} \geq 0 \\
            \iff & L(0)^4 + 1 + \frac{B(0)^2}{4} \geq 2L(0)^2       \\
            \iff & h(0) \geq \frac{1}{2}.
        \end{align*}
        From this we also get that \( h(0) = \frac{1}{2} \iff L(0)^2 = 1 \) and \( B(0) = 0 \).

        If \( h(0) > \frac{1}{2}  \), we have
        \begin{equation*}
            \left\{
                \begin{aligned}
                    2h(0) - L(0)^2 & = \cos(\xi(0)) \sqrt{4h(0)^2 - 1}  \\
                    \frac{B(0)}{2} & = \sin(\xi(0)) \sqrt{4h(0)^2 - 1},
                \end{aligned}
            \right.
            \implies \frac{B(0)/2}{2h(0) - L(0)^2} = \tan(\xi(0))
            ,
        \end{equation*}
        hence
        \begin{equation*}
            \xi(0) = \arctan\left( \frac{B(0)/2}{2h(0) - L(0)^2} \right)
            .
        \end{equation*}
        Otherwise, in the case \( h(0) = \frac{1}{2} \), the value of \( \xi(0) \) is not rigourously defined.
        However, as previously, the exact value of \( \xi(0) \) is not important because \( h(t) = h(0) = \frac{1}{2} \), which means that \( L(t)^2=1 \) and \( B(t) = 0 \).
        Therefore, in the case \( h(0) = \frac{1}{2} \), the mapping \( t\mapsto (L(t), B(t)) \) does not depend on the value of \( \xi(0) \).
        Finally, since the mapping \( t\mapsto L(t) \) does not depend on \( \xi(0) \) in the case \( h(0) = \frac{1}{2} \), we also have that \( t\mapsto A(t)\) does not depend on the exact value of \( \xi(0) \), thanks to the expression of \( A(t) = A(0) \left( L(t) / L(0)  \right)^{\frac{2-d}{2}}  \).

        Finally, it remains to show that if \( a_i(0) = 0,\, i\in \{1, \dots, d\} \) or \( h(0) = \frac{1}{2} \), then the behavior of the mappings \( t\mapsto \gamma(t) \) and \(t \mapsto s(t) \) do not depend on the exact value of \( \theta_i(0),\, i\in\{1, \dots, d\} \) or \( \xi(0) \).
        The exact formulae for \( \gamma(t) \) and \( s(t) \) are:
        \begin{align*}
            \gamma(t) & =
            \gamma(0) + \sum_{l=1}^d \frac{a_l(0)}{2} \left[ \sin(2\theta_l(t)) - \sin(2\theta_l(0)) \right]                                                        \\
            s(t)      & = \frac{1}{2} \arctan\left( \left( 2h(0) + \sqrt{4h(0)^2 - 1} \right) \tan\left( \frac{\xi(0)}{2} - 2t \right) \right)                      \\
                      & \qquad - \frac{1}{2} \arctan\left( \left( 2h(0) + \sqrt{4h(0)^2 - 1} \right) \tan\left( \frac{\xi(0)}{2} \right) \right) -m_t \frac{\pi}{2}
        \end{align*}
        It is clear that if \( a_i(0) = 0 \) then \( \gamma(t) \) does not depend on \( \theta_i(0) \) nor \( \theta_i(t) \), \( i\in \{1, \dots, d\} \).
        If \( h(0) = \frac{1}{2} \), then
        \begin{equation*}
            2h(0) + \sqrt{4h(0)^2 - 1} = 1
            ,
        \end{equation*}
        so that
        \begin{align*}
             & \frac{1}{2} \arctan\left( \left( 2h(0) + \sqrt{4h(0)^2 - 1} \right) \tan\left( \frac{\xi(0)}{2} - 2t \right) \right)                                             \\
             & \qquad - \frac{1}{2} \arctan\left( \left( 2h(0) + \sqrt{4h(0)^2 - 1} \right) \tan\left( \frac{\xi(0)}{2} \right) \right) -m_t \frac{\pi}{2}                      \\
             & = \frac{1}{2} \arctan\left( \tan\left(\frac{\xi(t)}{2}\right) \right) - \frac{1}{2} \arctan\left( \tan\left(\frac{\xi(0)}{2}\right) \right) - m_t\frac{\pi}{2} .
        \end{align*}
        Since \( \arctan: \mathbb{R}\mapsto \left( -\frac{\pi}{2}, \frac{\pi}{2} \right] \), we  write \( \widetilde{\frac{\xi(0)}{2}} := \frac{\xi(0)}{2} - m_0\pi \in \left( -\frac{\pi}{2}, \frac{\pi}{2} \right] \), and \( \widetilde{\frac{\xi(\tau)}{2}} := \frac{\xi(\tau)}{2} - (m_0 - m_t)\pi \in \left( -\frac{\pi}{2}, \frac{\pi}{2} \right] \). Then we have
        \begin{align*}
             & \frac{1}{2} \arctan\left( \tan\left(\frac{\xi(t)}{2}\right) \right) - \frac{1}{2} \arctan\left( \tan\left(\frac{\xi(0)}{2}\right) \right) - m_t\frac{\pi}{2} \\
             & = \frac{1}{2} \widetilde{\frac{\xi(t)}{2}} - \frac{1}{2} \widetilde{\frac{\xi(0)}{2}} - m_t \frac{\pi}{2}                                                    \\
             & = \frac{1}{2} \frac{\xi(t)}{2} - (m_0 - m_t)\frac{\pi}{2} - \frac{1}{2} \left(\frac{\xi(0)}{2} - m_0\pi \right) - m_t \frac{\pi}{2}                          \\
             & = \frac{1}{2} \left( \frac{\xi(t)}{2} - \frac{\xi(0)}{2}\right) = - t.
        \end{align*}
        This shows that, in the case \( h(0) = \frac{1}{2} \), the mapping \( t\mapsto \gamma(t) \) does not depend on the value chosen for the ill-defined quantity \( \xi(0) \).
    \end{proof}


    \section{Computing the coefficients of the linear system \eqref{eqn: linear system for DFMP}}
    \label{appendix: populating the linear system for DFMP}

    \subsection{Coefficients of the matrix \( \mathbf{A} \)}

    \noindent {\underline{\( \langle b_{l, 1}, b_{j, 1} \rangle \).}}

    \begin{align*}
        \langle b_{l, 1}, b_{j, 1} \rangle
         & = e^{i\gamma_l - i\gamma_j} \int_{ \mathbb{R}^d } e^{iL_l \beta_l \cdot \frac{x - X_l}{L_l} - i \frac{B_l}{4} \left |  \frac{x-X_l}{L_l} \right | ^2} e^{- \frac{1}{2} \left |  \frac{x-X_l}{L_l} \right | ^2}                                                             \\
         & \hspace{8em} \times e^{-iL_j \beta_j \cdot \frac{x - X_j}{L_j} + i \frac{B_j}{4} \left |  \frac{x-X_j}{L_j} \right | ^2} e^{- \frac{1}{2} \left |  \frac{x-X_j}{L_j} \right | ^2} dx                                                                                       \\
         & = e^{i(\gamma_l - \gamma_j)} \int_{ \mathbb{R}^d } e^{i \beta_l \cdot (x - X_l) - i\beta_j \cdot (x-X_j) } e^{-\frac{2+iB_l}{4} \left |  \frac{x-X_l}{L_l} \right | ^2} e^{- \frac{2-iB_j}{4}  \left |  \frac{x-X_j}{L_j} \right | ^2} dx                                  \\
         & = e^{i(\gamma_l - \gamma_j) - \frac{2+iB_l}{4L_l^2}  | X_l | ^2 - \frac{2-iB_j}{4L_j^2}  | X_j | ^2 - i\beta_l \cdot X_l + i\beta_j \cdot X_j }                                                                                                                            \\
         & \qquad \times \int_{ \mathbb{R}^d } e^{i(\beta_l - \beta_j) \cdot x} e^{-\frac{2+iB_l}{4L_l^2} ( | x | ^2 - 2x\cdot X_l)} e^{-\frac{2-iB_j}{4L_j^2} ( | x | ^2 - 2x\cdot X_j)} dx                                                                                          \\
         & = e^{i(\gamma_l - \gamma_j) - \frac{2+iB_l}{4L_l^2}  | X_l | ^2 - \frac{2-iB_j}{4L_j^2}  | X_j | ^2 - i\beta_l \cdot X_l + i\beta_j \cdot X_j }                                                                                                                            \\
         & \qquad \times\int_{ \mathbb{R}^d } e^{i(\beta_l - \beta_j + \frac{B_l}{2L_l^2} X_l - \frac{B_j}{2L_j^2} X_j  ) \cdot x} e^{x\cdot \left( \frac{1}{L_l^2} X_l + \frac{1}{L_j^2} X_j  \right)} e^{-\left(\frac{2+iB_l}{4L_l^2} + \frac{2-iB_j}{4L_j^2} \right)  | x | ^2} dx
    \end{align*}
    Let
    \begin{equation}
        \left |
        \begin{aligned}
            z   & := \frac{2+iB_l}{4L_l^2} + \frac{2-iB_j}{4L_j^2},                                                                                                               \\
            a   & := \frac{X_l}{L_l^2} + \frac{X_j}{L_j^2},                                                                                                                       \\
            \xi & := \frac{B_j}{2L_j^2} X_j + \beta_j - \frac{B_l}{2L_l^2} X_l - \beta_l,                                                                                         \\
            C   & = \exp\left\{i(\gamma_l - \gamma_j) - \frac{2+iB_l}{4L_l^2}  | X_l | ^2 - \frac{2-iB_j}{4L_j^2}  | X_j | ^2 - i\beta_l \cdot X_l + i\beta_j \cdot X_j \right\},
        \end{aligned}
        \right.
    \end{equation}
    and \( f(x) := e^{-z | x | ^2 + a\cdot x} \). Then
    \begin{equation*}
        \langle b_{l, 1}, b_{j, 1} \rangle
        = C \int_{ \mathbb{R}^d } e^{-i\xi\cdot x} f(x) dx = C \widehat{f}(\xi)
    \end{equation*}

    \noindent {\underline{\( \langle b_{l, n+1}, b_{j, 1} \rangle \), \( 1 \leq n \leq d \).}}

    \begin{align*}
        \langle b_{l, n+1}, b_{j, 1} \rangle
         & = C \int_{ \mathbb{R}^d } \frac{(x-X_l)_n}{L_l} e^{-i\xi\cdot x} f(x) dx  \\
         & = \frac{C}{L_l} \left( \widehat{x f}_n - (X_l)_n \widehat{f} \right)(\xi)
    \end{align*}

    \noindent {\underline{\( \langle b_{l, d+2}, b_{j, 1} \rangle \) }}

    \begin{align*}
        \langle b_{l, d+2}, b_{j, 1} \rangle
         & = C \int_{ \mathbb{R}^d } e^{-i\xi\cdot x} f(x) \frac{ | x-X_l | ^2}{L_l^2} dx                                         \\
         & = \frac{C}{L_l^2}  \int_{ \mathbb{R}^d } e^{-i\xi\cdot x} f(x) \left(  | x | ^2 - 2x\cdot X_l +  | X_l | ^2 \right) dx \\
         & = \frac{C}{L_l^2} \left( \widehat{ | x | ^2 f} - 2X_l \cdot \widehat{x f} +  | X_l | ^2 \widehat{f} \right)(\xi)
    \end{align*}

    \noindent {\underline{\( \langle b_{l, n+1}, b_{j, m+1} \rangle \), \( 1 \leq n, m \leq d \).}}

    \begin{align*}
        \langle b_{l, n+1}, b_{j, m+1} \rangle
         & = C \int_{ \mathbb{R}^d } \frac{x_n-(X_l)_n}{L_l} \frac{x_m-(X_j)_m}{L_j}  e^{-i\xi\cdot x} f(x) dx                                           \\
         & = \frac{C}{L_j L_l}  \int_{ \mathbb{R}^d } (x_n-(X_l)_n) (x_m-(X_j)_m)  e^{-i\xi\cdot x} f(x) dx                                              \\
         & = \frac{C}{L_j L_l}  \int_{ \mathbb{R}^d } \left[ x_n x_m - x_n (X_j)_m - x_m (X_l)_n + (X_l)_n (X_j)_m \right]                               \\
         & \hspace{2cm} \times e^{-i\xi\cdot x} f(x) dx                                                                                                  \\
         & = \frac{C}{L_j L_l} \left[ \widehat{x_n x_m f} - (X_l)_n \widehat{x_m f} - (X_j)_m \widehat{x_n f} + (X_l)_n (X_j)_m \widehat{f} \right](\xi)
        .
    \end{align*}

    \noindent {\underline{\( \langle b_{l, d+2}, b_{j, m+1} \rangle \), \( 1 \leq m \leq d\).}}

    \begin{align*}
         & \langle b_{l, d+2}, b_{j, m+1} \rangle                                                                                                                                       \\
         & = C \int_{ \mathbb{R}^d } e^{-i\xi\cdot x} e^{-z | x | ^2 + a\cdot x} \frac{ | x-X_l | ^2}{L_l^2} \frac{x_m-(X_j)_m}{L_j} dx                                                 \\
         & = \frac{C}{L_l^2 L_j} \int_{ \mathbb{R}^d } e^{-i\xi\cdot x} e^{-z | x | ^2 + a\cdot x} \left(  | x | ^2 - 2x\cdot X_l +  | X_l | ^2 \right) \left( x_m - (X_j)_m \right) dx \\
         & = \frac{C}{L_l^2 L_j} \left[ \widehat{x_m | x | ^2 f} - 2X_l \cdot \widehat{x_m x f} + | X_l | ^2 \widehat{x_m f} \right.                                                    \\
         & \hspace{1cm} \left. - (X_j)_m \widehat{ | x | ^2 f} + 2(X_j)_m X_l \cdot \widehat{x f} - | X_l | ^2 (X_j)_m \widehat{f} \right](\xi)
        .
    \end{align*}

    \noindent {\underline{\( \langle b_{l, d+2}, b_{j, d+2} \rangle \) }}

    \begin{align*}
        \langle b_{l, d+2}, b_{j, d+2} \rangle
         & = C \int_{ \mathbb{R}^d } e^{-i\xi\cdot x} e^{-z | x | ^2 + a\cdot x} \frac{ | x-X_l | ^2}{L_l^2} \frac{ | x-X_j | ^2}{L_j^2} dx                                   \\
         & = \frac{C}{L_l^2 L_j^2} \int_{ \mathbb{R}^d } e^{-i\xi\cdot x} e^{-z | x | ^2 + a\cdot x} \left(  | x | ^2 - 2x\cdot X_l +  | X_l | ^2 \right)                     \\
         & \hspace{2cm} \times \left(  | x | ^2 -2x\cdot X_j +  | X_j | ^2 \right) dx                                                                                         \\
         & = \frac{C}{L_l^2 L_j^2} \int_{ \mathbb{R}^d } e^{-i\xi\cdot x} e^{-z | x | ^2 + a\cdot x} \left(  | x | ^4 - 2 | x | ^2 x\cdot X_l +  | X_l | ^2  | x | ^2 \right. \\
         & \hspace{2cm} - 2(x\cdot X_j) | x | ^2 + 4(x\cdot X_l)(x\cdot X_j) - 2(x\cdot X_j)  | X_l | ^2                                                                      \\
         & \hspace{2cm} \left. +  | x | ^2  | X_j | ^2 - 2(x\cdot X_l)  | X_j | ^2 +  | X_l | ^2  | X_j | ^2 \right) dx                                                       \\
         & = \frac{C}{L_l^2 L_j^2} \left[ \widehat{ | x | ^4 f} - 2(X_l+X_j) \cdot \widehat{ | x | ^2 x f} +  \left(|X_l|^2 + |X_j|^2 \right) \widehat{ | x | ^2 f} \right.   \\
         & \hspace{2cm} + 4\widehat{(x\cdot X_l)(x\cdot X_j) f} - 2 \left(|X_l|^2 X_j + |X_j|^2 X_l \right) \cdot \widehat{x f}                                               \\
         & \hspace{2cm} \left. +  | X_l | ^2  | X_j | ^2 \widehat{f} \right](\xi)
    \end{align*}
    Moreover,
    \begin{align*}
        (x\cdot X_l)(x\cdot X_j)
         & = \left( \sum_{n=1}^d x_n (X_l)_n \right) \left( \sum_{m=1}^d x_m (X_j)_m \right) \\
         & = \sum_{n,m = 1}^d (X_l)_n (X_j)_m x_n x_m,
    \end{align*}
    Hence
    \begin{align*}
        \widehat{(x\cdot X_l)(x\cdot X_j) f} & = \sum_{n, m=1}^d (X_l)_n (X_j)_m \widehat{x_n x_m f}
        .
    \end{align*}

    \subsection{Coefficients of the vector of interactions \( S \)}

    \noindent {\underline{\( \langle u|u|^2, b_{j, 1} \rangle \)}}

    \begin{align*}
        \langle u|u|^2, b_{j, 1} \rangle
         & = \sum_{k,l,m} \frac{A_k A_l A_m}{L_k L_l L_m} \left\langle e^{i\Gamma_k + i\Gamma_l - i\Gamma_m} e^{-\frac{|y_k|^2 + |y_l|^2 + |y_m|^2}{2}}, e^{i\Gamma_j} e^{-\frac{1}{2} |y_j|^2} \right\rangle
    \end{align*}
    We recall the previously defined notations:
    \begin{equation*}
        \left|
        \begin{aligned}
            y_k         & = \frac{x-X_k}{L_k},                                               \\
            \Gamma_k(x) & = \gamma_k + \beta_k \cdot (x-X_k) - \frac{B_k}{4L_k^2} |x-X_k|^2.
        \end{aligned}
        \right.
    \end{equation*}
    Then,
    \begin{align*}
         & -\frac{1}{2} \left(|y_k|^2 + |y_l|^2 + |y_m|^2 + |y_j|^2\right)
        = - \frac{1}{2L_k^2} |x-X_k|^2 - \frac{1}{2L_l^2} |x-X_l|^2 - \frac{1}{2L_m^2} |x-X_m|^2                                                                                                                                    \\
         & \hspace{12em} - \frac{1}{2L_j^2} |x-X_j|^2                                                                                                                                                                               \\
         & = -\frac{1}{2}\left( \frac{1}{L_k^2} + \frac{1}{L_l^2} + \frac{1}{L_m^2} + \frac{1}{L_j^2} \right) |x|^2 + \left( \frac{1}{L_k^2} X_k + \frac{1}{L_l^2} X_l + \frac{1}{L_m^2} X_m + \frac{1}{L_j^2} X_j  \right) \cdot x \\
         & \qquad -\frac{1}{2} \left( \frac{|X_k|^2}{L_k^2} + \frac{|X_l|^2}{L_l^2} + \frac{|X_m|^2}{L_m^2} + \frac{|X_j|^2}{L_j^2} \right)
        ,
    \end{align*}
    and
    \begin{align*}
         & (\Gamma_k + \Gamma_l - \Gamma_m - \Gamma_j)                                                                                                                              \\
         & = \gamma_k + \beta_k \cdot (x-X_k) - \frac{B_k}{4L_k^2} |x-X_k|^2 + \gamma_l + \beta_l \cdot (x-X_l) - \frac{B_l}{4L_l^2} |x-X_l|^2                                      \\
         & \quad - \gamma_m - \beta_m \cdot (x-X_m) + \frac{B_m}{4L_m^2} |x-X_m|^2 - \gamma_j - \beta_j \cdot (x-X_j) + \frac{B_j}{4L_j^2} |x-X_j|^2                                \\
         & = (\gamma_k + \gamma_l - \gamma_m - \gamma_j) + (\beta_j \cdot X_j + \beta_m \cdot X_m - \beta_l\cdot X_l - \beta_k \cdot X_k)                                           \\
         & \quad - \left( \frac{B_k}{4L_k^2} |X_k|^2 + \frac{B_l}{4L_l^2} |X_l|^2 - \frac{B_m}{4L_m^2} |X_m|^2 - \frac{B_j}{4L_j^2} |X_j|^2  \right)                                \\
         & \quad + x \cdot \left( \beta_k + \beta_l - \beta_m - \beta_j + \frac{B_k}{2L_k^2} X_k + \frac{B_l}{2L_l^2} X_l - \frac{B_m}{2L_m^2} X_m - \frac{B_j}{2L_j^2} X_j \right) \\
         & \quad - \left( \frac{B_k}{4L_k^2} + \frac{B_l}{4L_l^2} - \frac{B_m}{4L_m^2} - \frac{B_j}{4L_j^2} \right) |x|^2
    \end{align*}

    Define
    \begin{equation*}
        \left|
        \begin{aligned}
            C_\Im & := \exp\left\{ i\left( \gamma_k + \gamma_l - \gamma_m - \gamma_j \right) \right\}                                                                                                                        \\
                  & \hspace{1em} \times \exp\left\{ i\left(\beta_j \cdot X_j + \beta_m \cdot X_m - \beta_l\cdot X_l - \beta_k \cdot X_k \right) \right\}                                                                     \\
                  & \hspace{1em} \times \exp\left\{- i\left( \frac{B_k}{4L_k^2} |X_k|^2 + \frac{B_l}{4L_l^2} |X_l|^2 - \frac{B_m}{4L_m^2} |X_m|^2 - \frac{B_j}{4L_j^2} |X_j|^2  \right) \right\}                             \\
            C_\Re & := \exp\left\{-\frac{1}{2} \left( \frac{|X_k|^2}{L_k^2} + \frac{|X_l|^2}{L_l^2} + \frac{|X_m|^2}{L_m^2} + \frac{|X_j|^2}{L_j^2} \right) \right\}                                                         \\
            C     & := \frac{A_k A_l A_m}{L_k L_l L_m} C_\Im C_\Re                                                                                                                                                           \\
            \xi   & := -\left[ \beta_k + \beta_l - \beta_m - \beta_j + \frac{B_k}{2L_k^2} X_k + \frac{B_l}{2L_l^2} X_l - \frac{B_m}{2L_m^2} X_m - \frac{B_j}{2L_j^2} X_j \right]                                             \\
            z     & := \frac{1}{2} \left( \frac{1}{L_k^2} + \frac{1}{L_l^2} + \frac{1}{L_m^2} + \frac{1}{L_j^2} \right) + i \left( \frac{B_k}{4L_k^2} + \frac{B_l}{4L_l^2} - \frac{B_m}{4L_m^2} - \frac{B_j}{4L_j^2} \right) \\
            a     & := \frac{1}{L_k^2}
            X_k + \frac{1}{L_l^2} X_l + \frac{1}{L_m^2} X_m + \frac{1}{L_j^2} X_j
        \end{aligned}
        \right.
    \end{equation*}
    and \( f(x) := e^{-z | x | ^2 + a\cdot x} \). Then
    \begin{equation}
        \langle u|u|^2, b_{j, 1} \rangle = \sum_{k,l,m}
        C \widehat{f}(\xi) .
    \end{equation}

    \noindent {\underline{\( \langle u|u|^2, b_{j, r+1} \rangle \), \( r = 1, \dots, d \)}}

    \begin{align*}
         & \langle u|u|^2, b_{j, r+1} \rangle                                                                                                                                                                                         \\
         & = \sum_{k,l,m} \frac{A_k A_l A_m}{L_k L_l L_m} \left\langle e^{i\Gamma_k + i\Gamma_l - i\Gamma_m} e^{-\frac{|y_k|^2 + |y_l|^2 + |y_m|^2}{2}}, e^{i\Gamma_j} e^{-\frac{1}{2} |y_j|^2} \frac{x_r-(X_j)_r}{L_j} \right\rangle \\
         & = \sum_{k,l,m} \frac{C}{L_j} \left( \widehat{x_r f} - (X_j)_r \widehat{f} \right).
    \end{align*}

    \noindent {\underline{\( \langle u|u|^2, b_{j, d+2} \rangle \)}}

    \begin{align*}
         & \langle u|u|^2, b_{j, d+2} \rangle                                                                                                                                                                                                  \\
         & = \sum_{k,l,m} \frac{A_k A_l A_m}{L_k L_l L_m} \left\langle e^{i\Gamma_k + i\Gamma_l - i\Gamma_m} e^{-\frac{|y_k|^2 + |y_l|^2 + |y_m|^2}{2}}, e^{i\Gamma_j} e^{-\frac{1}{2} |y_j|^2} \left|\frac{x-X_j}{L_j}\right|^2 \right\rangle \\
         & = \sum_{k,l,m} \frac{C}{L_j^2} \left( \widehat{|x|^2 f} - 2X_j \cdot \widehat{xf} + |X_j|^2 \widehat{f} \right).
    \end{align*}
    %


    \section{Fourier transforms of Gaussians}
    \label{appendix: fourier transforms of Gaussians}

    \begin{lemma}[Fourier transform of complex Gaussians]
        Let \( z\in \mathbb{C},\ \Re(z) \geq 0 \).
        Then,
        \begin{equation}
            \mathcal{F}\left( e^{-z|\cdot|^2} \right)(\xi) = \left( \frac{\pi}{z} \right)^{\frac{d}{2}} e^{-\frac{| \xi |^2}{4z} }, \quad \xi \in \mathbb{R}^d
            .
        \end{equation}
        More generally, let \( z = z_1 + iz_2 \in \mathbb{C} \), \( z_1, z_2 \in \mathbb{R}, \ z_1 > 0 \), \( a\in \mathbb{R}^d \) and
        \begin{equation}
            f:\ x \in \mathbb{R}^d \mapsto \exp\left( -z|x|^2 + a\cdot x \right) \in \mathbb{C}
            ,
        \end{equation}
        then we have the Fourier transforms given by Table \ref{table: useful Fourier transforms}.
    \end{lemma}

    \begin{table}
        \begin{center}
            \begin{tabular}{cc}
                \hline
                \( h(x) \)          & \( \widehat{h}(\xi) / e^{- \frac{(\xi + ia)\cdot (\xi + ia)}{4z}} \)                                                                                          \\[0.5em]
                \hline
                \( f \)             & \( \left( \frac{\pi}{z}  \right)^{\frac{d}{2}} \)                                                                                                             \\[1em]
                \( xf \)            & \( -i \left( \frac{\pi}{z} \right)^{\frac{d}{2}} \frac{\xi + ia}{2z} \)                                                                                       \\[1em]
                \( x_m x_n f \)     & \( - \frac{1}{4z^2} \left(\frac{\pi}{z} \right)^{\frac{d}{2}}  \left( \xi_n + ia_n \right) \left( \xi_m + ia_m \right) \)                                     \\[1em]
                \( x_m^2 f \)       & \( \frac{1}{2z} \left(\frac{\pi}{z} \right)^{\frac{d}{2}}  \left[ 1 - \frac{(\xi_m + ia_m)^2}{2z} \right] \)                                                  \\[1em]
                \( | x |^2 f \)     & \( \frac{1}{2z} \left(\frac{\pi}{z} \right)^{\frac{d}{2}} \left[ d - \frac{|\xi|^2 + 2ia\cdot \xi - |a|^2}{2z} \right] \)                                     \\[1em]
                \( x_m | x |^2 f \) & \( - \frac{i}{4z^2} \left(\frac{\pi}{z} \right)^{\frac{d}{2}} \left( \xi_m + ia_m \right) \left[ d+2 - \frac{|\xi|^2 + 2ia\cdot \xi - |a|^2}{2z} \right] \)   \\[1em]
                \( x_m^2 x_n^2 f \) & \( \frac{1}{4z^2} \left( \frac{\pi}{z} \right)^{\frac{d}{2}} \left( 1 - \frac{(\xi_n + ia_n)^2}{2z} \right) \left( 1 - \frac{(\xi_m + ia_m)^2}{2z} \right) \) \\[1em]
                \( x_m^4 f \)       & \( \frac{1}{4z^2} \left(\frac{\pi}{z} \right)^{\frac{d}{2}} \left[ 3 - 6 \frac{(\xi_m + ia_m)^2}{2z} + \frac{(\xi_m + ia_m)^4}{4z^2} \right] \)               \\[0.5em]
                \hline
            \end{tabular}
            \caption{\centering Fourier Transform of some polynomials in \( x = (x_1, \dots, x_d) \in \mathbb{R}^d \) multiplied by \( f(x) = e^{-z|x|^2 +a\cdot x}, \, z\in \mathbb{C}, \Re(z) > 0, a \in \mathbb{R}^d \).}
            \label{table: useful Fourier transforms}
        \end{center}
    \end{table}

    \begin{proof}

        For the sake of clarity, for \( \xi, a \in \mathbb{R}^d \) and \( z\in \mathbb{C} \), let
        \begin{equation*}
            E(\xi, a, z) := \exp\left\{- \frac{|\xi|^2 + 2ia\cdot \xi - |a|^2}{4z} \right\} = \exp\left\{- \frac{(\xi + ia)\cdot (\xi+ia)}{4z} \right\}
            .
        \end{equation*}

        \noindent \underline{\( \widehat{f} \).}

        We have
        \begin{equation*}
            -z|x|^2 + a\cdot x = -z\left| x - \frac{a}{2z_1} \right|^2 - i \frac{z_2 a}{z_1} \cdot x + \frac{z|a|^2}{4z_1^2}
            .
        \end{equation*}
        Recall the following usual properties on Fourier transform:
        \begin{equation*}
            \widehat{f(x-a)} = \widehat{f}(\xi)e^{-ia\cdot \xi},\quad \widehat{fe^{-ia\cdot x}} = \widehat{f}(\xi+a)
            .
        \end{equation*}
        Let
        \begin{equation*}
            g(x) = e^{-z\left| x - \frac{a}{2z_1} \right|^2}
            ,
        \end{equation*}
        then
        \begin{equation*}
            \hat{g}(\xi) = \left( \frac{\pi}{z} \right)^{\frac{d}{2} } e^{-\frac{|\xi|^2}{4z} - \frac{ia\cdot \xi}{2z_1} }
        \end{equation*}
        and
        \begin{equation*}
            f(x) = g(x) e^{-\frac{iz_2}{z_1} a\cdot x + \frac{z|a|^2}{4z_1^2} }
            .
        \end{equation*}
        Hence,
        \begin{align*}
            \widehat{f}(\xi)
             & = e^{\frac{z|a|^2}{4z_1^2}} \hat{g}\left( \xi + \frac{z_2}{z_1} a \right) = \left( \frac{\pi}{z} \right)^{\frac{d}{2} } e^{\frac{z|a|^2}{4z_1^2}} e^{-\frac{1}{4z} \left| \xi + \frac{z_2}{z_1} a \right|^2 - \frac{ia}{2z_1} \cdot \left( \xi + \frac{z_2}{z_1} a \right)} \\
             & = \left( \frac{\pi}{z}  \right)^{\frac{d}{2} } e^{\frac{z|a|^2}{4z_1^2} - \frac{1}{4z} \left( |\xi|^2 + 2\frac{z_2}{z_1} a\cdot \xi + \frac{z_2^2}{z_1^2} |a|^2 \right) - \frac{ia\cdot \xi}{2z_1} - \frac{i|a|^2 z_2}{2z_1^2} }                                            \\
             & = \left( \frac{\pi}{z}  \right)^{\frac{d}{2} } e^{-\frac{|\xi|^2}{4z} + (a\cdot \xi)\left( -\frac{z_2}{2zz_1} - \frac{i}{2z_1}  \right) + |a|^2 \left( \frac{z}{4z_1^2} - \frac{z_2^2}{4zz_1^2} - \frac{iz_2}{2z_1^2}  \right)}                                             \\
             & = \left( \frac{\pi}{z}  \right)^{\frac{d}{2}} e^{-\frac{|\xi|^2}{4z} - \frac{a\cdot \xi}{2zz_1} \left[ z_2 + i(z_1 + iz_2) \right] + \frac{|a|^2}{4zz_1^2} \left[ (z_1+iz_2)^2 - z_2^2 - 2iz_2(z_1+iz_2) \right]}                                                           \\
             & = \left( \frac{\pi}{z}  \right)^{\frac{d}{2}} e^{-\frac{|\xi|^2}{4z} - i\frac{a\cdot \xi}{2z} + \frac{|a|^2}{4z}}                                                                                                                                                           \\
             & = \left( \frac{\pi}{z}  \right)^{\frac{d}{2}}
            E(\xi, a, z).
        \end{align*}

        \vspace{1em}

        \noindent \underline{\( \widehat{x f} \).}

        \begin{align*}
            \widehat{xf}(\xi) & = i\nabla_\xi \hat{f} = i \nabla_\xi \left[ \left( \frac{\pi}{z} \right)^{\frac{d}{2}}  E(\xi, a, z) \right] = i\left( \frac{\pi}{z} \right)^{\frac{d}{2}} E(\xi, a, z) \left[ -\frac{\xi}{2z} -\frac{ia}{2z}  \right] \\
                              & = -i \left( \frac{\pi}{z} \right)^{\frac{d}{2}} \frac{\xi + ia}{2z}
            E(\xi, a, z).
        \end{align*}
        \vspace{1em}

        \noindent \underline{\( \widehat{x_m^2 f} \), \( m=1, \dots, d \).}

        \begin{align*}
            \widehat{x_m^2 f}(\xi) & = i\partial_{\xi_m} \left(\widehat{xf}\right)_m = i \partial_{\xi_m} \left[ -i \left( \frac{\pi}{z} \right)^{\frac{d}{2}} \frac{(\xi + ia)_m}{2z} E(\xi, a, z)   \right]           \\
                                   & = \frac{1}{2z} \left(\frac{\pi}{z} \right)^{\frac{d}{2}} \left[ E(\xi, a, z) + \left( \xi_m + ia_m \right) E(\xi, a, z) \left( -\frac{\xi_m}{2z} -\frac{ia_m}{2z}  \right) \right] \\
                                   & = \frac{1}{2z} \left(\frac{\pi}{z} \right)^{\frac{d}{2}}  \left[ 1 - \frac{(\xi_m + ia_m)^2}{2z} \right] E(\xi, a, z).
        \end{align*}
        \vspace{1em}

        \noindent \underline{\( \widehat{x_m x_n f} \), \( m, n=1, \dots, d, \, n\neq m \).}

        \begin{align*}
            \widehat{x_m x_n f}(\xi) & = i\partial_{\xi_m} \left(\widehat{xf}\right)_n = i \partial_{\xi_m} \left[ -i \left( \frac{\pi}{z} \right)^{\frac{d}{2}} \frac{\xi_n + ia_n}{2z} E(\xi, a, z)   \right] \\
                                     & = \frac{1}{2z} \left(\frac{\pi}{z} \right)^{\frac{d}{2}} \left( \xi_n + ia_n \right) \left[ - \frac{\xi_m + ia_m}{2z}  \right] E(\xi, a, z)                              \\
                                     & = - \frac{1}{4z^2} \left(\frac{\pi}{z} \right)^{\frac{d}{2}}  \left( \xi_n + ia_n \right) \left( \xi_m + ia_m \right) E(\xi, a, z).
        \end{align*}
        \vspace{1em}

        \noindent \underline{\( \widehat{|x|^2 f} \).}

        \begin{align*}
            \widehat{|x|^2 f}(\xi) & = \widehat{x_1^2 f}(\xi) + \dots + \widehat{x_d^2 f}(\xi)                                                                                       \\
                                   & = \frac{1}{2z} \left(\frac{\pi}{z} \right)^{\frac{d}{2}} \left[ d - \frac{(\xi_1 + ia_1)^2 + \dots + (\xi_d + ia_d)^2}{2z} \right] E(\xi, a, z) \\
                                   & = \frac{1}{2z} \left(\frac{\pi}{z} \right)^{\frac{d}{2}} \left[ d - \frac{|\xi|^2 + 2ia\cdot \xi - |a|^2}{2z} \right] E(\xi, a, z).
        \end{align*}
        \vspace{1em}

        \noindent \underline{\( \widehat{x_m|x|^2 f} \), \( m=1, \dots, d \).}

        \begin{align*}
             & \widehat{x_m|x|^2 f}(\xi)                                                                                                                                                                                                      \\
             & = i \partial_{\xi_m}\left[ \widehat{|x|^2 f}(\xi) \right] = i \partial_{\xi_m} \left[ \frac{1}{2z} \left(\frac{\pi}{z} \right)^{\frac{d}{2}} \left( d - \frac{|\xi|^2 + 2ia\cdot \xi - |a|^2}{2z} \right) E(\xi, a, z) \right] \\
             & = \frac{i}{2z} \left(\frac{\pi}{z} \right)^{\frac{d}{2}} \left[ -2\frac{\xi_m + ia_m}{2z} + \left( d - \frac{|\xi|^2 + 2ia\cdot \xi - |a|^2}{2z} \right) \left( - \frac{\xi_m + ia_m}{2z}  \right)  \right]  E(\xi, a, z)      \\
             & = - \frac{i}{4z^2} \left(\frac{\pi}{z} \right)^{\frac{d}{2}} \left( \xi_m + ia_m \right) \left[ d+2 - \frac{|\xi|^2 + 2ia\cdot \xi - |a|^2}{2z} \right]  E(\xi, a, z)
            .
        \end{align*}
        \vspace{1em}

        \noindent \underline{\( \widehat{x_m^3 f} \), \( m=1, \dots, d\).}

        \begin{align*}
            \widehat{x_m^3 f}(\xi)
             & = i \partial_{\xi_m} \left[ \widehat{x_m^2 f}(\xi) \right] = i \partial_{\xi_m} \left[ \frac{1}{2z} \left(\frac{\pi}{z} \right)^{\frac{d}{2}} \left( 1 - \frac{(\xi_m + ia_m)^2}{2z} \right) E(\xi, a, z) \right] \\
             & = \frac{i}{2z} \left(\frac{\pi}{z} \right)^{\frac{d}{2}} \left[ -2 \frac{\xi_m + ia_m}{2z} + \left( - \frac{\xi_m + ia_m}{2z}  \right) \left( 1 - \frac{(\xi_m + i a_m)^2}{2z}  \right) \right] E(\xi, a, z)      \\
             & = -\frac{i}{4z^2} \left(\frac{\pi}{z} \right)^{\frac{d}{2}} \left(\xi_m + ia_m\right) \left[ 3 - \frac{(\xi_m + ia_m)^2}{2z} \right] E(\xi, a, z)                                                                 \\
             & = -\frac{i}{4z^2} \left(\frac{\pi}{z} \right)^{\frac{d}{2}} \left[ 3(\xi_m+ia_m) - \frac{(\xi_m + ia_m)^3}{2z} \right] E(\xi, a, z).
        \end{align*}
        \vspace{1em}

        \noindent \underline{\( \widehat{x_m x_n^2 f} \), \( m, n=1, \dots, d,\, n\neq m \).}

        \begin{align*}
            \widehat{x_m x_n^2 f}(\xi) & = i\partial_{\xi_m} \left(\widehat{x_n^2 f}\right) = i \partial_{\xi_m} \left[ \frac{1}{2z} \left(\frac{\pi}{z} \right)^{\frac{d}{2}}  \left( 1 - \frac{(\xi_n + ia_n)^2}{2z} \right) E(\xi, a, z) \right] \\
                                       & = - \frac{i}{2z} \left(\frac{\pi}{z} \right)^{\frac{d}{2}}  \left( 1 - \frac{(\xi_n + ia_n)^2}{2z} \right) \frac{\xi_m + ia_m}{2z}
            E(\xi, a, z).
        \end{align*}
        \vspace{1em}

        \noindent \underline{\( \widehat{x_m^4 f} \), \( m=1, \dots, d \).}

        \begin{align*}
             & \widehat{x_m^4 f}(\xi)                                                                                                                                                                                                           \\
             & = i \partial_{\xi_m} \left[ \widehat{x_m^3 f}(\xi) \right] = i \partial_{\xi_m} \left[ -\frac{i}{4z^2} \left(\frac{\pi}{z} \right)^{\frac{d}{2}} \left( 3(\xi_m+ia_m) - \frac{(\xi_m + ia_m)^3}{2z} \right) E(\xi, a, z) \right] \\
             & = \frac{1}{4z^2} \left(\frac{\pi}{z} \right)^{\frac{d}{2}} \left[ 3 - 3\frac{(\xi+ia_m)^2}{2z} + \left( 3(\xi_m+ia_m) - \frac{(\xi_m + ia_m)^3}{2z} \right) \left( - \frac{\xi_m + ia_m}{2z}  \right) \right] E(\xi, a, z)       \\
             & = \frac{1}{4z^2} \left(\frac{\pi}{z} \right)^{\frac{d}{2}} \left[ 3 - 6 \frac{(\xi_m + ia_m)^2}{2z} + \frac{(\xi_m + ia_m)^4}{4z^2} \right] E(\xi, a, z).
        \end{align*}

        \vspace{1em}

        \noindent \underline{\( \widehat{x_m^2 x_n^2 f} \), \( m=1, \dots, d,\, n\neq m \).}

        \begin{align*}
            \widehat{x_m^2 x_n^2 f}(\xi) & = i\partial_{\xi_m} \left(\widehat{x_m x_n^2 f}\right)_n = i \partial_{\xi_m} \left[ - \frac{i}{2z} \left(\frac{\pi}{z} \right)^{\frac{d}{2}} \left( 1 - \frac{(\xi_n + ia_n)^2}{2z} \right) \frac{\xi_m + ia_m}{2z} E(\xi, a, z) \right] \\
                                         & = \frac{1}{4z^2} \left( \frac{\pi}{z} \right)^{\frac{d}{2}} \left( 1 - \frac{(\xi_n + ia_n)^2}{2z} \right) \partial_{\xi_m} \left[ (\xi_m + ia_m) E(\xi, a, z)  \right]                                                                   \\
                                         & = \frac{1}{4z^2} \left( \frac{\pi}{z} \right)^{\frac{d}{2}} \left( 1 - \frac{(\xi_n + ia_n)^2}{2z} \right) \left( 1 - \frac{(\xi_m + ia_m)^2}{2z} \right) E(\xi, a, z).
        \end{align*}

    \end{proof}


    \section{Miscellaneous computations}
    \label{appendix: Miscellaneous computations}

    We provide in this section some miscellaneous computations, which hold in dimension \( d=2 \) as long as \( v_j(s_j, y_j) = e^{- \frac{|y_j|^2}{2} },\ j=1, \dots, N \).

    \subsection{Conservative quantities in dimension \( d=2 \)}

    We give the explicit expressions for the conservative quantities involved in Lemma \ref{lemma: conserved quantities in HO}, in the two-dimensional case.

    The \( \mathbb{L}^2 \) norm of a sum of \( N \) bubbles is given by
    \begin{align*}
        \|u\|_{ \mathbb{L}^2 }^2 & = \sum_{k,l=1}^N \frac{A_k A_l}{L_k L_l} \langle b_{k, 1}, b_{l, 1} \rangle.
    \end{align*}

    The energy of a sum of bubbles is given by
    \begin{equation*}
        E_{\mu,\lambda} = \frac{\mu}{2} \left\langle -\Delta u + |x|^2 u, u \right\rangle + \frac{\lambda}{4} \left\langle |u|^2 u, u \right\rangle
        = E_{\mu, 0} + E_{0,\lambda} = \mu E_{1, 0} + \lambda E_{0, 1}
        .
    \end{equation*}
    We have
    \begin{equation*}
        2E_{1, 0} = \langle Hu, u \rangle = \langle -\Delta u, u \rangle + \langle |x|^2 u, u \rangle = \sum_{j,k=1}^N \langle \nabla_x u_j, \nabla_x u_k \rangle + \sum_{j,k=1}^N \langle |x|^2 u_j, u_k \rangle
        .
    \end{equation*}
    Furthermore,
    \begin{align*}
        \langle \nabla_x u_j, \nabla_x u_k \rangle
         & = \frac{A_j A_k}{L_j L_k}  \left\langle \left(i\beta_j - \frac{2+iB_j}{2L_j} y_j\right) b_{j, 1}, \left(i\beta_k - \frac{2+iB_k}{2L_k} y_k\right) b_{k, 1} \right\rangle       \\
         & = \frac{A_j A_k}{L_j L_k} \left\{ \beta_j \cdot \beta_k  \left\langle b_{j, 1}, b_{k, 1} \right\rangle +i \frac{2+iB_j}{2L_j} \beta_k \cdot
        \begin{pmatrix}
            \left\langle b_{j, 2},  b_{k, 1} \right\rangle \\
            \left\langle b_{j, 3}, b_{k, 1} \right\rangle
        \end{pmatrix}
        \right.                                                                                                                                                                           \\
         & \qquad - i  \frac{2-iB_k}{2L_k} \beta_j \cdot
        \begin{pmatrix}
            \left\langle  b_{j, 1}, b_{k, 2} \right\rangle \\
            \left\langle  b_{j, 1}, b_{k, 3} \right\rangle \\
        \end{pmatrix}
        \\
         & \qquad \left. + \frac{2+iB_j}{2L_j} \frac{2-iB_k}{2L_k}  \left(\left\langle b_{j, 2}, b_{k, 2} \right\rangle + \left\langle b_{(j, 3}, b_{k, 3} \right\rangle \right)\right\},
    \end{align*}
    and
    \begin{align*}
        \langle |x|^2 u_j, u_k \rangle
         & = \frac{A_j A_k}{L_j L_k} \left\langle \left( L_j^2 |y_j|^2 + 2L_j y_j \cdot X_j + |X_j|^2 \right) b_{j, 1}, b_{k, 1} \right\rangle \\
         & = \frac{A_j A_k}{L_j L_k} \left\{ L_j^2 \langle b_{j, 4}, b_{k, 1} \rangle + 2L_j X_j \cdot
            \begin{pmatrix}
                \langle b_{j, 2}, b_{k, 1} \rangle \\
                \langle b_{j, 3}, b_{k, 1} \rangle
            \end{pmatrix}
            + |X_j|^2 \langle b_{j, 1}, b_{k, 1} \rangle \right\}.
    \end{align*}

    We also have
    \begin{equation*}
        E_{0, 1} = \langle u|u|^2, u \rangle = \sum_{j=1}^N \frac{A_j}{L_j} \langle u|u|^2, b_{j, 1} \rangle .
    \end{equation*}

    We now proceed to computing the momentum, given by
    \begin{equation*}
        M_{\mu, \lambda} = \left( E_{\mu,\lambda} - \mu \|xu\|^2_{ \mathbb{L}^2 } \right)^2 + \mu^2 \left( \Im \int x\cdot \nabla u \bar{u} \right)^2
        .
    \end{equation*}
    We know how to compute \( E_{\mu, \lambda} \) from previously, as well as \( \|xu\|_{\mathbb{L}^2}^2 = \langle |x|^2 u, u \rangle \).
    It only remains to compute
    \begin{align*}
        \int x\cdot \nabla u \bar{u}
         & = \sum_{j,k=1}^N \frac{A_j A_k}{L_j L_k} \left\langle (L_j y_j + X_j) \cdot \left( i \beta_j - \frac{2+iB_j}{2L_j} y_j \right) b_{j, 1}, b_{k, 1} \right\rangle \\
         & = \sum_{j,k=1}^N \frac{A_j A_k}{L_j L_k} \left\{ iL_j \beta_j \cdot
            \begin{pmatrix}
                \left\langle b_{j, 2}, b_{k, 1} \right\rangle \\
                \left\langle b_{j, 3}, b_{k, 1} \right\rangle
            \end{pmatrix}
        - \frac{2+iB_j}{2} \left\langle b_{j, 4}, b_{k, 1} \right\rangle \right.                                                                                           \\
         & \qquad \left. + i \beta_j \cdot X_j \left\langle b_{j, 1}, b_{k, 1} \right\rangle - \frac{2+iB_j}{2L_j} X_j \cdot
        \begin{pmatrix}
            \left\langle b_{j, 2}, b_{k, 1} \right\rangle \\
            \left\langle b_{j, 3}, b_{k, 1} \right\rangle
        \end{pmatrix}
        \right\}
        .
    \end{align*}

    Note that all the inner products involved have already been computed when creating the DFMP matrix.

\end{appendices}

\bibliography{PhD}

\end{document}